\documentclass[a4paper,11pt]{article}

\usepackage[leqno]{amsmath}
\usepackage{amssymb,amsthm}
\usepackage{mathrsfs}
\usepackage{euscript}
\usepackage[all]{xy}
\usepackage{stmaryrd}
\usepackage{dsfont}
\usepackage{enumitem}
\usepackage{hyperref}

\theoremstyle{plain}
\newtheorem{thm}{Theorem}[section]
\newtheorem{prop}[thm]{Proposition}
\newtheorem{lem}[thm]{Lemma}
\newtheorem{cor}[thm]{Corollary}

\theoremstyle{definition}
\newtheorem{defin}[thm]{Definition}
\newtheorem{rem}[thm]{Remark}
\newtheorem{ex}[thm]{Example}

\newcommand{\aff}{{\operatorname{aff}}}
\newcommand{\A}{{\mathscr{A}}}
\newcommand{\adm}{{\operatorname{adm}}}

\newcommand{\BT}{{\mathscr{X}}}
\newcommand{\C}{{\mathcal{C}}}
\newcommand{\CC}{{\mathbb{C}}}
\newcommand{\Cscr}{{\mathscr{C}}}
\newcommand{\Cbar}{{\overline{C}}}
\newcommand{\Coeff}{{\operatorname{Coeff}}}

\newcommand{\D}{{\mathcal{D}}}
\newcommand{\Diag}{{\operatorname{Diag}}}
\newcommand{\E}{{\mathcal{E}}}
\newcommand{\End}{{\operatorname{End}}}

\newcommand{\F}{{\mathcal{F}}}
\newcommand{\fg}{{\operatorname{fg}}}

\newcommand{\id}{{\operatorname{id}}}
\newcommand{\im}{{\operatorname{im}}}
\newcommand{\ind}{{\operatorname{ind}}}

\newcommand{\GG}{{\mathbb{G}}}
\newcommand{\G}{{\mathcal{G}}}

\newcommand{\GL}{{\operatorname{GL}}}

\renewcommand{\H}{{\operatorname{H}}}
\newcommand{\Hom}{{\operatorname{Hom}}}
\newcommand{\K}{{\mathcal{K}}}

\newcommand{\Mod}{{\operatorname{Mod}}}

\renewcommand{\o}{{\mathfrak{o}}}
\newcommand{\op}{{\operatorname{op}}}
\newcommand{\ori}{{\mathit{or}}}

\newcommand{\Pbar}{{\overline{P}}}

\newcommand{\Qp}{{\mathbb{Q}_p}}
\newcommand{\res}{{\mathit{res}}}
\newcommand{\Rep}{{\operatorname{Rep}}}

\newcommand{\St}{{\operatorname{St}}}
\renewcommand{\t}{{\mathfrak{t}}}

\newcommand{\TT}{{\mathbb{T}}}

\newcommand{\Ubar}{{\overline{U}}}
\newcommand{\Y}{{\mathscr{Y}}}

\begin{document}
\setcounter{section}{0}


{\Large\bf Coefficient systems on the Bruhat-Tits building\\ and pro-$p$ Iwahori-Hecke modules}\\[6ex]
{\sc Jan Kohlhaase}\\[1ex]


{\bf Abstract.} Let $G$ be the group of rational points of a split connected reductive group over a nonarchimedean local field of residue characteristic $p$. Let $I$ be a pro-$p$ Iwahori subgroup of $G$ and let $R$ be a commutative quasi-Frobenius ring. If $H=R[I\backslash G/I]$ denotes the pro-$p$ Iwahori-Hecke algebra of $G$ over $R$ we clarify the relation between the category of $H$-modules and the category of $G$-equivariant coefficient systems on the semisimple Bruhat-Tits building of $G$. If $R$ is a field of characteristic zero this yields alternative proofs of the exactness of the Schneider-Stuhler resolution and of the Zelevinski conjecture for smooth $G$-representations generated by their $I$-invariants. In general, it gives a description of the derived category of $H$-modules in terms of smooth $G$-representations and yields a functor to generalized $(\varphi,\Gamma)$-modules extending the constructions of Colmez, Schneider and Vign\'eras.


\tableofcontents


\section*{Introduction}\label{section_0}
\addcontentsline{toc}{section}{Introduction}

Let $K$ be a nonarchimedean local field and let $\GG$ be a split connected reductive group over $K$. If $R$ is a commutative unital ring then we denote by $\Rep_R^\infty(G)$ the category of $R$-linear smooth representations of the locally profinite group $G=\GG(K)$. It lies at the heart of the local Langlands program in its various forms.\\

There are two particularly important techniques to study this category. One of them is by means of the category $\Coeff_G(\BT)$ of $G$-equivariant coefficient systems of $R$-modules on the semisimple Bruhat-Tits building $\BT$ of $G$. It is linked to the category $\Rep_R^\infty(G)$ through functors
\[
 \xymatrix{\Rep_R^\infty(G)\ar@<4pt>[rr]^{\F_{(\cdot)}}&&\Coeff_G(\BT).\ar@<4pt>[ll]^{\H_0(\BT,\;\cdot\,)}
 }
\]
Here $\F_V\in\Coeff_G(\BT)$ denotes the fixed point system of a representation $V\in\Rep_R^\infty(G)$ (cf.\ Example \ref{fixed_point_system}) and $\H_0(\BT,\F)\in\Rep_R^\infty(G)$ denotes the $0$-th homology group of the oriented chain complex $\C_c^\ori(\BT_{(\bullet)},\F)$ of a coefficient system $\F\in\Coeff_G(\BT)$ (cf.\ \S\ref{subsection_2_1}). If $R$ is the field of complex numbers the precise relation between the two categories was the subject of the seminal article \cite{SS} of Schneider and Stuhler. As an outcome one obtains functorial finite projective resolutions of complex smooth $G$-representations, a proof of the Zelevinski conjecture and a description of large parts of  $\Rep_R^\infty(G)$ as a localization of $\Coeff_G(\BT)$.\\

If $I$ is a compact open subgroup of $G$ then we denote by $H=R[I\backslash G/I]$ the corresponding Hecke algebra over $R$ and by $\Mod_H$ the category of left $H$-modules. The $R$-algebra $H$ can also be realized as the opposite endomorphism ring of the compactly induced smooth $G$-representation $X=\ind_I^G(R)$ whence $X$ is naturally a right $H$-module. By Frobenius reciprocity there is a pair of adjoint functors
\[
 \xymatrix{\Rep_R^\infty(G)\ar@<4pt>[rr]^{(\cdot)^I}&&\Mod_H.\ar@<4pt>[ll]^{X\otimes_H(\cdot)}
 }
\]
If $R$ is the field of complex numbers and if $I$ is an Iwahori subgroup then the precise relation between these categories was clarified by Bernstein and A.\ Borel. They showed that if $\Rep_R^I(G)\subseteq\Rep_R^\infty(G)$ denotes the full subcategory of representations generated by their $I$-invariants then the above functors give mutually quasi-inverse equivalences $\Rep_R^I(G)\cong\Mod_H$ of abelian categories  (cf.\ \cite{Ber}, Corollaire 3.9). This was used crucially by Kazhdan and Lusztig to establish the local Langlands correspondence for this kind of representations.\\

Let $p$ denote the characteristic of the residue class field of $K$. The emerging $p$-adic and mod-$p$ variants of the local Langlands program make it necessary to consider the case where $p$ is nilpotent in $R$ and $I$ is a pro-$p$ Iwahori subgroup of $G$ (cf.\ \S\ref{subsection_1_1}). If $R=\overline{\mathbb{F}}_p$ and if $G=\GL_2(\Qp)$ or $G=\mathrm{SL}_2(\Qp)$ then Ollivier and Koziol showed that $\Rep_R^I(G)\cong\Mod_H$ via the above functors (cf.\ \cite{Koz}, Corollary 5.3 and \cite{Oll1}, Th\'eor\`eme 1.2 (a)). However, this is not true in general and the precise relation between the categories $\Rep_R^I(G)$ and $\Mod_H$ is unknown. Likewise, if $R=\overline{\mathbb{F}}_p$ and if $G=\GL_2(K)$ then Paskunas, Breuil and Hu used coefficient systems to construct interesting examples of $G$-representations in \cite{BP}, \cite{Hu} and \cite{Pas}. However, the relation between the categories $\Rep_R^\infty(G)$ and $\Coeff_G(\BT)$ in the modular setting has never been studied systematically.\\

We continue to assume that $I$ is a pro-$p$ Iwahori subgroup of $G$. The aim of the present article is to clarify the relation between the categories $\Mod_H$ and $\Coeff_G(\BT)$ and to give applications to the theory of smooth $R$-linear $G$-representations. For the general setup $R$ is allowed to be any commutative unital ring. For most of the deeper results, however, we will assume in addition that $R$ is a quasi-Frobenius ring, i.e.\ that $R$ is noetherian and selfinjective. The most important case for arithmetic applications is $R=S/tS$ where $S$ is a principal ideal domain and $t\in S$ is non-zero (cf.\ \cite{Lam}, Example 3.12). Of course, $R$ could still be any field.\\

At the beginning of the article we gather the necessary input from the theory of Bruhat-Tits buildings, pro-$p$ Iwahori-Hecke algebras and coefficient systems. We also generalize Paskunas' notion of a diagram from $\GL_2(K)$ to any $G$ (cf.\ Remark \ref{comparison_Paskunas} and Proposition \ref{diagrams}).\\

Depending on $R$ the augmented chain complex $0\to\C_c^\ori(\BT_{(\bullet)},\F_V)\to V\to 0$ of the fixed point system $\F_V$ of a representation $V\in\Rep_R^\infty(G)$ may or may not be exact (cf.\ \cite{OS1}, Remark 3.2). However, Ollivier and Schneider observed that the complex $\C_c^\ori(\BT_{(\bullet)},\F_V)^I$ of $I$-invariants is always acyclic (cf.\ \cite{OS1}, Theorem 3.4, which was inspired by the work \cite{Bro} of Broussous). It therefore seems natural to consider the functor
\[
 M:\Coeff_G(\BT)\longrightarrow\Mod_H,\quad M(\F)=\H_0(\C_c^\ori(\BT_{(\bullet)},\F)^I).
\]
In fact, the above acyclicity result holds for a larger class of coefficient systems that we study in \S\ref{subsection_2_2}. Note that the order of the functors $\H_0$ and $(\cdot)^I$ is a subtle point here. If $p$ is nilpotent in $R$ then it can generally not be reversed.\\

In order to construct a functor in the other direction let $F$ be an arbitrary face of $\BT$. We denote by $P_F^\dagger=\{g\in G\;|\;gF=F\}$ the stabilizer of $F$ in $G$, by $P_F$ the parahoric subgroup of $G$ corresponding to $F$ and by $I_F$ the pro-$p$ radi\-cal of $P_F$ (cf.\ \S\ref{subsection_1_1}). Given an $H$-module $M$ we consider the smooth $R$-linear $P_F^\dagger$-representation $\t_F(M)=\im(X^{I_F}\otimes_HM\to\Hom_H(\Hom_H(X^{I_F},H),M))$. This is a local version of a construction appearing in \cite{OS2}. Letting $F$ vary we obtain the functor
\[
\F(\cdot):\Mod_H\longrightarrow\Coeff_G(\BT),\quad\F(M)=(\t_F(M))_F.
\]
In \S\ref{subsection_3_2} we single out a full subcategory $\C$ of $\Coeff_G(\BT)$ such that the functors $M(\cdot)$ and $\F(\cdot)$ are mutually quasi-inverse equivalences $\Mod_H\cong\C$ of additive categories (cf.\ Theorem \ref{equivalence} and Theorem \ref{quasi_inverse_coefficient_system}). In particular, the functor $\F(\cdot):\Mod_H\to\Coeff_G(\BT)$ is fully faithful. For any $H$-module $M$ we obtain a functorial resolution
\begin{equation}\label{generalized_GP_resolution}
 0\longrightarrow\C_c^\ori(\BT_{(\bullet)},\F(M))^I\longrightarrow M\longrightarrow 0
\end{equation}
generalizing the Gorenstein projective resolution of $M$ constructed by Ollivier and Schneider if $R$ is a field (cf.\ Proposition \ref{acyclic} (ii), Theorem \ref{quasi_inverse_coefficient_system} and Remark \ref{GP_resolution}).\\

The definition of the category $\C$ and the proof of the above equivalence relies on the representation theory of finite reductive groups as developed by Cabanes (cf.\ \cite{Cab}). In \S\ref{subsection_3_1} we take up this theory and reprove it in a framework which is sufficiently general for our purposes. First of all, we need to work over an arbitrary quasi-Frobenius ring $R$ and the underlying $R$-modules of our representations are not necessarily finitely generated. Moreover, beyond representations of the finite reductive groups $P_F/I_F$ we are interested in representations of the groups $P_F^\dagger/I_F$. Since the corresponding Hecke algebras $H_F^\dagger$ are generally not selfinjective (cf.\ Remark \ref{not_selfinjective}) the strategy of Cabanes does not apply directly. In Definition \ref{condition_H} we introduce a certain condition (H) on smooth $R$-linear representations of $P_F$ or $P_F^\dagger$. It is a generalization of condition $(**)$ in \cite{Cab} to filtered unions. We obtain full subcategories of $\Rep_R^\infty(P_F)$ and $\Rep_R^\infty(P_F^\dagger)$ which are equivalent to the corresponding categories of Hecke modules (cf.\ Theorem \ref{Cabanes} and Theorem \ref{Cabanes_general} (ii)). We also show that a quasi-inverse is given by a variant of the above functor $\t_F$ (cf.\ Theorem \ref{quasi_inverse_OS}). Even in the case of finite reductive groups this does not seem to have been observed before. \\

Given a representation $V\in\Rep_R^\infty(G)$ the corresponding fixed point system $\F_V$ may or may not belong to the category $\C$ (cf.\ the discussion after Remark \ref{basic_0_diagrams}). However, it turns out that an object $\F\in\Coeff_G(\BT)$ belongs to the category $\C$ if and only if for any vertex $x$ of $\BT$ the restriction of $\F$ to the star of $x$ is a Ronan-Smith sheaf associated to a $P_x$-representation satisfying condition (H) (cf.\ Proposition \ref{characterization_C}). At the end of \S\ref{subsection_3_2} we also relate our constructions to the torsion theory of Ollivier and Schneider introduced in \cite{OS2}. Given a Hecke module $M\in\Mod_H$ there is a functorial $G$-equivariant surjection $\H_0(\BT,\F(M))\to\t(M)$ whose precise behavior presently remains unclear (cf.\ Proposition \ref{comparison_representations} and Remark \ref{t_isomorphism}).\\

The second part of our article is concerned with applications to the category $\Rep_R^\infty(G)$ of smooth $R$-linear $G$-representations. We denote by $\Rep_R^I(G)$ the full subcategory of $\Rep_R^\infty(G)$ consisting of all representations generated by their $I$-invariants.\\

In \S\ref{subsection_4_1} we are mainly concerned with the case that $R$ is a field of characteristic zero. In this case we simply have $\F(M)\cong\F_X\otimes_HM$ for any $M\in\Mod_H$ and $\F(V^I)\cong\F_V$ for any $V\in\Rep_R^I(G)$ (cf.\ Theorem \ref{description_functors_0}). According to Bernstein, the functor $(\cdot)^I:\Rep_R^I(G)\to\Mod_H$ is an equivalence of abelian categories. Since we could not find a reference checking the hypotheses of \cite{Ber}, Corollaire 3.9, we give a quick argument relying on the known case of a first congruence subgroup (cf.\ Theorem \ref{Bernstein}). As a consequence, the $0$-th homology functor $\H_0(\BT,\cdot):\C\to\Rep_R^I(G)$ is an equivalence of categories (cf.\ Corollary \ref{H_0_equivalence} and Corollary \ref{Schneider_Stuhler} (i)). Moreover, we can follow an argument of Broussous to show that the augmented oriented chain complex $0\to\C_c^\ori(\BT_{(\bullet)},\F_V)\to V\to 0$ is exact for any $V\in\Rep_R^I(G)$ (cf.\ Corollary \ref{Schneider_Stuhler} (ii)). This is a particular case of a much more general result of Schneider and Stuhler (cf.\ \cite{SS}, Theorem II.3.1). Finally, we use Bernstein's theorem to interprete the Zelevinski involution in terms of certain Ext-duals on the category $\Mod_H$. If $\GG$ is semisimple we use the homological properties of $H$ as established by Ollivier and Schneider to prove the main properties of the Zelevinski involution (cf.\ Theorem \ref{Zelevinski}). In fact, this interpretation of the Zelevinski involution shows that it has good properties way beyond the case of admissible representations studied classically (cf.\ Remark \ref{d_step_duality}).\\

At the end of \S\ref{subsection_4_1} we assume that $p$ is nilpotent in $R$. If $M\in\Mod_H$ then the exactness of the complex $\C_c^\ori(\BT_{(\bullet)},\F(M))$ remains an open problem. At least, we can treat the case that the semisimple rank of $\GG$ is equal to one. Under these assumptions the augmented complex
\[
 0\longrightarrow\C_c^\ori(\BT_{(1)},\F(M))\longrightarrow\C_c^\ori(\BT_{(0)},\F(M))\longrightarrow\H_0(\BT,\F(M))\longrightarrow 0
\]
is exact and there is an $H$-linear embedding $M\hookrightarrow\H_0(\BT,\F(M))^I$ (cf.\ Proposition \ref{semisimple_rank_one}). If the underlying $R$-module of $M$ is finitely generated the $G$-representation $\H_0(\BT,\F(M))$ may thus be called finitely presented. However, it is generally not irreducible and admissible and the embedding $M\hookrightarrow\H_0(\BT,\F(M))^I$ is generally not an isomorphism (cf.\ Remark \ref{non_admissible}). We also show $\F(\cdot)\cong\F_X\otimes_H(\cdot)$ if $G=\mathrm{SL}_2(K)$, $G=\GL_2(K)$ or $G=\mathrm{PGL}_2(K)$ and if the residue class field of $K$ is the field with $p$ elements (cf.\ Proposition \ref{exceptional_flat}).\\

If $p$ is nilpotent in $R$ then the functor $(\cdot)^I:\Rep_R^I(G)\to\Mod_H$ is generally not fully faithful (cf.\ \cite{Oll1}, Th\'eor\`eme). In \S\ref{subsection_4_2} we introduce the full subcategory $\Rep_R^\ind(G)$ of $\Rep_R^I(G)$ consisting of all representations isomorphic to a finite direct sum of compactly induced representations $\ind_{P_F^\dagger}^G(V_F)$ where $V_F\in\Rep_R^\infty(P_F^\dagger)$ satisfies the generalized condition (H) of Cabanes and $F$ is contained in the closure of the chamber fixed by $I$. Likewise, we denote by $\Mod_H^\ind$ the full subcategory of $\Mod_H$ consisting of all $H$-modules isomorphic to a finite direct sum of scalar extensions $H\otimes_{H_F^\dagger}M_F$ with $M\in\Mod_{H_F^\dagger}$ (cf.\ Definition \ref{reductive_ind}). If $R$ is an arbitrary quasi-Frobenius ring then the functor $(\cdot)^I:\Rep_R^\ind(G)\to\Mod_H^\ind$ turns out to be an equivalence of additive categories (cf.\ Theorem \ref{equivalence_induced_objects}). Particular cases of its fully faithfulness have also been shown by Ollivier and Vign\'eras (cf.\ Remark \ref{fully_faithful_hyperspecial}). However, the building theoretic arguments needed to treat the general case are much more involved (cf.\ Proposition \ref{invariants_induction_H} and Remark \ref{subtle_set_of_roots}). Our proof also uses the full force of the relation between $P_F^\dagger$-representations and $H_F^\dagger$-modules as developed in \S\ref{subsection_3_1}.\\

In Proposition \ref{module_homotopy_equivalence} we use the functorial resolutions (\ref{generalized_GP_resolution}) to show that in a suitable sense the inclusion $\Mod_H^\ind\subseteq\Mod_H$ induces a triangle equivalence on the level of bounded derived categories. One can then use the equivalence $\Rep_R^\ind(G)\cong\Mod_H^\ind$ to realize the bounded derived category of $H$-modules as a somewhat exotic localization of the homotopy category of bounded complexes over $\Rep_R^\ind(G)$ (cf.\ Theorem \ref{homotopy_equivalence}). However, our results rather indicate that the true relation between $\Rep_R^I(G)$ and $\Mod_H$ is that of a Quillen equivalence. We will come back to this in a future work. Moreover, since the equivalence $\Rep_R^\ind(G)\cong\Mod_H^\ind$ is generally not compatible with the homological properties of the categories $\Rep_R^\infty(G)$ and $\Mod_H$ the relation to Schneider's derived equivalence in \cite{Sch}, Theorem 9, is presently unclear.\\

In \S\ref{subsection_4_3} we relate our constructions to the theory of generalized $(\varphi,\Gamma)$-modules as developed by Schneider and Vign\'eras (cf.\ \cite{SV}). Note that the more general definitions and constructions of \cite{SV} make sense for any artinian coefficient ring $R$. We choose a Borel subgroup $\Pbar$ of $G$ and a suitable vector chamber $\Cscr^0$ of $\BT$ which is stabilized by a certain submonoid $\Pbar^+$ of $\Pbar$. Any vector chamber $\Cscr$ of $\BT$ contained in $\Cscr^0$ gives rise to the subcomplex $\BT^+(\Cscr)=\Pbar^+\Cscr$ of the building $\BT$. For any coefficient system $\F\in\Coeff_G(\BT)$ we consider the complex of $R$-modules
\begin{equation}\label{etale_complex_intro}
 \varinjlim_{\Cscr\subseteq\Cscr^0}\C_c^\ori(\BT^+_{(\bullet)}(\Cscr),\F)^*
\end{equation}
where $(\cdot)^*$ denotes the $R$-linear dual and the transition maps in the inductive limit are dual to the inclusions $\C_c^\ori(\BT^+_{(\bullet)}(\Cscr'),\F)\subseteq \C_c^\ori(\BT^+_{(\bullet)}(\Cscr),\F)$ for all vector chambers $\Cscr'\subseteq\Cscr\subseteq\Cscr^0$. This procedure is vaguely reminiscent of passing to the stalk at a boundary point of $\BT$. We observe that the complex (\ref{etale_complex_intro}) carries actions of the completed monoid rings $R\llbracket\Pbar^+\rrbracket$ and $R\llbracket(\Pbar^+)^{-1}\rrbracket$ introduced in \cite{SV}, \S1. In fact, the $R\llbracket\Pbar^+\rrbracket$-module structure is always \'etale (cf.\ Proposition \ref{etale}). If $p$ is nilpotent in $R$ and if the semisimple rank of $\GG$ is equal to one then the complex (\ref{etale_complex_intro}) is acyclic and its $0$-th cohomology group is non-trivial (cf.\ Proposition \ref{non_vanishing}). Moreover, if the $R$-modules $\F_F$ are finitely generated for any face $F$ of $\BT$ then there is an isomorphism of complexes of $R\llbracket(\Pbar^+)^{-1}\rrbracket$-modules
\[
 \varinjlim_{\Cscr\subseteq\Cscr^0}\C_c^\ori(\BT^+_{(\bullet)}(\Cscr),\F)^* \cong
 D(\C_c^\ori(\BT_{(\bullet)},\F).
\]
Here $D$ is the functor introduced in \cite{SV}, \S2. In fact, under suitable assumptions we establish the existence of an $E_2$-spectral sequence of \'etale $R\llbracket\Pbar^+\rrbracket$-modules
\[
 D^j\H_i(\BT,\F)\Longrightarrow\H^{j+i}(\varinjlim_{\Cscr\subseteq\Cscr^0}\C^\ori_c(\BT_{(i)}^+(\Cscr),\F)^*)
\]
where $(D^j)_{j\geq 0}$ is the universal $\delta$-functor of Schneider and Vign\'eras (cf.\ Proposition \ref{etale_spectral_sequence} (ii) and \cite{SV}, \S4). Here we need to put ourselves in the situation of \cite{SV}. More precisely, we assume $K=\Qp$ and $R=o/\pi^no$ for the valuation ring $o$ of some finite field extension of $\Qp$, a uniformizer $\pi\in o$ and a positiver integer $n$. We note in passing that the proof of Proposition \ref{etale_spectral_sequence} (i) shows the invariance of the $\delta$-functor $(D^j)_{j\geq 0}$ under central isogenies -- a result which seems interesting on its own.\\

Assume in addition that the semisimple rank of $\GG$ is equal to one and that $\F=\F(M)$ for some $M\in\Mod_H$. If the underlying $R$-module of $M$ is finitely generated and if the center of $G$ acts on $M$ through a character then the above spectral sequence degenerates  (cf.\ Proposition \ref{etale_spectral_sequence} (ii)). In this case the \'etale $R\llbracket \Pbar^+\rrbracket$-module $D^j\H_0(\BT,\F(M))$ is the $j$-th cohomology group of the complex (\ref{etale_complex_intro}) and we have $D^j\H_0(\BT,\F(M))=0$ for all $j\geq 1$. Note that the $\Pbar$-representation $\H_0(\BT,\F(M))$ is actually finitely presented (cf.\ Proposition \ref{semisimple_rank_one} and the proof of Proposition \ref{SV_functor}). However, we do not know if it is admissible so that \cite{SV}, Remark 11.4, might not be applicable.\\

Finally, assume that $G=\GL_2(\Qp)$, $R=o/\pi o$ and $\F=\F(V^I)$ for some admissible representation $V\in\Rep_R^I(G)$ admitting a central character. If $o$ is chosen suitably then Ollivier's equivalence of categories and the comparison results of \cite{SV}, \S11, show that the complex (\ref{etale_complex_intro}) eventually leads to the \'etale $(\varphi,\Gamma)$-module corresponding to $V$ under the $p$-adic local Langlands correspondence of Colmez (cf.\ Remark \ref{generalized_phi_Gamma}). We hope that our geometric constructions will be useful in extending this correspondence to other groups.\\

{\bf Acknowledgments.} We would like to emphasize that the present article would have been unthinkable without the previous work of Cabanes, Ollivier, Schneider and Vign\'eras. The articles \cite{Cab} and \cite{OS1} have been a particularly important source of inspiration. When writing this article the author was a member of the SFB/TR 45 {\it ''Periods, Moduli Spaces, and Arithmetic of Algebraic Varieties``}. He gratefully acknowledges the financial support of the DFG. He would also like to thank the Heinloths for an enlightening discussion about locally constant coefficient systems.\\

{\bf Notation and conventions.} Throughout the article $R$ will denote a fixed commutative unital ring. Let $K$ be a nonarchimedean local field with normalized valuation $\mathrm{val}$ and valuation ring $\o$. We denote by $k$ the residue class field of $K$ and by $p$ and $q$ the characteristic and the cardinality of $k$, respectively. For any unital ring $S$ we denote by $\Mod_S$ the category of $S$-modules. Unless specified otherwise an $S$-module will always mean a left $S$-module. For any topological monoid $J$ we denote by $\Rep_R^\infty(J)$ the category of $R$-linear smooth representations of $J$, i.e.\ the category of all $R$-modules $V$ carrying an $R$-linear action of $J$ such that the stabilizer of any element $v\in V$ is open in $J$. We shall write $\Hom_J(V,W)$ for the $R$-module of $R$-linear and $J$-equivariant maps between two objects $V,W\in\Rep_R^\infty(J)$. If $J$ is a group and if $J_0\subseteq J$ is an open subgroup then we denote by $\ind_{J_0}^J:\Rep_R^\infty(J_0)\to\Rep_R^\infty(J)$ the compact induction functor (cf.\ \cite{Vig1}, \S I.5).


\section{A reminder on the Bruhat-Tits building}\label{section_1}


\subsection{Stabilizers and Bruhat decompositions}\label{subsection_1_1}

Let $\GG$ denote a split connected reductive group over $K$. We fix a maximal split $K$-torus $\TT$ of $\GG$ and let $\CC$ denote the connected component of the center of $\GG$. Let $d$ denote the semisimple rank of $\GG$, i.e.\ the dimension of a maximal split $K$-torus of the derived group of $\GG$. This is equal to the dimension of $\TT/\CC$. By $G=\GG(K)$ and $T=\TT(K)$ we denote the group of $K$-rational points of $\GG$ and $\TT$, respectively.\\

Let $\BT$ denote the semisimple Bruhat-Tits building of $G$ (cf.\ \cite{Tit}) and let $\A=X_*(\TT/\CC)\otimes_{\mathbb{Z}}\mathbb{R}$ denote the apartment of $\BT$ corresponding to $T$. Recall that $\BT$ is a $d$-dimensional polysimplicial complex with a simplicial action of $G$ whose $0$-dimensional (resp.\ $d$-dimensional) faces are usually called the vertices (resp.\ the chambers) of $\BT$. For any face $F$ of $\BT$ we denote by
\[P_F^\dagger=\{g\in G\;|\;gF=F\}\]
the stabilizer of $F$ in $G$. We denote by $\BT^1$ the enlarged Bruhat-Tits building of $G$ in the sense of \cite{BT2}, \S4.2.16, and denote by $\mathrm{pr}:\BT^1\to\BT$ the projection map. The pointwise stabilizer of $\mathrm{pr}^{-1}(F)$ in $G$ is the group of $\o$-rational points
\[\GG_F(\o)\subseteq\GG_F(K)=\GG(K)=G\]
of a smooth group scheme $\GG_F$ over $\o$ with generic fiber $\GG$ (cf.\ \cite{Tit}, \S3.4.1). We denote by $\mathring{\GG}_F$ the connected component of $\GG_F$ and by
\[P_F=\mathring{\GG}_F(\o)\]
its group of $\o$-rational points. It is called the parahoric subgroup of $G$ associated with $F$. Let $\pi_F:P_F=\mathring{\GG}_F(\o)\to\mathring{\GG}_F(k)$ denote the group homomorphism induced by the residue class map $\o\to k$, and let $\mathrm{R}^u(\mathring{\GG}_{F,k})$ denote the unipotent radical of the special fiber $\mathring{\GG}_{F,k}$ of $\mathring{\GG}_F$. We then obtain the pro-$p$ group
\[I_F=\pi_F^{-1}(\mathrm{R}^u(\mathring{\GG}_{F,k})(k))\subseteq P_F\]
which is in fact the pro-$p$ radical of $P_F$.\\

The origin $x_0$ of $\A=X_*(\TT/\CC)\otimes_{\mathbb{Z}}\mathbb{R}$ is a hyperspecial vertex in $\BT$. Throughout the article we fix a chamber $C$ in $\A$ containing $x_0$ and set
\[
I=I_C\quad\mbox{and}\quad I'=P_C.
\]
The subgroups $I$ and $I'$ are called a pro-$p$ Iwahori subgroup and an Iwahori subgroup of $G$, respectively. The chamber $C$ determines a set $\Phi^+$ of positive roots of the root system $\Phi\subseteq X^*(\TT/\CC)$ of $(\GG,\TT)$. We shall also view the elements of $\Phi$ as characters of $\TT$ and $T$.\\

If $T_0$ denotes the maximal compact subgroup of $T$ and if $N_G(T)$ denotes the normalizer of $T$ in $G$ then we denote by
\[
W=N_G(T)/T_0\cong T/T_0\rtimes W_0
\]
the extended Weyl group of $(G,T)$. Here $W_0=N_G(T)/T$ denotes the finite Weyl group of $(G,T)$. The action of $G$ on $\BT$ restricts to an action of $W$ on $\A$ by affine automorphisms. We denote by $\langle\cdot,\cdot\rangle:X^*(\TT/\CC)\times X_*(\TT/\CC)\to\mathbb{Z}$ the canonical pairing and by $\nu:T\to X_*(\TT/\CC)$ the group homomorphism chacterized by
\[
\langle\alpha,\nu(t)\rangle=-\mathrm{val}(\alpha(t))\quad\mbox{for all}\quad\alpha\in\Phi.
\]
The action of $t\in T$ on $\A$ is then given by translation with $\nu(t)$.\\

If $F'$ and $F$ are faces of $\BT$ such that $F'$ is contained in the topological closure $\overline{F}$ of $F$ then by \cite{BT2}, Proposition 4.6.24 (i), or \cite{Tit}, \S3.4.3, we have the inclusions
\begin{equation}\label{inclusions}
I_{F'}\subseteq I_F\subseteq P_F\subseteq P_{F'}\subseteq P_{F'}^\dagger
\end{equation}
in which $I_{F'}$ and $P_{F'}$ are normal in $P_{F'}^\dagger$. Moreover, we have $I\cap P_F^\dagger=I\cap P_F$ by \cite{OS1}, Lemma 4.10, which combined with (\ref{inclusions}) yields
\begin{equation}\label{Iwahori_intersection_inclusion}
 I\cap P_F^\dagger=I\cap P_F\subseteq I\cap P_{F'}=I\cap P_{F'}^\dagger.
\end{equation}
By \cite{Vig2}, Proposition 1, the group $W$ is equipped with a length function $\ell:W\to\mathbb{N}$ such that
\[\Omega=\{w\in W\;|\;\ell(w)=0\}\]
is an abelian subgroup of $W$ contained in $P_C^\dagger$, hence normalizes $I$ and $I'$. According to \cite{Vig3}, Appendice, the group $\Omega$ is isomorphic to the quotient of $X_*(\TT)$ modulo the subgroup generated by the coroots $\check{\Phi}$. Consequently, $\Omega$ is finite if and only if $\GG$ is semisimple. Further, $W=W_\aff\rtimes\Omega$ is the semidirect product of $\Omega$ and the so-called affine Weyl group $W_\aff$. This is an affine Coxeter group generated by the reflections about the affine roots of $G$ (cf.\ \cite{OS1}, \S4.3). Moreover, the length function $\ell$ is constant on the double cosets $\Omega\backslash W/\Omega$. Let $T_1$ denote the unique pro-$p$ Sylow subgroup of $T_0$ and
\[
\tilde{W}=N_G(T)/T_1\cong (T_0/T_1)\rtimes W.
 \]
We denote by $\tilde{\Omega}$ and $\tilde{W}_\aff$ the preimage of $\Omega$ and $W_\aff$ under the surjection $\tilde{W}\to W$, respectively, so that $\tilde{W}/\tilde{\Omega}\cong W_\aff$. We extend the length function $\ell$ to $\tilde{W}$ by inflation, i.e.\ we have $\ell(\omega w\omega')=\ell(w)$ for all $\omega,\omega'\in\tilde{\Omega}$ and $w\in\tilde{W}$. Note that the group $\tilde{W}$ acts on $\A$ through its quotient $W$. \\

Denote by $G_\aff$ the subgroup of $G$ generated by the parahoric subgroups $P_F$ of all faces $F$ of $\BT$. By \cite{BT2}, Proposition 5.2.12, and \cite{Tit}, \S3.3.1, we have the Bruhat decompositions
\begin{equation}\label{Bruhat_decomposition}
G=\coprod_{w\in W}I'wI'=\coprod_{\tilde{w}\in\tilde{W}}I\tilde{w}I\quad\mbox{and}\quad G_\aff=\coprod_{w\in W_\aff}I'wI'=\coprod_{\tilde{w}\in\tilde{W}_\aff}I\tilde{w}I.
\end{equation}
Here we follow the usual abuse of notation using that the double cosets $I'wI'$ and $I\tilde{w}I$ do not depend on the choice of representatives of $w$ and $\tilde{w}$ in $N_G(T)$ because $T_0\subseteq I'$ and $T_1=T_0\cap I\subseteq I$. The group homomorphism $T_0/T_1\to I'/I$ is actually bijective.\\

For any face $F$ of $\BT$ contained in $\overline{C}$ we adopt the notation of \cite{OS1}, \S4, and denote by $W_F$ the subgroup of $W$ generated by all affine reflections fixing $F$ pointwise. Further, we let $\Omega_F=\{w\in\Omega\;|\;wF=F\}$. The subgroup $W_F^\dagger$ of $W$ generated by $W_F$ and $\Omega_F$ is the semidirect product $W_F^\dagger= W_F\rtimes\Omega_F$. The group $W_F$ is always finite whereas $\Omega_F$ is finite if and only if $\GG$ is semisimple. We also note that the canonical surjection $W\to W_0$ restricts to an isomorphism $W_{x_0}\cong W_0$ (cf.\ \cite{Tit}, \S1.9).\\

Let $\tilde{\Omega}_F$, $\tilde{W}_F$ and $\tilde{W}_F^\dagger$ denote the preimages of $\Omega_F$, $W_F$ and $W_F^\dagger$ in $\tilde{W}$, respectively. By \cite{OS1}, Lemma 4.9, these groups give rise to the decompositions
\begin{equation}\label{Bruhat_parahoric}
P_F=\coprod_{w\in W_F}I'wI'=\coprod_{\tilde{w}\in\tilde{W}_F}I\tilde{w}I\quad\mbox{and}\quad P_F^\dagger=\coprod_{w\in W_F^\dagger}I'wI'=\coprod_{\tilde{w}\in\tilde{W}_F^\dagger}I\tilde{w}I.
\end{equation}
It follows that $P_F^\dagger/P_F\cong\tilde{W}_F^\dagger/\tilde{W}_F\cong\Omega_F$.\\

We continue to assume that $F\subseteq\Cbar$. There is a set $D_F\subseteq W$ of representatives of the left cosets $W/W_F$ which is characterized by the property that for any $d\in D_F$ the element $d$ is of minimal length in $dW_F$ (cf.\ \cite{OS1}, Proposition 4.6). The set $D_F$ is stable under right multiplication with elements of $\Omega_F$ (cf.\ \cite{OS1}, Lemma 4.11). Moreover, if $F'$ is a face with $F'\subseteq\overline{F}\subseteq\Cbar$ then we have $W_F\subseteq W_{F'}$. Therefore, any element which is of minimal length in its left coset modulo $W_{F'}$ is also of minimal length in its left coset modulo $W_F$. Thus,
\begin{equation}\label{D_F_inclusion}
D_{F'}\subseteq D_F\quad\mbox{whenever}\quad F'\subseteq\overline{F}\subseteq\Cbar.
\end{equation}
\begin{rem}\label{Omega_in_DF}
Note that $\Omega\subseteq D_F$ for any face $F\subseteq\Cbar$. Indeed, if $\omega\in\Omega$ and $w\in W_F$ then $\ell(\omega w)=\ell(w)$ whence $\omega$ is of minimal length in $\omega W_F$.
\end{rem}
We denote by $\tilde{D}_F\subseteq\tilde{W}$ the preimage of $D_F$ in $\tilde{W}$ under the surjection $W\to\tilde{W}$. The length function $\ell:\tilde{W}\to\mathbb{N}$ factors through $W$ whence \cite{OS1}, Proposition 4.6 (i), implies
\begin{equation}\label{additive_length}
\ell(dw)=\ell(d)+\ell(w)\quad\mbox{for all}\quad d\in\tilde{D}_F,w\in\tilde{W}_F.
\end{equation}
\begin{lem}\label{face_representatives}
For any face $F$ of $\BT$ there is a unique face $[F]$ of $\BT$ which is contained in $\Cbar$ and $G_\aff$-conjugate to $F$. If $F'$ is another face of $\BT$ and if $[F']$ denotes its unique $G_\aff$-conjugate in $\Cbar$, then $F'\subseteq \overline{F}$ implies $[F']\subseteq\overline{[F]}$.
\end{lem}
\begin{proof}
Since $\BT$ is a building there is an apartment containing $F$ and $C$ (cf.\ \cite{BT}, Th\'eor\`eme 7.4.18 (i)). By \cite{BT}, Corollaire 7.4.9 (i), there is an element $i\in P_C^\dagger$ with $iF\subseteq\A$. By (\ref{Bruhat_parahoric}) we may even assume $i\in I$. By \cite{BGL}, V.3.2, Th\'eor\`eme 1, there is an element $w\in W_\aff$ with $wiF\subseteq\Cbar$. This proves the existence of $[F]$. The uniqueness follows from $P_F^\dagger\cap G_\aff=P_F$ (cf.\ \cite{OS1}, Lemma 4.10). The last statement is a direct consequence of the uniqueness because $G_\aff$ acts by simplicial automorphisms.
\end{proof}


\subsection{Hecke algebras}\label{subsection_1_2}

Following the notation of \cite{OS1}, \S2, we set
\[
X=\ind_I^G(R)\in\Rep_R^\infty(G).
\]
Note that $X$ is a right module over its opposite endomorphism ring
\[
H=\End_G(X)^\op,
\]
the so-called pro-$p$ Iwahori-Hecke algebra of $G$. By Frobenius reciprocity (cf.\ \cite{Vig1}, Proposition I.5.7 (ii)), $H$ can be identified with the $R$-module
\[
H=\Hom_G(X,X)\cong X^I=\ind_I^G(R)^I= R[I\backslash G/I]
\]
of $I$-biinvariant compactly supported maps $G\to R$. Under this identification, the multiplication of $H$ is the convolution product
\[
 (f\cdot f')(g')=\sum_{g\in G/I}f(g)f'(g^{-1}g').
\]
Let $V\in\Rep_R^\infty(G)$. By Frobenius reciprocity again, the $R$-module $V^I\cong\Hom_G(X,V)$ of $I$-invariants of $V$ naturally is a left $H$-module via $f\cdot\varphi=\varphi\circ f$ for any $f\in H=\End_G(X)^\op$ and $\varphi\in\Hom_G(X,V)$, i.e.\ we have the functor
\[
(\cdot)^I:\Rep_R^\infty(G)\longrightarrow\Mod_H,\quad V\mapsto V^I.
\]
Interpreting $f$ as an element of $R[I\backslash G/I]$ and $\varphi$ as an element of $V^I$, the above module structure is given by
\begin{equation}\label{convolution}
f\cdot\varphi=\sum_{g\in I\backslash G/I}\sum_{g'\in I/(I\cap gIg^{-1})}f(g)g'g\cdot\varphi\in V^I.
\end{equation}
Given $w\in\tilde{W}$ we set $\tau_w=IwI\in R[I\backslash G/I]=H$ and call $\tau_w$ the Hecke operator associated with $w$. By (\ref{Bruhat_decomposition}) the $R$-module $H$ is free with basis $(\tau_w)_{w\in\tilde{W}}$. The ring structure of $H$ is determined by the relations
\begin{eqnarray}
\label{braid_relations}
 \tau_{vw}&=&\tau_v\tau_w\quad\mbox{if}\quad\ell(vw)=\ell(v)+\ell(w)\quad\mbox{and}\\
\label{quadratic_relations}
 \tau_s^2&=&q\tau_{s^2}+\tau_s\theta_s
\end{eqnarray}
for any element $s\in\tilde{W}$ whose image in $W$ belongs to the set of distinguished generators of the Coxeter group $W_\aff$ (cf.\ \cite{Vig2}, Theorem 1). Here $\theta_s$ is a specific element of $R[T_0/T_1]\subseteq H$.\\

The following result is essentially proven in \cite{OS1}, Proposition 4.13.
\begin{lem}\label{closest_chamber}
Let $F$ be a face of $\BT$. Among all chambers of $\BT$ which contain $F$ in their closure there is a unique one $C(F)$ with minimal gallery distance to $C$. It satisfies $I_{C(F)}=I_F(I\cap P_F^\dagger)=I_F(I\cap P_F)$ and $I\cap P_F^\dagger=I\cap P_{C(F)}^\dagger=I\cap P_{C(F)}$. The chamber $C(F)$ is contained in any apartment of $\BT$ containing $F$ and $C$. Moreover, we have $gC(F)=C(gF)$ for all $g\in I$.
\end{lem}
\begin{proof}
Let $\A'$ be an arbitrary apartment of $\BT$ containing $F$ and $C$. As seen in the proof of Lemma \ref{face_representatives} there is $\gamma\in I$ with $\gamma\A'=\A$. Thus, $\gamma F\subseteq\A$. According to \cite{OS1}, Proposition 4.13, there is a unique chamber $C(\gamma F)$ in $\A$ containing $\gamma F$ in its closure and which has minimal gallery distance to $C$. Note that the gallery distance of \cite{OS1}, Proposition 4.13, refers to the gallery distance computed in $\A$. By the existence of retractions, however, this agrees with the gallery distance computed in the entire building $\BT$ (cf.\ \cite{AB}, Corollary 4.34).  Let us put $C(F)=\gamma^{-1}C(\gamma F)\subseteq\A'$. Clearly, this is a chamber of $\BT$ containing $F$ in its closure. \\

Let us denote by $d(\cdot\,,\cdot)$ the gallery distance function and assume that $D$ is a chamber of $\BT$ containing $F$ in its closure and which has minimal gallery distance to $C$. As above, there is $g\in I$ such that $gD\subseteq\A$. This implies $gF\subseteq\A$ and hence $gF=\gamma F$ because $\A$ contains a unique $I$-conjugate of $F$ (cf.\ \cite{OS1}, Remark 4.17 (2)). On the one hand, this gives
\[d(gD,C)\geq d(C(\gamma F),C)\]
because $gD$ contains $gF=\gamma F$ in its closure and because of the minimality property of $C(\gamma F)$. On the other hand,
\[d(gD,C)=d(D,C)\leq d(g^{-1}C(\gamma F),C)=d(C(\gamma F),C)\]
because of the minimality property of $D$ and since $g$ fixes $C$. Note that $g^{-1}C(\gamma F)$ contains $g^{-1}\gamma F=F$ in its closure. Therefore, $d(gD,C)=d(C(\gamma F),C)$ and $gD=C(\gamma F)$ by the uniqueness of $C(\gamma F)$. By \cite{OS1}, Proposition 4.13 (ii), we have $g\gamma^{-1}\in I\cap P_{\gamma F}^\dagger\subseteq I_{C(\gamma F)}\subseteq P_{C(\gamma F)}^\dagger$, whence $D=g^{-1}C(\gamma F)=\gamma^{-1}C(\gamma F)=C(F)$. This proves the first part of the lemma. By (\ref{Iwahori_intersection_inclusion}) and \cite{OS1}, Proposition 4.13 (ii), we have $I_{C(\gamma F)}=I_{\gamma F}(I\cap P_{\gamma F}^\dagger)$ and $I\cap P_{\gamma F}=I\cap P_{\gamma F}^\dagger=I\cap P_{C(\gamma F)}^\dagger=I\cap P_{C(\gamma F)}$. Conjugation with $\gamma^{-1}\in I$ yields the second part of the lemma. The final assertions follow directly from the construction.
\end{proof}
Let $F$ be a face of $\BT$ and let the chamber $C(F)$ be as in Lemma \ref{closest_chamber}. Since $I_{C(F)}\subseteq P_F\subseteq P_F^\dagger$ we have the objects
\[X_F^\dagger=\ind_{I_{C(F)}}^{P_F^\dagger}(R)\in\Rep_R^\infty(P_F^\dagger)
\quad\mbox{and}\quad
X_F=\ind_{I_{C(F)}}^{P_F}(R)\in\Rep_R^\infty(P_F),\]
following the notation of \cite{OS1}, \S3.3. They are naturally right modules over their opposite endomorphism rings
\begin{equation}\label{H_F}
H_F^\dagger=\End_{P_F^\dagger}(X_F^\dagger)^\op
\quad\mbox{and}\quad
H_F=\End_{P_F}(X_F)^\op,
\end{equation}
the so-called Hecke algebras at $F$. As above, the $R$-algebras $H_F^\dagger$ and $H_F$ can be identified with the $R$-modules
\[
H_F^\dagger\cong R[I_{C(F)}\backslash P_F^\dagger/I_{C(F)}]
\quad\mbox{and}\quad
H_F\cong R[I_{C(F)}\backslash P_F/I_{C(F)}]
\]
of $I_{C(F)}$-biinvariant compactly supported maps $P_F^\dagger\to R$ and $P_F\to R$, respectively. The multiplication is again the convolution product. It makes $H_F$ an $R$-subalgebra of $H_F^\dagger$.\\

If $F\subseteq\Cbar$ then the $R$-modules $H_F$ and $H_F^\dagger$ are free with bases $(\tau_w)_{w\in\tilde{W}_F}$ and $(\tau_w)_{w\in\tilde{W}_F^\dagger}$, respectively. This follows from (\ref{Bruhat_parahoric}). Moreover, the $R$-algebra $H_F^\dagger$ may then be viewed as a subalgebra of $H$.
\begin{lem}\label{Hecke_generation}
As an $R$-algebra, $H$ is generated by the subalgebras $H_F^\dagger$ with $F\subseteq\Cbar$.
\end{lem}
\begin{proof}
By the braid relations (\ref{braid_relations}) the $R$-algebra $H$ is generated by the Hecke operators $\tau_\omega$ with $\omega\in\tilde{\Omega}$ and the Hecke operators $\tau_s$ such that the image of $s$ in $W$ belongs to the set of distinguished generators of the Coxeter group $W_\aff$. Note that $\tau_\omega\in H_C^\dagger$ for all $\omega\in\tilde{\Omega}$. Further, any $s$ as above induces an affine reflection of $\A$ pointwise fixing a unique codimension one face $F(s)$ in $\Cbar$. Thus, $\tau_s$ is an element of $H_{F(s)}\subseteq H_{F(s)}^\dagger$.
\end{proof}
Let $F$ be an arbitrary face of $\BT$. If $F'$ is a face of $\BT$ with $F'\subseteq\overline{F}\subseteq\overline{C(F')}$ then $C(F)=C(F')$ by the uniqueness assertion of Lemma \ref{closest_chamber}. In this case
\[
 H_F^\dagger\cap H_{F'}^\dagger=R[I_{C(F)}\backslash(P_F^\dagger\cap P_{F'}^\dagger)/I_{C(F)}]
\]
inside the $R$-module of all compactly supported maps $G\to R$. Note that by (\ref{inclusions}) we have $I_{C(F)}=I_{C(F')}\subseteq P_F\subseteq P_{F'}$ so that in this situation $X_F$ is a $P_F$-subrepresentation of $X_{F'}$ and
\[
H_{F'}\supseteq H_F\subseteq H_F^\dagger\cap H_{F'}^\dagger.
\]

If $F\subseteq\Cbar$ then any $G_\aff$-conjugate of $F$ is of the form $gdF$ for some $d\in\tilde{D}_F$ and $g\in I$ (cf.\ the proof of Lemma \ref{face_representatives}). As a consequence of Lemma \ref{closest_chamber} and \cite{OS1}, Proposition 4.13, we have
\begin{equation}\label{d_conjugation}
gdC=C(gdF)\quad\mbox{and}\quad gd\, I(gd)^{-1}=I_{C(gdF)}.
\end{equation}
Thus, we have the isomorphisms
\begin{equation}\label{phi_gd}
 \varphi_{gd,F}:\left\{
\begin{array}{l}
X_F^\dagger\longrightarrow X_{gdF}^\dagger,\\
H_F^\dagger\longrightarrow H_{gdF}^\dagger,\\
X_F\longrightarrow X_{gdF},\\
H_F\longrightarrow H_{gdF},
\end{array}\right.
\end{equation}
of $R$-modules all of which are given by $f\mapsto(g'\mapsto f(gdg'(gd)^{-1}))$. Here conjugation with $d$ is defined by choosing a representative in $N_G(T)$. Note, however, that $T_1\subseteq I_F$. Therefore, the $I_F$-biinvariance of $f$ and the fact that $I_F$ is a normal subgroup of $I$ imply that the value $f(gdg'd^{-1}g^{-1})$ does not depend on the choice of a representative of $d$ in $N_G(T)$.\\

Clearly, on $H_F^\dagger$ and $H_F$ the map $\varphi_{gd,F}$ is an isomorphism of $R$-algebras. On $X_F^\dagger$ (resp.\ $X_F$) it is an isomorphism of representations of $P_F^\dagger$ (resp.\ of $P_F$) if the action on $X_{gdF}^\dagger$ (resp.\ on $X_{gdF}$) is pulled back along conjugation with $gd$.\\

Note that if $F'\subseteq\overline{F}\subseteq\Cbar$ then the restrictions of $\varphi_{gd,F}$ and $\varphi_{gd,F'}$ to $H_F^\dagger\cap H_{F'}^\dagger$ agree for all $g\in I$ and $d\in\tilde{D}_{F'}\subseteq \tilde{D}_F$. Here the last inclusion results from (\ref{D_F_inclusion}). In particular, we have $\varphi_{gd,F'}|_{H_F}=\varphi_{gd,F'}$ as isomorphisms $H_F\to H_{gdF}$.\\

If $R$ is a field then the structural results in Proposition \ref{free} and Proposition \ref{I_F_invariants} below all appear in \cite{OS1}. However, the proofs work more generally.
\begin{prop}\label{free}Let $F$ be a face of $\BT$.
\begin{enumerate}[wide]
\item[(i)]$H_F^\dagger$ is free as a left and as a right module over $H_F$.
\item[(ii)]If $F\subseteq\Cbar$ then $H$ is a free left and a free right module over $H_F$ and $H_F^\dagger$.
\item[(iii)]If $F'$ is a face of $\BT$ with $F'\subseteq\overline{F}$ and $C(F')=C(F)$ then $H_{F'}$ is finitely generated and free as a left and as a right module over $H_F$.
\end{enumerate}
\end{prop}
\begin{proof}
As for (i), let $[F]\subseteq\Cbar$ be as in Lemma \ref{face_representatives}. By construction, there are elements $g\in I$ and $w\in W_\aff$ with $F=gw[F]$. If $d\in \tilde{D}_{[F]}$ is chosen so that its image in $W$ is the element of minimal length in $wW_{[F]}$ we also have $gd[F]=F$. The compatible isomorphisms $\varphi_{gd,F}$ in (\ref{d_conjugation}) then allow us to assume $F\subseteq\Cbar$. In this situation, the decomposition (\ref{Bruhat_parahoric}) and the braid relations (\ref{braid_relations}) imply that if $S\subseteq\tilde{\Omega}_F$ denotes a complete set of representatives of the coset space $\tilde{\Omega}_F/(T_0/T_1)\cong\Omega_F$ then the family $(\tau_\omega)_{\omega\in S}$ is a basis of $H_F^\dagger$ as a left and as a right $H_F$-module.\\

The proof of (ii) is identical to that of \cite{OS1}, Proposition 4.21 (i). It relies only on the braid relations of $H$ which hold over any coefficient ring $R$. We note that a basis of $H$ as a right $H_F^\dagger$-module is given by the elements $\tau_d$ if $d$ runs through a system of representatives of $D_F$ modulo $\Omega_F$.\\

As for (iii), the above arguments, Lemma \ref{face_representatives} and the relation $\varphi_{gd,F'}|_{H_F}=\varphi_{gd,F'}$ again allow us to assume $F'\subseteq\overline{F}\subseteq\Cbar$. In this situation, (\ref{additive_length}) and the braid relations (\ref{braid_relations}) imply that the family $(\tau_{d^{-1}})_{d\in\tilde{W}_{F'}\cap\tilde{D}_F}$ (resp.\ the family $(\tau_d)_{d\in\tilde{W}_{F'}\cap\tilde{D}_F}$) is a basis of $H_{F'}$ as a left (resp.\ as a right) $H_F$-module. Note that the group $\tilde{W}_{F'}$ and its subset $\tilde{W}_{F'}\cap\tilde{D}_F$ are finite.
\end{proof}
\begin{rem}\label{group_ring}If $F\subseteq\Cbar$ then there is a more precise version of Proposition \ref{free} (i). In fact, the decomposition (\ref{Bruhat_parahoric}) implies that mapping $\omega\in\tilde{\Omega}_F$ to the Hecke operator $\tau_\omega=I\omega I=\omega I\in H_F^\dagger$ gives a well-defined and injective homomorphism $R[\tilde{\Omega}_F]\to H_F^\dagger$ of $R$-algebras. As in \cite{OS1}, Lemma 4.20 (i), the $R$-algebra $H_F^\dagger$ is a twisted tensor product of $H_F$ and $R[\tilde{\Omega}_F]$ over $R[T_0/T_1]$. Note that $\tau_\omega$ is a unit in $H_F^\dagger$ with inverse $\tau_\omega^{-1}=\tau_{\omega^{-1}}$. If $V\in\Rep_R^\infty(P_F^\dagger)$ and if $m\in V^I$ then $\tau_\omega\cdot m=\omega\cdot m$ by (\ref{convolution}).
\end{rem}
\begin{prop}\label{I_F_invariants}Let $F$ be a face of $\BT$.
\begin{enumerate}[wide]
\item[(i)]The map $X_F\otimes_{H_F}H_F^\dagger\to X_F^\dagger$, $f\otimes h\mapsto h(f)$, is a well-defined isomorphism of $(R[P_F],H_F^\dagger)$-bimodules.
\item[(ii)]If $F\subseteq\Cbar$ then the maps $X_F\otimes_{H_F}H\to X^{I_F}$ and $X_F^\dagger\otimes_{H_F^\dagger}H\to X^{I_F}$, $f\otimes h\mapsto h(f)$, are well-defined isomorphisms of $(R[P_F],H)$- and $(R[P_F^\dagger],H)$-bimodules, respectively.
\item[(iii)]If $F'$ is a face of $\BT$ with $F'\subseteq\overline{F}$ and $C(F')=C(F)$ then the map $X_F\otimes_{H_F}H_{F'}\to X_{F'}^{I_F}$, $f\otimes h\mapsto h(f)$, is a well-defined isomorphism of $(R[P_F],H_{F'})$-bimodules.
\end{enumerate}
\end{prop}
\begin{proof}
The inclusions $P_F\subseteq P_F^\dagger$ and $H_F\subseteq H_F^\dagger$ make $X_F$ a subobject of the $(R[P_F],H_F)$-bimodule $X_F^\dagger$. Since $H_F^\dagger$ acts on $X_F\otimes_{H_F}H_F^\dagger$ and $X_F^\dagger$ by homomorphisms of left $R[P_F]$-modules, the map in (i) is well-defined and $(R[P_F],H_F^\dagger)$-linear. The corresponding assertions in (ii) and (iii) are proved similarly, noting that the actions of $I_F$ on $X_F$ and $X_F^\dagger$ are trivial. The rest of the proof of part (i) is identical to that of \cite{OS1}, Lemma 4.24. Note that as in the proof of Proposition \ref{free} (i) we may assume $F\subseteq\Cbar$.\\

Now consider part (ii). Using (i) it suffices to prove the bijectivity of the map $X_F\otimes_{H_F}H\to X^{I_F}$. Fix $d\in \tilde{D}_F$ and let $X^{I_F}_d$ denote the $R$-submodule of $X^{I_F}$ consisting of all functions supported on $P_Fd^{-1}I$. As in the proof of \cite{OS1}, Proposition 4.25, it suffices to see that the map $X_F\to X^{I_F}_d$ given by $f\mapsto\tau_{d^{-1}}(f)$ is bijective. Note that this makes use of Proposition \ref{free} (ii). By (\ref{Bruhat_parahoric}) we have the decomposition
\[
P_F=\coprod_{w\in W_F}\coprod_{g\in I/(wIw^{-1}\cap I)}gwI.
\]
The characteristic functions of the sets $gwI$ form a basis of the $R$-module $X_F$. Applying $\tau_{d^{-1}}$ each of them is mapped to the characteristic function of the set $gwId^{-1}I$. We need to see that these form a basis of the $R$-module $X^{I_F}_d$ or equivalently that the decomposition
\[
P_Fd^{-1}I=\bigcup_{w\in \tilde{W}_F}\bigcup_{g\in I/(wIw^{-1}\cap I)}gwId^{-1}I
\]
is disjoint and consists of double cosets modulo $(I_F,I)$. According to the proof of \cite{OS1}, Proposition 4.25, we have $Id^{-1}I=I_Fd^{-1}I$ where $I_F$ is a normal subgroup of $P_F$. Therefore, $gwId^{-1}I=gwI_Fd^{-1}I=I_Fgwd^{-1}I$ is indeed a double coset as required. In order to see that the above union is disjoint, let $g,g'\in I$ and $w,w'\in \tilde{W}_F$ be such that $I_Fgwd^{-1}I=I_Fg'w'd^{-1}I$. Since $I_F\subseteq I$ left multiplication with $I$ gives $w=w'$ by (\ref{Bruhat_decomposition}). Consequently, $w^{-1}g^{-1}g'w\in I_Fd^{-1}Id\cap w^{-1}Iw$. Now $dC=C(dF)$ by \cite{OS1}, Proposition 4.13, whence
\begin{eqnarray*}
I_Fd^{-1}Id\cap w^{-1}Iw & = & I_F(d^{-1}Id\cap P_F)\cap w^{-1}Iw\\
& = & d^{-1}(I_{dF}(I\cap P_{dF}))d\cap w^{-1}Iw \\
& = & d^{-1}I_{C(dF)}d\cap w^{-1}Iw = I\cap w^{-1}Iw
\end{eqnarray*}
by Lemma \ref{closest_chamber}. Note that $wIw^{-1}$ and $I_F$ are subgroups of $P_F$. This proves (ii). Given Proposition \ref{free} (iii), the proof of part (iii) is analogous.
\end{proof}


\section{Coefficient systems}\label{section_2}


\subsection{Coefficient systems and diagrams}\label{subsection_2_1}

Let $\Y$ be a polysimplicial complex or more generally any subset of a polysimplicial complex which is a union of some of its faces. Following \cite{God}, \S I.3.3, \cite{RS}, \S1, or \cite{SS}, \S{II.2}, a coefficient system of $R$-modules on $\Y$ is a family $\F=((\F_F)_F, (r^F_{F'})_{F'\subseteq\overline{F}})$ of $R$-modules $\F_F$ indexed by the faces $F$ of $\Y$, together with $R$-linear maps $r^F_{F'}:\F_F\to\F_{F'}$ for any pair of faces $F,F'$ of $\Y$ with $F'\subseteq\overline{F}$ such that
\[
 r_F^F=\id_{\F_F}\quad\mbox{and}\quad r^F_{F''}=r^{F'}_{F''}\circ r^F_{F'}
\]
whenever $F''\subseteq \overline{F}'\subseteq\overline{F}$. The maps $r^F_{F'}$ are usually called the restriction maps of the coefficient system $\F$.\\

A homomorphism $f:\F\to\G$ of coefficient systems of $R$-modules on $\Y$ is a family $(f_F)_F$ of $R$-linear maps $f_F:\F_F\to\G_F$ indexed by the faces $F$ of $\Y$ such that the diagram
\[
 \xymatrix{\F_F\ar[d]_{r^F_{F'}}\ar[r]^{f_F} & \G_F\ar[d]^{r^F_{F'}} \\
 \F_{F'}\ar[r]_{f_{F'}} & \G_{F'} }
\]
commutes whenever $F$ and $F'$ are faces of $\Y$ with $F'\subseteq\overline{F}$. We denote by $\Coeff(\Y)$ the category of $R$-linear coefficient systems on $\Y$, oppressing the symbol $R$ from the notation. It is an $R$-linear abelian category in the obvious way.
\begin{ex}\label{constant_system}
For any $R$-module $M$ the family $\K_M=((M)_F,(\id_M)_{F'\subseteq\overline{F}})$ is a coefficient system of $R$-modules on $\Y$. It is called the {\it constant} coefficient system on $\Y$ associated with $M$.
\end{ex}
Assume there is a group $J$ which acts on $\Y$ by simplicial automorphisms. Given $j\in J$ and $\F\in\Coeff(\Y)$ we let $j_*\F$ denote the object of $\Coeff(\Y)$ defined by $j_*\F_F=\F_{jF}$ with transition maps $j_*r^F_{F'}=r^{jF}_{jF'}$. Note that we have $j_*i_*\F=(ij)_*\F$ for any $j,i\in J$. Further, any homomorphism $f:\F\to\G$ in $\Coeff(\Y)$ naturally induces a homomorphism $j_*f:j_*\F\to j_*\G$ by $(j_*f)_F=f_{jF}$. A $J$-equivariant coefficient system of $R$-modules on $\Y$ is an object $\F\in\Coeff(\Y)$ together with a family $(c_j)_{j\in J}$ of homomorphisms $c_j:\F\to j_*\F$ in $\Coeff(\Y)$ such that $c_1=\id_\F$ and $j_*c_i\circ c_j=c_{ij}$ for all $j,i\in J$. The latter is usually called the cocycle relation.\\

If $\F$ and $\G$ are $J$-equivariant coefficient systems on $\Y$ and if $(c_j)_{j\in J}$ and $(d_j)_{j\in J}$ denote the $J$-actions on $\F$ and $\G$, respectively, then a morphism $f:\F\to\G$ in $\Coeff(\Y)$ is called $J$-equivariant if $j_*f\circ c_j=d_j\circ f$ for all $j\in J$. We denote by $\Coeff_J(\Y)$ the category of $J$-equivariant coefficient systems of $R$-modules on $\Y$. Note that a morphism $f=(f_F)_F$ in $\Coeff_J(\Y)$ is an isomorphism if and only if $f_F$ is bijective for all faces $F$ of $\Y$.\\

If $\F\in\Coeff_J(\Y)$ with $J$-action $(c_j)_{j\in J}$ and if $F$ is a face of $\Y$ then the stabilizer of $F$ in $J$ acts on $\F_F$ by $jm=c_j(m)$ for all $m\in\F_F$. If $F'$ and $F$ are faces of $\Y$ with $F'\subseteq\overline{F}$ then the restriction map $r^F_{F'}$ is equivariant for the intersection of the stabilizer groups of $F$ and $F'$.\\

In the particular case of the Bruhat-Tits building $\Y=\BT$ and $J=G$, we say that $\F$ is of level zero if the action of $I_F$ on $\F_F$ is trivial for any face $F$ of $\BT$. We denote by $\Coeff_G^0(\BT)$ the full subcategory of $\Coeff_G(\BT)$ consisting of all $G$-equivariant coefficient systems on $\BT$ which are of level zero.
\begin{ex}\label{fixed_point_system}
Given a smooth $R$-linear $G$-representation $V\in\Rep_R^\infty(G)$ the associated fixed point system $\F_V\in\Coeff_G^0(\BT)$ is defined by $(\F_V)_F=V^{I_F}$ with the inclusions induced by (\ref{inclusions}) as restriction maps. For any face $F$ of $\BT$ and any element $g\in G$ the map $c_{g,F}:V^{I_F}\to V^{I_{gF}}=V^{gI_Fg^{-1}}$ is defined by $c_{g,F}(v)=gv$.
\end{ex}
Let $\Y$ be a polysimplicial complex of dimension $d$. For $0\leq i\leq d$ we denote by $\Y_i$ the set of $i$-dimensional faces of $\Y$ and by $\Y_{(i)}$ the set of oriented $i$-dimensional faces of $\Y$ in the sense of \cite{SS}, \S II.1. The elements of $\Y_{(i)}$ are pairs $(F,c)$ where $c$ is an orientation of $F$ with the convention that the $0$-dimensional faces always carry the trivial orientation. Note that if $i>0$ then also $(F,-c)\in\Y_{(i)}$ where $-c$ denotes the orientation opposite to $c$. If $(F,c)\in\Y_{(i)}$ and if $F'\in\Y_{i-1}$ with $F'\subseteq\overline{F}$ then we denote by $\partial^F_{F'}(c)$ the induced orientation of $F'$. It satisfies $\partial_{F'}^F(-c)=-\partial_{F'}^F(c)$.\\

Given an object $\F\in\Coeff(\Y)$ we denote by $(\C^\ori_c (\Y_{(\bullet)},\F),\partial_\bullet)$ the oriented chain complex
\begin{equation}\label{oriented_chain_complex}
 0\longrightarrow\C_c^\ori(\Y_{(d)},\F)\stackrel{\partial_{d-1}}{\longrightarrow} \ldots \stackrel{\partial_1}{\longrightarrow}\C_c^\ori(\Y_{(1)},\F) \stackrel{\partial_0}{\longrightarrow}\C_c^\ori(\Y_{(0)},\F) \longrightarrow 0
\end{equation}
of $\F$ in analogy to \cite{SS}, \S II.2. Recall that $\C_c^\ori(\Y_{(i)},\F)$ denotes the $R$-module of $i$-dimensional oriented chains on $\Y$ with values in $\F$, i.e.\ the $R$-module of finitely supported maps
\[f:\Y_{(i)}\to\coprod_{(F,c)\in\Y_{(i)}}\F_F\]
such that $f((F,c))\in\F_F$ for any element $(F,c)\in\Y_{(i)}$ and such that $f((F,-c))=-f((F,c))$ in case $i\geq 1$. The $R$-linear differentials are defined by
\begin{equation}\label{differentials}
 \partial_i(f)(F',c')=\sum_{\scriptsize\begin{array}{ccc}(F,c)\in\Y_{(i+1)}\\F'\subseteq\overline{F},\;\partial^F_{F'}(c)=c'\end{array}}r^F_{F'}(f(F,c))
\end{equation}
for any $f\in\C_c^\ori(\Y_{(i+1)},\F)$ and $(F',c')\in\Y_{(i)}$ with $0\leq i\leq d-1$. Formally, we let $\partial_{d+1}=\partial_{-1}=0$. For $0\leq i\leq d$ the $R$-module
\[
 H_i(\Y,\F)=\ker(\partial_{i-1})/\im(\partial_i)
\]
is called the $i$-th homology group of $\Y$ with coefficients in $\F$. If $\Y$ carries a simplicial action of the group $J$ and if $\F$ is a $J$-equivariant coefficient system on $\Y$ with $J$-action $(c_j)_{j\in J}$ then (\ref{oriented_chain_complex}) is a complex of smooth $R$-linear $J$-representations via
\[
(j\cdot f)(F,c)=c_{j,j^{-1}F}(f(j^{-1}F,j^{-1}c)).
\]
In this situation, the homology groups of $\Y$ with coefficients in $\F$ are objects of $\Rep^\infty_R(J)$, as well.
\begin{defin}\label{local_system}
A coefficient system $\F\in\Coeff(\Y)$ is called locally constant if the transition maps $r^F_{F'}:\F_F\to\F_{F'}$ are bijective for all faces $F',F$ of $\Y$ with $F'\subseteq\overline{F}$.
\end{defin}
Recall that the open star of a face $F'$ of $\Y$ is defined as
\[\St(F')=\bigcup_{F'\subseteq\overline{F}}F.\]
It is a contractible open neighborhood of $F'$ in $\Y$. Note also that for any subset $\mathscr{Z}$ of $\Y$ which is a union of faces we have the restriction functor $\Coeff(\Y)\to\Coeff(\mathscr{Z})$ defined by sending a coefficient system $\F$ on $\Y$ to
\[
 \F|_{\mathscr{Z}}=((\F_F)_{F\subseteq\mathscr{Z}},(r^F_{F'})_{F'\subseteq\overline{F}\subseteq\mathscr{Z}})\quad\mbox{in}\quad \Coeff(\mathscr{Z}).
\]
\begin{rem}\label{criterion_loc_constant}A coefficient system $\F$ is locally constant if and only if for all faces $F'$ of $\Y$ the restriction of $\F$ to $\St(F')$ is isomorphic to a constant coefficient system. Indeed, if $\F$ is locally constant then $(r^F_{F'})_{F'\subseteq\overline{F}}:\F|_{\St(F')}\to\K_{\F_{F'}}$ is an isomorphism in $\Coeff(\St(F'))$. One can show that on a simply connected polysimplicial complex any locally constant coefficient system is isomorphic to a constant coefficient system as in Example \ref{constant_system}.
\end{rem}
We view $\Cbar$ as a finite subcomplex of $\BT$. The following terminology was first introduced in \cite{Pas}, \S5.5.
\begin{defin}\label{diagram}A diagram of $R$-modules on $\Cbar$ is a $P_C^\dagger$-equivariant coefficient system $((\D_F)_F,(r^F_{F'})_{F'\subseteq\overline{F}},(c_g)_{g\in P_C^\dagger})$ on $\Cbar$ together with an $R$-linear action of $P_F^\dagger$ on $\D_F$ for each face $F\subseteq\Cbar$ such that
\begin{enumerate}[wide]
 \item[(i)] the $(P_F^\dagger\cap P_C^\dagger)$-action agrees with the action induced by $(c_g)_{g\in P_F^\dagger\cap P_C^\dagger}$,
 \item[(ii)] for all $g\in P_C^\dagger$ and for all $h\in P_F^\dagger$ we have $ghg^{-1}\circ c_{g,F} = c_{g,F}\circ h$,
 \item[(iii)] for all faces $F'\subseteq\overline{F}$ the restriction map $r^F_{F'}$ is $(P_F^\dagger\cap P_{F'}^\dagger)$-equivariant.
\end{enumerate}
\end{defin}
In other words, a diagram is a $P_C^\dagger$-equivariant coefficient system of $R$-modules on $\Cbar$ such that on each face $F\subseteq\Cbar$ the action of $P_C^\dagger\cap P_F^\dagger$ on $\D_F$ is extended to $P_F^\dagger$ in a way which is compatible with restriction and such that for any $g\in P_C^\dagger$ the maps $c_{g,F}:\D_F\to\D_{gF}$ are $P_F^\dagger$-equivariant if the action on $\D_{gF}$ is pulled back along conjugation with $g$.\\

A homomorphism $f:\D\to\E$ of diagrams is a homomorphism of the underlying $P_C^\dagger$-equivariant coefficient systems on $\Cbar$ such that for any face $F\subseteq\Cbar$ the $R$-linear map $f_F:\D_F\to\E_F$ is $P_F^\dagger$-equivariant. We denote by $\Diag(\Cbar)$ the corresponding category of diagrams on $\Cbar$. A diagram $\D$ is called of level zero if the action of $I_F$ on $\D_F$ is trivial for any face $F\subseteq\Cbar$. We denote by $\Diag^0(\Cbar)\subseteq\Diag(\Cbar)$ the full subcategory of diagrams on $\Cbar$ which are of level zero.
\begin{rem}\label{comparison_Paskunas}
If $G=\GL_2(K)$ then the notion of a diagram was originally introduced by Paskunas (cf.\ \cite{Pas}, Definition 5.14). His definition does not literally agree with ours. Namely, in the case of a tree the closed chamber $\Cbar$ is an edge with two adjacent vertices whereas Paskunas only works with an edge and one of its vertices. Since these form a complete set of representatives of the $G$-orbits in $\BT$ there is indeed some redundance in Definition \ref{diagram}. However, the case of a tree is special in that the above set of representatives also reflects all possible face relations in $\BT$. In the general case, it is not clear that a complete system of representatives can be chosen in such a way that any two incident faces also have incident representatives. In any case, the aim is to use the transitivity properties of the $G$-action on $\BT$ to reduce the information encoded in a $G$-equivariant coefficient systems to a finite amount of data. Taking into account \cite{Pas}, Theorem 5.17, and Proposition \ref{diagrams} below, the two definitions lead to equivalent categories.
\end{rem}
\begin{prop}\label{diagrams}
The restriction functor
\begin{eqnarray*}
 \res:\Coeff_G(\BT)&\longrightarrow&\Diag(\Cbar)\\
 ((\F_F)_{F\subseteq\BT}, (r^F_{F'})_{F'\subseteq\overline{F}\subseteq\BT},(c_g)_{g\in G})&\longmapsto&((\F_F)_{F\subseteq\Cbar}, (r^F_{F'})_{F'\subseteq\overline{F}\subseteq\Cbar},(c_g)_{g\in P_C^\dagger})\\
 (f_F)_{F\subseteq\BT}&\longmapsto&(f_F)_{F\subseteq\Cbar}
\end{eqnarray*}
is a well-defined equivalence of categories. It restricts to an equivalence of categories $\Coeff_G^0(\BT)\longrightarrow\Diag^0(\Cbar)$.
\end{prop}
\begin{proof}Let $\F$ be a $G$-equivariant coefficient system on $\BT$. Endowing $\F_F$ with the induced $R$-linear action of $P_F^\dagger$, the family $\res(\F)$ is clearly a diagram on $\Cbar$ and $\res$ is a functor. Since it preserves objects of level zero, it suffices to prove the first assertion.\\

For any face $F$ of $\BT$ we fix an element $g_F\in G_\aff$ such that $g_FF\subseteq\Cbar$ (cf.\ Lemma \ref{face_representatives}). If $F\subseteq\Cbar$ then we choose $g_F=1$. Depending on these choices we shall construct a quasi-inverse of the restriction functor as follows. We let $[F]=g_FF$ which is independent of $g_F\in G_\aff$ by Lemma \ref{face_representatives}. Let $\D$ be a diagram on $\Cbar$ with $P_C^\dagger$-action $(c_g)_{g\in P_C^\dagger}$. If $F$ is a face of $\BT$ we define the $R$-module
\[\F_F=\D_{[F]}.\]
If $F'$ is a face of $\BT$ with $F'\subseteq\overline{F}$ then $g_FF'\subseteq g_F\overline{F}\subseteq\Cbar$ whence $g_FF'=[F']=g_{F'}F'$ by the uniqueness assertion of Lemma \ref{face_representatives}. In other words, $g_{F'}g_F^{-1}\in P_{[F']}^\dagger\cap G_\aff=P_{[F']}$ where the last equality comes from \cite{OS1}, Lemma 4.10. We define
\[r^F_{F'}=g_{F'}g_F^{-1}\circ r^{[F]}_{[F']},\]
using the given $P_{[F']}^\dagger$-action on $\D_{[F']}$. Let us first check that the family $\F=((\F_F)_F,(r^F_{F'})_{F'\subseteq\overline{F}})$ lies in $\Coeff(\BT)$. Clearly, $r^F_F=\id_{\F_F}$ for all faces $F$ of $\BT$. If $F$, $F'$ and $F''$ are faces of $\BT$ with $F''\subseteq\overline{F'}\subseteq\overline{F}$ then $g_{F'}g_F^{-1}\in P_{F'}\subseteq P_{F''}$ by (\ref{inclusions}) and we have
\begin{eqnarray*}
r^F_{F''} &=& g_{F''}g_F^{-1}\circ r^{[F]}_{[F'']}\\
&=& g_{F''}g_{F'}^{-1}\circ g_{F'}g_F^{-1}\circ r^{[F']}_{[F'']}\circ r^{[F]}_{[F']}\\
&=& (g_{F''}g_{F'}^{-1} \circ r^{[F']}_{[F'']})\circ ( g_{F'}g_F^{-1} \circ r^{[F]}_{[F']})\\
&=& r^{F'}_{F''}\circ r^F_{F'}
\end{eqnarray*}
because of the $(P_{[F']}^\dagger\cap P_{[F'']}^\dagger)$-equivariance of $r^{[F']}_{[F'']}$.\\

In order to define the required $G$-action on $\F$ let $g\in G$. Then $g_{gF}gg_F^{-1}$ maps $[F]=g_FF$ to $[gF]$, both of which are contained in $\Cbar$. By \cite{OS1}, Lemma 4.12 (ii) and Remark 4.14, there are elements $g_1\in P_C^\dagger$ and $g_2\in P_{[F]}^\dagger$ with $g_{gF}gg_F^{-1}=g_1g_2$. We then define $c_{g,F}:\F_F\to\F_{gF}$ as the composition
\[\F_F=\D_{[F]}\stackrel{g_2}{\longrightarrow}\D_{[F]}\stackrel{c_{g_1,[F]}}{\longrightarrow}\D_{[gF]}=\F_{gF},\]
using the $P_C^\dagger$-action on $\D$ and the $P_{[F]}^\dagger$-action on $\D_{[F]}$. Note first that the $R$-linear map $c_{g,F}$ is independent of the chosen product decomposition. Indeed, if $g_1'\in P_C^\dagger$ and $g_2'\in P_{[F]}^\dagger$ with $g_{gF}gg_F^{-1}=g_1g_2=g_1'g_2'$ then $g_1^{-1}g_1'=g_2(g_2')^{-1}\in P_C^\dagger\cap P_{[F]}^\dagger$ whence
\begin{eqnarray*}
c_{g_1,[F]}\circ g_2 &=& c_{g_1,[F]}\circ g_2(g_2')^{-1}g_2'\\
&=& c_{g_1,[F]} \circ c_{g_2(g_2')^{-1},[F]} \circ g_2'\\
&=& c_{g_1,[F]} \circ c_{g_1^{-1}g_1',[F]} \circ g_2'\\
&=& c_{g_1',[F]} \circ g_2'
\end{eqnarray*}
as $R$-linear maps $\D_{[F]}\to\D_{[gF]}$ by the compatibility of the actions of $P_C^\dagger$ on $\D$ and $P_{[F]}^\dagger$ on $\D_{[F]}$. Consequently, if $F\subseteq\Cbar$ and if $g\in P_C^\dagger$ then $g_F=g_{gF}=1$ by convention and the action map $c_{g,F}:\F_F=\D_F\to\D_{gF}=\F_{gF}$ agrees with the original one of the diagram $\D$. Similarly, if $F\subseteq\Cbar$ and if $g\in P_F^\dagger$ then $c_{g,F}$ is equal to the original action of $g$ on $\F_F=\D_F$ as an element of $P_F^\dagger$. Moreover, we have $c_{1,F}=\id_{\F_F}$ for all faces $F$ of $\BT$.\\

Let us now show that the families $c_g=(c_{g,F})_{F\subseteq\BT}$ satisfy the cocycle relation $h_*c_g\circ c_h=c_{gh}$ for all $h,g\in G$. To see this, let $F$ be an arbitrary face of $\BT$. Choose elements $g_1,g_1'\in P_C^\dagger$, $g_2\in P_{[F]}^\dagger$ and $g_2'\in P_{[hF]}^\dagger$ such that $g_{hF}hg_F^{-1}=g_1g_2$ and $g_{ghF}gg_{hF}^{-1}=g_1'g_2'$. In particular, we have $g_1[F]=[hF]$ and $g_1'[hF]=[ghF]$ whence
\[
g_{ghF}ghg_F^{-1}=g_{ghF}gg_{hF}^{-1}\cdot g_{hF}hg_F^{-1}=g_1'g_2'\cdot g_1g_2=g_1'g_1\cdot(g_1^{-1}g_2'g_1)\cdot g_2
\]
where $g_1 $ and $g_1'$ are contained in $P_C^\dagger$ and $g_1^{-1}g_2'g_1$ and $g_2$ are contained in $P_{[F]}^\dagger$. We use this decomposition and Definition \ref{diagram} (ii) to compute
\begin{eqnarray*}
c_{gh,F} &=& c_{g_1'g_1,[F]}\circ g_1^{-1}g_2'g_1g_2\\
&=& c_{g_1',[hF]}\circ c_{g_1,[F]}\circ g_1^{-1}g_2'g_1\circ g_2\\
&=& c_{g_1',[hF]}\circ g_2' \circ c_{g_1,[F]} \circ g_2\\
&=& c_{g,hF}\circ c_{h,F}.
\end{eqnarray*}
Using the cocycle relation we can also show that $c_g=(c_{g,F})_F:\F\to g_*\F$ is a homomorphism of coefficient systems on $\BT$. Fix faces $F'$ and $F$ of $\BT$ with $F'\subseteq\overline{F}$ and choose $g_1,g_1'\in P_C^\dagger$, $g_2\in P_{[F]}^\dagger$ and $g_2'\in P_{[F']}^\dagger$ with
\[g_{gF}gg_F^{-1}=g_1g_2\quad\mbox{and}\quad g_{gF'}gg_{F'}^{-1}=g_1'g_2'.\]
Since $P_C^\dagger P_F^\dagger=P_C^\dagger P_F$ by (\ref{Bruhat_parahoric}) we may assume $g_2\in P_F$, whence $g_2\in (P_F^\dagger\cap P_{F'}^\dagger)$ by (\ref{inclusions}). Recall that $g_FF'=g_{F'}F'=[F']$, $g_{gF}gF'=g_{gF'}gF'=[gF']$ and that consequently $g_1g_2[F']=g_{gF}gg_F^{-1}[F']=[gF']$. We can therefore consider
\[
 \xymatrix@C=2cm{\D_{[F]}\ar[r]^{g_2}\ar[d]_{r^{[F]}_{[F']}} & \D_{[F]} \ar[r]^{c_{g_1,[F]}} \ar[d]^{r^{[F]}_{[F']}} & \D_{[gF]} \ar[d]^{r^{[gF]}_{[gF']}} \\
 \D_{[F']} \ar[r]_{g_2} & \D_{[F']} \ar[r]_{c_{g_1,[F']}} & \D_{[gF']}.}
\]
Since the two small squares are commutative by the definition of $\D$, also the outer square is commutative and we get
\begin{eqnarray*}
 r^{gF}_{gF'}\circ c_{g,F} &=& g_{gF'}g_{gF}^{-1}\circ r^{[gF]}_{[gF']} \circ c_{g_1,[F]}\circ g_2\\
 &=& g_{gF'} g_{gF}^{-1} \circ c_{g_1,[F']}\circ g_2 \circ r^{[F]}_{[F']}\\
 &=& c_{g_{gF'} g_{gF}^{-1},[gF']}\circ c_{g_1,[F']}\circ c_{g_2,[F']} \circ r^{[F]}_{[F']}\\
 &=& c_{g_{gF'} g_{gF}^{-1}g_1g_2,[F']}\circ r^{[F]}_{[F']}\\
 &=& c_{g_1'g_2'g_{F'}g_F^{-1},[F']}\circ r^{[F]}_{[F']}\\
 &=& c_{g_1',[F']}\circ g_2'\circ g_{F'}g_F^{-1}\circ r^{[F]}_{[F']} = c_{g,[F']}\circ r^F_{F'}.
\end{eqnarray*}
If $f:\D_1\to\D_2$ is a homomorphism of diagrams and if $\F_1$ and $\F_2$ denote the associated $G$-equivariant coefficient systems on $\BT$ as above, then we extend $f$ to homomorphism $f:\F_1\to\F_2$ in $\Coeff_G(\BT)$ by setting $f_F=f_{[F]}$. If $F'$ and $F$ are faces of $\BT$ with $F'\subseteq\overline{F}$ then
\begin{eqnarray*}
 f_{F'}\circ r^F_{F'} &=& f_{[F]'}\circ g_{F'}g_F^{-1}\circ r^{[F]}_{[F']}= g_{F'}g_F^{-1}\circ f_{[F']}\circ r^{[F]}_{[F']}\\
 &=&g_{F'}g_F^{-1}\circ r^{[F]}_{[F']} \circ f_{[F]}= r^F_{F'}\circ f_F
\end{eqnarray*}
by the properties of $f:\D_1\to\D_2$ because $g_{F'}g_F^{-1}\in P_{[F']}^\dagger$. Moreover, if $g\in G$ write $g_{gF}gg_F^{-1}=g_1g_2$ with $g_1\in P_C^\dagger$ and $g_2\in P_{[F]}^\dagger$. Since $g_1[F]=g_1g_2[F]=g_{gF}gg_F^{-1}[F]=g_{gF}gF=[gF]$ we can compute
\begin{eqnarray*}
f_{gF}\circ c_{g,F} &=& f_{[gF]} \circ c_{g_1,[F]} \circ g_2 = f_{g_1[F]} \circ c_{g_1,[F]} \circ g_2\\
&=& c_{g_1,[F]} \circ f_{[F]}\circ g_2 = c_{g_1,[F]} \circ g_2 \circ f_{[F]}\\
&=& c_{g,F} \circ f_F
\end{eqnarray*}
by the properties of $f:\D_1\to\D_2$.\\

Clearly, these constructions yield a functor $\Diag(\Cbar)\to\Coeff_G(\BT)$ whose composition with the restriction functor is the identity on $\Diag(\Cbar)$. Conversely, if $\F\in\Coeff_G(\BT)$ and if $\E$ denotes the $G$-equivariant coefficient system on $\BT$ constructed from the diagram $\res(\F)$, then the family $\iota=(c_{g_F,F})_{F\subseteq\BT}:\F_F\to\F_{g_FF}=\E_F$ of $R$-linear maps is an isomorphism in $\Coeff_G(\BT)$. Indeed, each map $c_{g_F,F}$ is bijective with inverse $c_{g_F^{-1},g_FF}$ by the cocycle relation. If $F'$ and $F$ are faces of $\BT$ with $F'\subseteq\overline{F}$ then
\begin{eqnarray*}
r^F_{F'} \circ \iota_F &=&r^F_{F'} \circ c_{g_F,F}  = c_{g_{F'}g_F^{-1},[F']}\circ r^{[F]}_{[F']} \circ c_{g_F,F} \\
&=& c_{g_{F'}g_F^{-1},[F']} \circ c_{g_F,F'} \circ r^F_{F'}\\
&=& c_{g_{F'},F'} \circ r^F_{F'} = \iota_{F'}\circ r^F_{F'},
\end{eqnarray*}
using that $g_FF'=[F']$. Finally, if $g\in G$ and if $g_{gF}gg_F^{-1}=g_1g_2$ with $g_1\in P_C^\dagger$ and $g_2\in P_{[F]}^\dagger$ then
\begin{eqnarray*}
 c_{g,F}\circ\iota_F &=& c_{g,F}\circ c_{g_F,F} = c_{g_1,[F]}\circ c_{g_2,[F]}\circ c_{g_F,F}\\
 &=& c_{g_1g_2g_F,F} = c_{g_{gF},gF} \circ c_{g,F} = \iota_{gF}\circ c_{g,F}.
\end{eqnarray*}
Clearly, the formation of the isomorphism $\iota$ is functorial in the coefficient system $\F$. This completes the proof of the proposition.
\end{proof}


\subsection{Acyclic coefficient systems on the standard apartment}\label{subsection_2_2}

Let $\F\in\Coeff_G^0(\BT)$ be of level zero. Denote its restriction maps by $r^F_{F'}$ and its $G$-action by $(c_g)_{g\in G}$. Slightly generali\-zing \cite{OS1}, \S3.2, we define the coefficient system $\F^I\in\Coeff(\A)$ by setting
\[
 \F_F^I=\F_F^{I\cap P_F^\dagger}=\F_F^{I_F(I\cap P_F^\dagger)}=\F_F^{I_{C(F)}}\in\Mod_R\quad\mbox{for any face}\quad F\subseteq\A,
\]
where the chamber $C(F)\subseteq\A$ is as in Lemma \ref{closest_chamber}. The restriction maps $t^F_{F'}:\F_F^I\to\F_{F'}^I$ are defined by setting
\[
 t^F_{F'}(m)=\sum_{g\in(I\cap P_{F'}^\dagger)/(I\cap P_F^\dagger)}g\cdot r^F_{F'}(m)\quad\mbox{for all faces}\quad F'\subseteq\overline{F}\subseteq\A,
\]
making use of the inclusion relation (\ref{Iwahori_intersection_inclusion}). By the same reference the sum does not depend on the choice of the representatives $g$ because $m$ is $(I\cap P_F^\dagger)$-invariant and $r^F_{F'}$ is equivariant for the action of $I\cap P_F^\dagger=I\cap P_F\subseteq (P_F^\dagger\cap P_{F'}^\dagger)$. This also shows $t^F_{F''}=t^{F'}_{F''}\circ t^F_{F'}$ if $F''\subseteq\overline{F}'\subseteq\overline{F}$ so that we obtain the functor
\[(\F\mapsto\F^I):\Coeff_G^0(\BT)\to\Coeff(\A).\]
We need the following straightforward generalization of \cite{OS1}, Proposition 3.3. The idea of the proof goes back to Broussous (cf.\ \cite{Bro}, Proposition 11).
\begin{prop}\label{restriction_chain_complexes}
 Let $\F\in\Coeff_G^0(\BT)$. Restricting $I$-invariant oriented chains from $\BT$ to $\A$ induces an isomorphism
\[
 (\C^\ori_c(\BT_{(\bullet)},\F)^I,\partial_{\bullet})\stackrel{\cong}{\longrightarrow}(\C_c^\ori(\A_{(\bullet)},\F^I),\partial_{\bullet})
\]
of complexes of $R$-modules.
\end{prop}
\begin{proof}
Let $0\leq i\leq d$. That restriction $\C^\ori_c(\BT_{(i)},\F)^I\to\C_c^\ori(\A_{(i)},\F^I)$ is a well-defined isomorphism of $R$-modules is proven exactly as in \cite{OS1}, Proposition 3.3. In order to see that the restriction maps commute with the differentials we only slightly need to adjust the notation. Letting $(F',c')\in\BT_{(i)}$ and $f\in\C_c^\ori(\BT_{(i+1)},\F)^I$ we rewrite equation (3.7) of \cite{OS1} as
\begin{eqnarray*}
 \partial_{i+1}(f)(F',c')
 &=&
 \sum_{F\in\A_{i+1}\atop F'\subseteq\overline{F}}\sum_{g\in(I\cap P_{F'}^\dagger)/(I\cap P_F^\dagger)}r^{gF}_{F'}(c_{g,F}(f(F,c)))\\
 &=&
 \sum_{F\in\A_{i+1}\atop F'\subseteq\overline{F}}\sum_{g\in(I\cap P_{F'}^\dagger)/(I\cap P_F^\dagger)}r^{gF}_{gF'}(c_{g,F}(f(F,c)))\\
 &=&
 \sum_{F\in\A_{i+1}\atop F'\subseteq\overline{F}}\sum_{g\in(I\cap P_{F'}^\dagger)/(I\cap P_F^\dagger)}g\cdot r^{F}_{F'}(f(F,c))\\
 &=&
 \sum_{F\in\A_{i+1}\atop F'\subseteq\overline{F}}t^F_{F'}(f(F,c))=\partial_{i+1}(f)(F',c'),
\end{eqnarray*}
where $c$ always induces the orientation $c'$. Here we use that any face of $\BT$ has a unique $I$-conjugate in $\A$ and that $g(F',c')=(F',c')$ for any $g\in I\cap P_{F'}^\dagger$ (cf.\ \cite{OS1}, Lemma 3.1 and Remark 4.17.2).
\end{proof}
The coefficient system $\F^I\in\Coeff(\A)$ carries some additional structure coming from the $G$-action on $\F$. First of all, by Frobenius reciprocity
\[
 \F^I_F=\F_F^{I_{C(F)}}\cong\Hom_{P_F^\dagger}(X_F^\dagger,\F_F)\in\Mod_{H_F^\dagger}
\]
is a left module over the Hecke algebra $H_F^\dagger$ introduced in \S\ref{subsection_1_2}. Moreover, if $F'$ and $F$ are faces of $\A$ with $F'\subseteq\overline{F}\subseteq\overline{C(F')}$ then $C(F')=C(F)$ by the uniqueness assertion of Lemma \ref{closest_chamber}. In this case, the restriction map $t^F_{F'}:\F^I_F\to\F^I_{F'}$ is obtained from $r^F_{F'}:\F_F\to\F_{F'}$ by passage to the invariants under $I_{C(F)}=I_{C(F')}$. Since $r^F_{F'}$ is $(P_F^\dagger\cap P_{F'}^\dagger)$-equivariant it follows once more from Frobenius reciprocity that $t^F_{F'}$ is linear with respect to $H_F^\dagger\cap H_{F'}^\dagger=R[I_{C(F)}\backslash(P_F^\dagger\cap P_{F'}^\dagger)/I_{C(F)}]$.\\

Let $V\in\Rep_R^\infty(G)$ and let $\F_V\in\Coeff_G^0(\BT)$ denote the associated fixed point system (cf.\ Example \ref{fixed_point_system}). If $F'$ and $F$ are faces of $\A$ with $F'\subseteq\overline{F}$ and $C(F')=C(F)$ then the restriction map $t^F_{F'}$ of $\F_V^I\in\Coeff(\A)$ is the identity on $(\F_V^I)_F=V^{I_{C(F)}}=V^{I_{C(F')}}=(\F_V^I)_{F'}$. As was observed in \cite{OS1}, Theorem 3.4, this leads to the fact that the complexes $\C_c^\ori(\BT_{(\bullet)},\F_V)^I\cong\C_c^\ori(\A_{(\bullet)},\F_V^I)$ are acyclic, i.e.\ have trivial homology in positive degrees. This can be proven in the following slightly more general situation.
\begin{prop}\label{acyclic}
Let $\F\in\Coeff_G^0(\BT)$ and assume that the restriction maps $t^F_{F'}:\F_F^I\to\F_{F'}^I$ of the coefficient system $\F^I\in\Coeff(\A)$ are bijective for all faces $F'$ and $F$ of $\A$ with $F'\subseteq\overline{F}$ and $C(F')=C(F)$. Then the following is true.
\begin{enumerate}[wide]
\item[(i)]The complexes $\C_c^\ori(\BT_{(\bullet)},\F)^I$ and $\C_c^\ori(\A_{(\bullet)},\F^I)$ are acyclic.
\item[(ii)]For all faces $F\subseteq\Cbar$ and all vertices $x\in\overline{F}$ the map
\[
\iota_F:\F_F^I\stackrel{t^F_x}{\longrightarrow}\F_x^I\hookrightarrow\bigoplus_{y\in\BT_0}\F_y^I\cong\C_c^\ori(\BT_{(0)},\F)^I\to \H_0(\C_c^\ori(\BT,\F)^I)
\]
is an isomorphism of $H_F^\dagger$-modules which is independent of the choice of $x$.
\item[(iii)]If $F'$ and $F$ are faces of $\A$ with $F'\subseteq\overline{F}\subseteq\Cbar$ then $\iota_F=\iota_{F'}\circ t^F_{F'}$.
\end{enumerate}
\end{prop}
\begin{proof}
As for (i), by Proposition \ref{restriction_chain_complexes} it suffices to prove that the complex $\C_c^\ori(\A_{(\bullet)},\F^I)$ is acyclic. The proof is almost identical to that of \cite{OS1}, Theorem 3.4. We just indicate the few formal changes that have to be made.\\

For any $n\in\mathbb{N}$ denote by $\A(n)$ the subcomplex of all faces $F$ of $\A$ such that $C(F)$ and $C$ have gallery distance less than or equal to $n$. By assumption, $(t^C_F)_{F\subseteq\Cbar}:\K_{\F_C^I}\to\F^I|_\Cbar$ is an isomorphism from the constant coefficient system $\K_{\F_C^I}$ on $\Cbar=\A(0)$ to $\F^I|_\Cbar$. Since $\Cbar$ is contractible the complex $\C_c^\ori(\Cbar_{(\bullet)},\F^I|_\Cbar)$ is acyclic and the map
\[
\F_F^I\stackrel{t^F_x}{\longrightarrow}\F_x^I\hookrightarrow\bigoplus_{y\in\Cbar_0}\F_y^I\cong\C_c^\ori(\Cbar_{(0)},\F^I)\to H_0(\C_c^\ori(\Cbar_{(\bullet)},\F^I|_\Cbar))
\]
is bijective. In order to complete the proof one shows inductively that the complexes $\C_c^\ori(\A(n)_{(\bullet)},\F^I|_{\A(n)})/\C_c^\ori(\A(n-1)_{(\bullet)},\F^I|_{\A(n-1)})$ are exact for any $n\geq 1$. Given a chamber $D$ of distance $n$ from $C$ set $\sigma\overline{D}=\overline{D}\cup\A(n-1)$. Again, the restriction of $\F^I$ to $\overline{D}\backslash\A(n-1)$ is isomorphic to the constant coefficient system with value $\F_D^I$. Therefore, one can proceed as in the proof of \cite{OS1}, Theorem 3.4.\\

As for (ii), let $y$ be another vertex of $\A$ contained in $\overline{F}$. Since $F$ is a polysimplex there is a sequence of vertices $x=x_0,x_1,\ldots,x_n=y$ of $\A$ such that $x_i$ and $x_{i+1}$ lie in a common one dimensional face contained in $\overline{F}$ for all $0\leq i<n$. By induction on $n$ and the transitivity of the restriction maps of $\F^I$ we may assume that $F$ is one dimensional with vertices $x$ and $y$. Let $m\in\F_F^I$ and let $c$ denote the orientation of $F$ inducing the trivial orientation on $x$. Note that this implies $\partial^F_{y}(-c)=1$. Define the oriented $1$-chain $f_m$ on $\A$ by
\[
f_m(F',c')=\left\{
\begin{array}{r@{\quad}l}
m,&\mbox{if }(F',c')=(F,c),\\
-m,&\mbox{if }(F',c')=(F,-c),\\
0,&\mbox{else.}
\end{array}
\right.
\]
Then $\partial_0(f_m)(x)=t^F_x(m)$, $\partial_0(f_m)(y)=-t^F_{y}(m)$ and $\partial_0(f_m)(z)=0$ for all vertices $z$ of $\A$ distinct from $x$ and $y$. This shows that $t^F_x(m)-t^F_{y}(m)=\partial_0(f_m)$ in $\bigoplus_{z\in\A_0}\F_z^I$, hence maps to zero in $H_0(\A,\F^I)=\mathit{coker}(\partial_0)$. This proves the independence of $\iota_F$ from the choice of $x$ and we get (iii) as an immediate consequence.\\

That $\iota_F$ is bijective was shown in the course of the proof of (i). In order to see that it is $H_F^\dagger$-linear note that the map $\F_F\stackrel{r^F_x}{\longrightarrow}\F_x\longrightarrow\C_c^\ori(\BT_{(0)},\F)$ is $P_F$-equivariant because $P_F\subseteq P_F^\dagger\cap P_x^\dagger$ by (\ref{inclusions}). Passing to $I$-invariants we conclude that $\iota_F$ is $H_F$-linear. According to Remark \ref{group_ring} it remains to show that $\iota_F$ is $\tilde{\Omega}_F$-equivariant. More generally, if $\omega\in\tilde{\Omega}$ and $m\in\F_F^I$ consider the commutative diagram
\begin{equation}\label{iota_equivariant}
\begin{gathered}
 \xymatrix{
 \F_F^I\ar[r]^{t^F_x}\ar[d]^{c_{\omega,F}}&\F_x^I\ar@{^{(}->}[r]\ar[d]^{c_{\omega,x}}& \C_c^\ori(\BT_{(0)},\F)^I\ar[d]^{\omega=\tau_\omega} \ar[r] & \H_0(\C_c^\ori(\BT,\F)^I) \ar[d]^{\omega=\tau_\omega}\\
 \F_{\omega F}^I\ar[r]_{t^{\omega F}_{\omega x}}&\F_{\omega x}^I\ar@{^{(}->}[r]&\C_c^\ori(\BT_{(0)},\F)^I \ar[r] & \H_0(\C_c^\ori(\BT,\F)^I).
 }
\end{gathered}
\end{equation}
Since the upper horizontal composition is $\iota_F$ and since the lower one is $\iota_{\omega F}$ we obtain $\iota_{\omega F}(c_{\omega,F}(m))=\omega\cdot\iota_F(m)$. In case $\omega\in\tilde{\Omega}_F$ this shows the required equivariance property of $\iota_F$.
\end{proof}
If $p$ is invertible in $R$ and if $\F\in\Coeff_G^0(\BT)$ then the condition imposed on $\F$ in Proposition \ref{acyclic} admits the following characterization .
\begin{lem}\label{acyclic_locally_constant}
Let $\F\in\Coeff_G^0(\BT)$, let $F$ be a face of $\BT$ and let $\gamma\in P_F^\dagger$. If $\tau_\gamma\in H_F^\dagger$ denotes the characteristic function of the double coset $I_{C(F)}\gamma I_{C(F)}$ then the diagram
\begin{equation}\label{restriction_Hecke_operator}
\begin{gathered}
 \xymatrix{
 \F_{C(F)} \ar[rr]^{t^{C(F)}_F} \ar[d]_{c_{\gamma,C(F)}} && \F_F^{I_{C(F)}} \ar[d]^{\tau_\gamma} \\
 \F_{\gamma C(F)} \ar[rr]_{t^{\gamma C(F)}_F} && \F_F^{I_{C(F)}}
 }
\end{gathered}
\end{equation}
commutes. If $p$ is invertible in $R$ then $\F$ satisfies the hypotheses of Proposition \ref{acyclic} if and only if the coefficient system $\F^I\in\Coeff(\A)$ is locally constant in the sense of Definition \ref{local_system}.
\end{lem}
\begin{proof}
Letting $D=\gamma C(F)$ the action of $\tau_\gamma$ on $\F_F^{I_{C(F)}}$ is given by
\[
 \tau_\gamma(m)=\sum_{g\in I_{C(F)}/(I_{C(F)}\cap \gamma I_{C(F)}\gamma^{-1})}g\gamma m =
 \sum_{g\in I_{C(F)}/(I_{C(F)}\cap I_D)}g\gamma m.
\]
By Proposition \ref{closest_chamber} we have $I_{C(F)}=I_F(I\cap P_F)$ and $I_D=I_D(I\cap P_D)$ with $I_F\subseteq I_D$. Thus, $I_{C(F)}\cap I_D=I_F(I_D\cap I\cap P_F)=I_F(I\cap I_D)$ because $I_D\subseteq P_F$. On the other hand, the equality $I_D=I_D(I\cap P_D)$ implies $I\cap P_D\subseteq I_D\cap I$ whence $I\cap I_D=I\cap P_D$. Altogether, $I_{C(F)}\cap I_D=I_F(I\cap P_D)$. Therefore, the inclusion $I\cap P_F\subseteq I_{C(F)}$ induces a bijection $(I\cap P_F)/(I\cap P_D)\cong I_{C(F)}/(I_{C(F)}\cap I_D)$. This proves the commutativity of (\ref{restriction_Hecke_operator}).\\

Assume that $\F$ satisfies the hypotheses of Proposition \ref{acyclic} and that $p$ is invertible in $R$. In order to see that $\F^I$ is locally constant let $D$ be an arbitrary chamber of $\A$ with $F\subseteq\overline{D}$. By the transitivity of restriction it suffices to see that $t^D_F$ is bijective. By the proof of Lemma \ref{face_representatives} there are elements $g,h\in N_G(T)\cap G_\aff$ with $gD=hC(F)=C$. Moreover, the uniqueness assertion in Lemma \ref{face_representatives} implies that $\gamma=g^{-1}h\in P_F^\dagger$. Since the maps $c_{\gamma,F}$ and $t^{C(F)}_F$ in (\ref{restriction_Hecke_operator}) are bijective it suffices to see that $\tau_\gamma$ induces a bijective endomorphism of  $\F_F^{I_{C(F)}}$. However, if $p$ is invertible in $R$ then $\tau_\gamma$ is a unit in $H_F^\dagger$. Indeed, pulling back along the isomorphism (\ref{phi_gd}) this follows from the fact that $\tau_w$ is a unit in $H$ for any $w\in\tilde{W}$ (cf.\ \cite{Vig2}, Corollary 1).
\end{proof}
Assume that $p$ is invertible in $R$ and that $\F\in\Coeff_G^0(\BT)$ satisfies the hypotheses of Proposition \ref{acyclic}. It follows from Remark \ref{criterion_loc_constant} and Lemma \ref{acyclic_locally_constant} that $\F^I\in\Coeff(\A)$ is isomorphic to a constant coefficient system. The acyclicity result in Proposition \ref{acyclic} (i) is then obvious because $\A$ is contractible. This situation was considered by Broussous in \cite{Bro}, page 746.


\section{The equivalence of categories}\label{section_3}


\subsection{Representations and Hecke modules of stabilizer groups}\label{subsection_3_1}

Throughout this subsection we fix an arbitrary face $F$ of $\BT$ and let the chamber $C(F)$ be as in Lemma \ref{closest_chamber}. In the case of finite dimensional representations over a field, the following condition was first introduced by Cabanes (cf.\ \cite{Cab}, Definition 1).
\begin{defin}\label{condition_H}
\begin{enumerate}[wide]
\item[(i)]We say that an object $V\in\Rep_R^\infty(P_F)$ satisfies condition (H) if $V\cong\varinjlim_{j\in J}V_j$ is isomorphic to the inductive limit of objects $V_j\in\Rep_R^\infty(P_F)$ such that the transition maps of the inductive system are injective and such that for each $j\in J$ there is a non-negative integer $n_j$ and an element $\varphi_j\in\End_{P_F}(X_F^{n_j})$ with $V_j\cong\im(\varphi_j)$ in $\Rep_R^\infty(P_F)$. We denote by $\Rep_R^{H}(P_F)$ the full subcategory of $\Rep_R^\infty(P_F)$ consisting of all representations satisfying condition (H).\\[-3ex]
\item[(ii)]We say that an object $V\in\Rep_R^\infty(P_F^\dagger)$ satisfies condition (H) if it is an object of $\Rep_R^H(P_F)$ when viewed as a $P_F$-representation via restriction. We denote by $\Rep_R^H(P_F^\dagger)$ the full subcategory of $\Rep_R^\infty(P_F^\dagger)$ consisting of all representations satisfying condition (H).
\end{enumerate}
\end{defin}
The notation (H) is supposed to reflect the close relation to the respective categories of Hecke modules described below.
\begin{rem}\label{H_I_F_trivial}
Apparently, the condition $V_j\cong\im(\varphi_j)$ in Definition \ref{condition_H} (i) is equivalent to $V_j$ being both a quotient and a submodule of a finite direct sum of copies of $X_F$. It follows that the action of $I_F$ is trivial on any smooth representation satisfying condition (H). If a representation satisfies condition (H) and if its underlying $R$-module is noetherian then there is an isomorphism of $P_F$-representations $V\cong\im(\varphi)$ for some non-negative integer $n$ and some element $\varphi\in\End_{P_F}(X_F^n)$. The categories $\Rep_R^H(P_F)$ and $\Rep_R^H(P_F^\dagger)$ are closed under arbitrary direct sums and inductive limits with injective transition maps.
\end{rem}
Condition (H) is preserved under induction in the following sense.
\begin{lem}\label{induction_H}\begin{enumerate}[wide]
\item[(i)]If $V\in\Rep_R^H(P_F)$ then $\ind_{P_F}^{P_F^\dagger}(V)\in\Rep_R^H(P_F^\dagger)$.
\item[(ii)]Let $F'$ be a face of $\BT$ with $F'\subseteq\overline{F}$ and $C(F')=C(F)$. If $V\in\Rep_R^H(P_F)$ then $\ind_{P_F}^{P_{F'}}(V)\in\Rep_R^H(P_{F'})$.
\end{enumerate}
\end{lem}
\begin{proof}
As for (i), note that compact induction preserves inductive limits with injective transition maps. Thus, we may assume $V=\im(\varphi)$ for some $\varphi\in\End_{P_F}(X_F^n)$. Setting $\psi=\ind_{P_F}^{P_F^\dagger}(\varphi)$ this yields $\ind_{P_F}^{P_F^\dagger}(V)=\im(\psi)$ where $\psi$ is an endomorphism of the $P_F^\dagger$-representation $\ind_{P_F}^{P_F^\dagger}(X_F^n)\cong (X_F^\dagger)^n$. As a $P_F$-representation $(X_F^\dagger)^n\cong\oplus_{j\in J}X_F$ is a direct sum of copies of $X_F$ (cf.\ Proposition \ref{I_F_invariants} (i)). If $J'\subseteq J$ is finite then there is $J''\subseteq J$ finite with $\psi(\oplus_{j\in J'}X_F)\subseteq\oplus_{j\in J''}X_F$ because the $P_F$-representation $X_F$ is finitely generated. Thus, $\psi(\oplus_{j\in J'}X_F)$ and $\im(\psi)$ satisfy condition (H). Part (ii) follows from $\ind_{P_F}^{P_{F'}}(X_F)=X_{F'}$ and the exactness of compact induction.
\end{proof}
We say that a representation $V\in\Rep_R^\infty(P_F)$ (resp.\ $V\in\Rep_R^\infty(P_F^\dagger))$ is generated by its $I_{C(F)}$-invariants if
\[
R[P_F]\cdot V^{I_{C(F)}}=V\quad(\mbox{resp.\ } R[P_F^\dagger]\cdot V^{I_{C(F)}}=V).
\]
Clearly, $X_F$ and $X_F^\dagger$ have this property. Like condition (H) it is insensitive to restriction.
\begin{lem}\label{invariants_restriction}
If $V\in\Rep_R^\infty(P_F^\dagger)$ then $R[P_F^\dagger]\cdot V^{I_{C(F)}}=R[P_F]\cdot V^{I_{C(F)}}$.
\end{lem}
\begin{proof}
By conjugation and transport of structure we may assume $F\subseteq\Cbar$. But then $P_F^\dagger/P_F\cong\Omega_F$ and $\omega I\omega^{-1}=I$ for any element $\omega\in N_G(T)\cap P_F^\dagger$ whose image in $W$ lies in $\Omega_F$. Since any such element stabilizes $V^I$, the claim follows.
\end{proof}
We now clarify the relation between condition (H) and the condition ($A$+$A^*$) of \cite{Cab}, Proposition 8. In the case of a finitely generated $R$-module the latter means that $V$ and its $R$-linear contragredient $V^*$ are both generated by their $I_{C(F)}$-invariants. Recall that a ring is called quasi-Frobenius if it is noetherian and selfinjective (cf.\ \cite{Lam}, \S15).
\begin{lem}\label{properties_H}Let $V$ be a smooth $R$-linear representation of $P_F$ or of $P_F^\dagger$.
\begin{enumerate}[wide]
\item[(i)]If $V$ satisfies condition (H) then it is generated by its $I_{C(F)}$-invariants.
\item[(ii)]Assume that $R$ is a quasi-Frobenius ring and that the underlying $R$-module of $V$ is finitely generated. Then $V$ satisfies condition (H) if and only if $V$ and its contragredient $V^*=\Hom_R(V,R)$ are generated by their $I_{C(F)}$-invariants.
\item[(iii)]Assume that $R$ is a quasi-Frobenius ring and that the order of the finite group $P_F/I_F$ is invertible in $R$. If $V$ is generated by its $I_{C(F)}$-invariants and if the underlying $R$-module of $V$ is finitely generated projective then also $V^*$ is generated by its $I_{C(F)}$-invariants. In particular, $V$ satisfies condition (H).
\end{enumerate}
\end{lem}
\begin{proof}
In part (i) we may assume $V=\im(\varphi)$ for some $\varphi\in\End_{P_F}(X_F^n)$. But then $V$ is a quotient of $X_F$ hence is generated by its $I_{C(F)}$-invariants.\\

As for (ii) we can write $V=\im(\varphi)$ as above. Dualizing the embedding $V=\im(\varphi)\hookrightarrow X_F^n$ yields a surjection $(X_F^*)^n\to V^*$ because $R$ is selfinjective. Note that $X_F\cong X_F^*$ in $\Rep_R^\infty(P_F)$. Indeed, for $p\in P_F$ let $\psi_p\in X_F^*$ denote the element determined by $\psi_p(f)=f(p)$. Then the map $(f\mapsto\sum_{p\in P_F/I_{C(F)}}f(p)\psi_p):X_F\to X_F^*$ is an isomorphism in $\Rep_R^\infty(P_F)$. Consequently, also $V^*$ is generated by its $I_{C(F)}$-invariants.\\

Conversely, assume that $V$ and $V^*$ are generated by their $I_{C(F)}$-invariants. Since $V$ is finitely generated and $R$ is noetherian also $V^{I_{C(F)}}$ is a finitely generated $R$-module. Choose a non-negative integer $n$ and an $R$-linear surjection $R^n\to V^{I_{C(F)}}$ of trivial $I_{C(F)}$-representations. Applying $\ind_{I_{C(F)}}^{P_F}$ and composing with the natural map $\ind_{I_{C(F)}}^{P_F}(V^{I_{C(F)}})\to V$ we obtain a homomorphism $X_F^n\to V$ in $\Rep_R^\infty(P_F)$ which is surjective because $V$ is generated by its $I_{C(F)}$-invariants. In a similar manner, one constructs a surjection $X_F^m\to V^*$. Passing to the $R$-linear dual, we obtain an injective homomorphism $V\cong (V^*)^*\hookrightarrow X_F^m$ because over a selfinjective ring any finitely generated module is reflexive (cf.\ \cite{Lam}, Theorem 15.11). Thus, $V$ satisfies condition (H).\\

As for (iii), we construct a surjection $X_F^n\twoheadrightarrow V$ in $\Rep_R^\infty(P_F)$ as above. Passing to the $R$-linear dual, there is an injection $V^*\hookrightarrow(X_F^*)^n\cong X_F^n$ in $\Rep_R^\infty(P_F)$. It admits an $R$-linear section because over a quasi-Frobenius ring the classes of projective and injective modules coincide (cf.\ \cite{Lam}, Theorem 15.9). By our assumption, the embedding $V^*\hookrightarrow X_F^n$ even admits a $P_F/I_F$-linear section by the usual averaging construction. Consequently, $V^*$ is generated by its $I_{C(F)}$-invariants and satisfies (H) by (ii).
\end{proof}
Over fields, the following fundamental results are due to Sawada, Tinberg, Schneider and Ollivier, respectively (cf.\ \cite{Saw}, Theorem 2.4, \cite{Tin}, Proposition 3.7 and \cite{OS1}, Proposition 5.5). However, the proofs work more generally.
\begin{prop}\label{selfinjective}
\begin{enumerate}[wide]
\item[(i)]The ring homomorphism $R\to H_F$ is an $\mathrm{id}_R$-Frobe\-nius extension. In particular, the rings $R$ and $H_F$ have the same injective dimension. If $R$ is a quasi-Frobenius ring then so is $H_F$.
\item[(ii)]
Assume that $\GG$ is semisimple. Then the ring homomorphism $R\to H_F^\dagger$ is an $\mathrm{id}_R$-Frobenius extension. In particular, the rings $R$ and $H_F^\dagger$ have the same injective dimension and if $R$ is a quasi-Frobenius ring then so is $H_F^\dagger$.
\end{enumerate}
\end{prop}
\begin{proof}
Using the isomorphisms (\ref{phi_gd}) we may assume $F\subseteq\Cbar$. As for (i), the group $P_F$ is compact and the $R$-module $H_F$ is finitely generated and free of rank $|\tilde{W}_F|$. In particular, if $R$ is noetherian then so is $H_F$.\\

By \cite{OS1}, Example 5.1, the ring homomorphism $R\to R[T_0/T_1]$ is an $\id_R$-Frobenius extension. Choose an element $w_0\in \tilde{W}_F$ which is of maximal length and consider the $R$-linear ring automorphism $\alpha$ of $R[T_0/T_1]$ given by $\xi\mapsto w_0\xi w_0^{-1}$. By the transitivity of Frobenius extensions it suffices to see that the canonical injection $R[T_0/T_1]\to H_F$ is an $\alpha$-Frobenius extension. We fix representatives $\tilde{v}\in\tilde{W}_F$ of the elements $v\in W_F\cong\tilde{W}_F/(T_0/T_1)$. The $R[T_0/T_1]$-algebra $H_F$ admits the two bases $(\tau_{\tilde{v}})_{v\in W_F}$ and $(\tau_{{\tilde{w}}^{-1}w_0})_{w\in W_F}$. We define the map $\theta:H_F\to R[T_0/T_1]$ by
\[
 \theta(\sum_{w\in W_F}a_w\tau_{\tilde{w}})=\sum_{\xi\in T_0/T_1}a_{\xi w_0}\tau_{\tilde{\xi}}.
\]
It suffices to see that the matrix $(\theta(\tau_{\tilde{v}}\tau_{{\tilde{w}}^{-1}w_0}))_{v,w\in W_F}$ over $R[T_0/T_1]$ is invertible (cf.\ \cite{OS1}, Lemma 5.2 and Lemma 5.3). This is done as in \cite{OS1}, Proposition 5.4 (i), relying on the relation (\ref{braid_relations}) and (\ref{quadratic_relations}) of $H$ which hold over any coefficient ring $R$ (cf.\ \cite{Vig2}, Theorem 1). Alternatively, it suffices to see that the above matrix is invertible after reduction modulo every maximal ideal of the commutative ring $R$. However, this leads to the case of a field which is treated in \cite{Tin}, Proposition 3.7. If $\GG$ is semisimple then also $P_F^\dagger$ is compact. The arguments for $H_F^\dagger$ are then similar on replacing $W_F$ by the finite group $W_F^\dagger$.
\end{proof}
\begin{rem}\label{not_selfinjective}
If $\GG$ is not semisimple then $H_F^\dagger$ is not selfinjective unless $R=0$. In fact, as in \cite{OS1}, Proposition 5.5, its selfinjective dimension is equal to the rank of the center $\CC$ of $\GG$ for any non-zero quasi-Frobenius ring $R$.
\end{rem}
The following theorem is essentially due to Cabanes (cf.\ \cite{Cab}, Theorem 2). We follow his arguments over arbitrary quasi-Frobenius rings and also treat representations whose underlying $R$-modules are not finitely generated.
\begin{thm}\label{Cabanes}Assume that $R$ is a quasi-Frobenius ring. The functor of $I_{C(F)}$-invariants restricts to an equivalence $\Rep_R^{H}(P_F)\longrightarrow\Mod_{H_F}$ of additive categories.
\end{thm}
\begin{proof}
As before, we may assume $F\subseteq\Cbar$. Denote by $\Rep_R^H(P_F)^\fg$ the full subcategory of $\Rep_R^H(P_F)$ consisting of all objects whose underlying $R$-modules are finitely generated. Moreover, denote by $\Mod_{H_F}^\fg$ the category of finitely generated $H_F$-modules. In a first step we show that the functor $(\cdot)^I$ induces an equivalence of categories $\Rep_R^H(P_F)^\fg\to\Mod_{H_F}^\fg$. Note that an $H_F$-module is finitely generated if and only if its underlying $R$-module is finitely generated because $H_F$ is finitely generated over $R$.\\

To prove the essential surjectivity, let $M$ be a finitely generated $H_F$-module. Since $H_F$ is quasi-Frobenius (cf.\ Proposition \ref{selfinjective} (i)) there is a non-negative integer $n$ and an $H_F$-linear embedding $M\hookrightarrow H_F^n$ (cf.\ \cite{Lam}, Theorem 15.11). Since $H_F^n=(X_F^n)^I$ we may set $V=R[P_F]\cdot M$ viewed as a subrepresentation of $X_F^n$. We claim that $V^I=M$ where $M\subseteq V^I$ is true by definition.\\

Since $V^I\subseteq H_F^n$, the $H_F$-module $V^I/M$ is finitely generated. As above, there is a non-negative integer $m$ and an $H_F$-linear embedding $\overline{g}:V^I/M\hookrightarrow H_F^m$. The $H_F$-linear map $V^I\twoheadrightarrow V^I/M\stackrel{\overline{g}}{\to}H_F^m$ extends to an $H_F$-linear map $g:H_F^n\to H_F^m$ by the selfinjectivity of $H_F$. Note that the functor $(\cdot)^I$ induces an isomorphism $\Hom_{P_F}(X_F^n,X_F^m)\longrightarrow\Hom_{H_F}(H_F^n,H_F^m)$. Consequently, there is an element $f\in\Hom_{P_F}(X_F^n,X_F^m)$ such that $g=f^I$ is the restriction of $f$ to the $I$-invariants of $X_F^n$. By construction, $f(M)=g(M)=0$ whence $f(V)=0$ because $f$ is $P_F$-equivariant and $M$ generates $V$. This implies $\overline{g}(V^I/M)=g(V^I)=f(V^I)=0$ whence $V^I/M=0$ by the injectivity of $\overline{g}$. Therefore, $V^I=M$ as claimed.\\

By construction, $V$ is a subobject of $X_F^n$. Since $V^I=M$ is a finitely generated $H_F$-module it is also finitely generated over $R$. As in the proof of Lemma \ref{properties_H} (ii) we see that $V$ is also a quotient of $X_F^n$ for $n$ sufficiently large. Consequently, $V$ satisfies condition (H) and its underlying $R$-module is finitely generated.\\

Now let $V$ and $W$ be arbitrary objects of $\Rep_R^{H}(P_F)^\fg$. Since $V$ is generated by its $I$-invariants (cf.\ Lemma \ref{properties_H} (i)) any non-zero element $f\in\Hom_{P_F}(V,W)$ restricts to a non-zero element $f^I\in\Hom_{H_F}(V^I,W^I)$, i.e.\ the functor $(\cdot)^I:\Rep_R^{H}(P_F)^\fg\to\Mod_{H_F}^\fg$ is faithful.\\

In order to see that it is full, let $g\in\Hom_{H_F}(V^I,W^I)$. We may assume $V=\im(\varphi)$ and $W=\im(\psi)$ for some non-negative integer $n$ and elements $\varphi,\psi\in\End_{P_F}(X_F^n)$. In particular, this realizes $V,W$ as $P_F$-subrepresentations of $X_F^n$ and $V^I,W^I$ as $H_F$-submodules of $H_F^n$. By the selfinjectivity of $H_F$, the $H_F$-linear map $V^I\to W^I\hookrightarrow H_F^n$ extends to an $H_F$-linear endomorphism of $H_F^n$. As above, the latter is the restriction of an element $f\in\End_{P_F}(X_F^n)$ to the space of $I$-invariants. Since $V$ and $W$ are generated by their $I$-invariants (cf.\ Lemma \ref{properties_H} (i)) we have
\begin{eqnarray*}
f(V)&=&f(R[P_F]\cdot V^I)=R[P_F]\cdot f(V^I)=R[P_F]\cdot g(V^I)\\&\subseteq& R[P_F]\cdot W^I=W,
\end{eqnarray*}
i.e.\ $f$ restricts to an element of $\Hom_{P_F}(V,W)$ with $f^I=g$. This establishes the equivalence of categories $(\cdot)^I:\Rep_R^{H}(P)^\fg\to\Mod_{H_F}^\fg$.\\

Before going on, note that the image of the homomorphism $f$ constructed above is equal to the $P_F$-subrepresentation of $X_F^n$ generated by $f(V^I)=g(V^I)$. By the first part of our proof this is a representation satisfying condition (H) with $\im(f)^I=g(V^I)$. Now if $g$ happens to be injective then $f:V\to\im(f)$ is a homomorphism in $\Rep_R^{H}(P_F)^\fg$ such that the induced homomorphism on $I$-invariants $g:V^I\to\im(f)^I=g(V^I)$ is an isomorphism. Therefore, $f:V\to\im(f)$ is an isomorphism itself because of our equivalence of categories. Thus, $f:V\to W$ is injective, too.\\

Let us now treat the general case. Given $V,W\in\Rep_R^{H}(P_F)$ write $V=\varinjlim_{j\in J}V_j$ and $W=\varinjlim_{j'\in J'}W_{j'}$ as in Definition \ref{condition_H} (i). Since the underlying $R$-modules of $V_j$ and $V_j^I$ are finitely generated and since the transition maps in the inductive systems $(W_{j'})_{j'}$ and $(W_{j'}^I)_{j'}$ are injective the natural maps
\begin{eqnarray*}
\Hom_{P_F}(\varinjlim_{j\in J}V_j,\varinjlim_{j'\in J'}W_{j'}) & \longrightarrow & \varprojlim_{j\in J}\varinjlim_{j'\in J'}\Hom_{P_F}(V_j,W_{j'})\quad\mbox{and}\\
\Hom_{H_F}(\varinjlim_{j\in J}V_j^I,\varinjlim_{j'\in J'}W_{j'}^I) & \longrightarrow & \varprojlim_{j\in J}\varinjlim_{j'\in J'}\Hom_{H_F}(V_j^I,W_{j'}^I)
\end{eqnarray*}
are bijective. Since the natural maps $\varinjlim_{j\in J}V_j^I\to V^I$ and $\varinjlim_{j'\in J'}W_{j'}^I\to W^I$ are isomorphisms of $H_F$-modules, the functor $(\cdot)^I$ is fully faithful in general.\\

Finally, let $M\in\Mod_{H_F}$ and write $M=\varinjlim_{j\in J}M_j$ as the filtered union of a familiy of finitely generated $H_F$-modules $M_j$. By what we have already proven there are objects $V_j\in\Rep_R^{H}(P_F)^\fg$ with $V_j^I=M_j$ for all $j\in J$. For $j\leq j'$ the map $\Hom_{P_F}(V_j,V_{j'})\to\Hom_{H_F}(M_j,M_{j'})$ is bijective and we let $\varphi_{jj'}:V_j\to V_{j'}$ denote the homomorphism corresponding to the inclusion $M_j\hookrightarrow M_{j'}$. As was noted above, the map $\varphi_{jj'}$ is automatically injective. Moreover, the bijectivity implies that the family $(V_j,\varphi_{jj'})_{j\leq j'}$ is an inductive system. Setting $V=\varinjlim_{j\in J}V_j\in\Rep_R^{H}(P_F)$ we have $V^I\cong\varinjlim_{j\in J}V_j^I=\varinjlim_{j\in J}M_j=M$.
\end{proof}
The essential surjectivity in Theorem \ref{Cabanes} was proved by an inductive limit procedure. This can be avoided through the following construction.
\begin{prop}\label{quasi_inverse}
Assume that $R$ is a quasi-Frobenius ring. If $M\in\Mod_{H_F}$ and if $E$ is an injective $H_F$-module containing $M$ then the $P_F$-subrepresen\-tation $V=\im(X_F\otimes_{H_F}M\to X_F\otimes_{H_F}E)$ of $X_F\otimes_{H_F}E$ generated by the image of the natural map $M\hookrightarrow E\hookrightarrow X_F\otimes_{H_F}E$ satisfies condition (H) and $V^{I_{C(F)}}=M$.
\end{prop}
\begin{proof}
Note first that $E$ is a projective $H_F$-module because $H_F$ is a quasi-Frobenius ring (cf.\ Lemma \ref{selfinjective} and \cite{Lam}, Theorem 15.9). Therefore, the map $E\to X_F\otimes_{H_F}E$ is injective and $(X_F\otimes_{H_F}E)^{I_{C(F)}}=X_F^{I_{C(F)}}\otimes_{H_F}E=E$.\\

Let $E'$ be a free $H_F$-module containing $E$ as a direct summand and let $M_j$ be a finitely generated $H_F$-submodule of $M$. Then $M_j$ is contained in a finitely generated free direct summand $H_F^n$ of $E'$ and the $P_F$-subrepresentation $V_j=R[P_F]\cdot M_j$ of $V$ is contained in $X_F\otimes_{H_F}H_F^n\cong X_F^n$. By the proof of Theorem \ref{Cabanes}, the $P_F$-representation $V_j$ satisfies condition (H) with $V_j^{I_{C(F)}}=M_j$.
\end{proof}
The previous construction allows us to prove an analog of Theorem \ref{Cabanes} for the pair $(P_F^\dagger,H_F^\dagger)$. Note that the strategy of Cabanes does not apply directly if $\GG$ is not semisimple (cf.\ Remark \ref{not_selfinjective}). Instead, one has to reduce to the situation in Theorem \ref{Cabanes}.
\begin{thm}\label{Cabanes_general}
Assume that $R$ is a quasi-Frobenius ring.
\begin{enumerate}[wide]
\item[(i)]If $M\in\Mod_{H_F^\dagger}$ and if $E$ is an injective $H_F^\dagger$-module containing $M$ then the $P_F^\dagger$-subrepresentation $V=\im(X_F^\dagger\otimes_{H_F^\dagger}M\to X_F^\dagger\otimes_{H_F^\dagger}E)$ of $X_F^\dagger\otimes_{H_F^\dagger}E$ generated by the image of the natural map $M\hookrightarrow E\hookrightarrow X_F^\dagger\otimes_{H_F^\dagger}E$ satisfies condition (H) and $V^{I_{C(F)}}=M$.
\item[(ii)]The functor $(\cdot)^{I_{C(F)}}$ restricts to an equivalence $\Rep_R^{H}(P_F^\dagger)\longrightarrow\Mod_{H_F^\dagger}$ of additive categories.
\end{enumerate}
\end{thm}
\begin{proof}
As for part (i), note that $E$ is an injective $H_F$-module by restriction of scalars because $H_F^\dagger$ is free over $H_F$ (cf.\ Proposition \ref{free} (i) and \cite{Lam}, Corollary 3.6A). Moreover, for any $H_F^\dagger$-module $N$ there is a natural $P_F$-equivariant bijection $X_F^\dagger\otimes_{H_F^\dagger}N\cong X_F\otimes_{H_F}N$ because of Proposition \ref{I_F_invariants} (i). The statements in (i) therefore follow from Proposition \ref{quasi_inverse}.\\

In (ii) it remains to see that the functor is fully faithful. That it is faithful is again a consequence of Lemma \ref{properties_H} (i). We assume once more that $F\subseteq\Cbar$. Let $V,W\in\Rep_R^H(P_F^\dagger)$ and let $g:V^I\to W^I$ be $H_F^\dagger$-linear. By Theorem \ref{Cabanes} there is an $R$-linear $P_F$-equivariant map $f:V\to W$ with $f^I=g$. We claim that it is $P_F^\dagger$-equivariant. Let $v\in V$ and $g\in P_F^\dagger$. Since $R[P_F]\cdot V^I=V$ we may assume $v=hw$ with $h\in P_F$ and $w\in V^I$. By (\ref{Bruhat_parahoric}) there are elements $q\in P_F$ and $\omega\in N_G(T)\cap P_F^\dagger$ such that $gh=q\omega$ and such that the image of $\omega$ in $W$ lies in $\Omega_F$. Note that this gives $\omega w=\tau_{\omega}w$ by Remark \ref{group_ring}. Therefore, we obtain $f(gv)=f(ghw)=f(q\omega w)=qf(\tau_\omega w)=qg(\tau_\omega w)=q\tau_\omega g(w)=q\omega f(w)=ghf(w)=gf(hw)=gf(v)$.
\end{proof}
Although the categories $\Mod_{H_F}$ and $\Mod_{H_F^\dagger}$ are abelian, the categories $\Rep_R^{H}(P_F)$ and $\Rep_R^{H}(P_F^\dagger)$ are generally not. This has to do with the failure of the exactness of the functor $(\cdot)^{I_{C(F)}}$. However, some of the abelian structure of the module categories is visible in the respective categories of representations.
\begin{cor}\label{injective_surjective}
Assume that $R$ is a quasi-Frobenius ring. Let $P\in\{P_F,P_F^\dagger\}$ and write $S=H_F$ if $P=P_F$ and $S=H_F^\dagger$ if $P=P_F^\dagger$. Moreover, let $V,W\in\Rep_R^\infty(P)$ and $f\in\Hom_P(V,W)$.
\begin{enumerate}[wide]
\item[(i)]Assume that $V$ satisfies condition (H). If $M$ is an $S$-submodule of $V^{I_{C(F)}}$ then the $P$-subrepresentation of $V$ generated by $M$ satisfies condition (H) and has $I_{C(F)}$-invariants $M$. In particular, any subrepresentation of $V$ which is generated by its $I_{C(F)}$-invariants also satisfies condition (H).
\item[(ii)]If $V$ and $W$ satisfy condition (H) then so does $\im(f)$.
\item[(iii)]Assume that $V$ and $W$ satisfy condition (H). Then $f$ is injective if and only if $f^{I_{C(F)}}:V^{I_{C(F)}}\to W^{I_{C(F)}}$ is injective.
\item[(iv)]Assume that $V$ and $W$ satisfy condition (H). Then $f$ is surjective if and only if $f^{I_{C(F)}}:V^{I_{C(F)}}\to W^{I_{C(F)}}$ is surjective.
\end{enumerate}
\end{cor}
\begin{proof}
As for (i), the arguments in Lemma \ref{invariants_restriction} allow us to assume $P=P_F$ and $S=H_F$. Choose an $H_F$-linear embedding $V^I\hookrightarrow E$ where $E$ is an injective $H_F$-module. Then $V$ is isomorphic to the $P_F$-subrepresentation of $X_F\otimes_{H_F}E$ generated by $V^I$ (cf.\ Theorem \ref{Cabanes} and Proposition \ref{quasi_inverse}). Under this isomorphism the subrepresentation of $V$ generated by $M$ is isomorphic to the subrepresentation of $X_F\otimes_{H_F}E$ generated by $M$. By Proposition \ref{quasi_inverse} this satisfies condition (H) and has $I$-invariants $M$.\\

Part (ii) follows from (i) and Lemma \ref{properties_H} (i). Note that $\im(f)$ is a quotient of $V$ and hence is generated by its $I_{C(F)}$-invariants.\\

As for (iii), the injectivity of $f$ clearly implies the injectivity of its restriction $f^{I_{C(F)}}$. Conversely, assume that $f^{I_{C(F)}}$ is injective. In order to see that $f$ is injective we may replace $V$ and $W$ by suitable $P_F$-subrepresentations whose underlying $R$-modules are finitely generated. This case was treated in the proof of Theorem \ref{Cabanes}.\\

As for (iv), if $f^{I_{C(F)}}$ is surjective then so is $f$ because of Lemma \ref{properties_H} (i). Conversely, if $f$ is surjective then $f(V^{I_{C(F)}})$ is an $S$-submodule of $W^{I_{C(F)}}$ with $R[P]\cdot f(V^{I_{C(F)}})=f(R[P]\cdot V^{I_{C(F)}})=f(V)=W$ by Lemma 3.5 (i) once more. Thus, $f(V^{I_{C(F)}})=W^{I_{C(F)}}$ by (i).
\end{proof}
The construction of a quasi-inverse of the equivalences in Theorem \ref{Cabanes} and Theorem \ref{Cabanes_general} (ii) relies on choices as in Proposition \ref{quasi_inverse} and Theorem \ref{Cabanes_general} (i). However, there is an alternative construction of a quasi-inverse which is closer to our intuitive idea of a functor. This approach is inspired by  \cite{OS2}, \S1.2. We continue to let $P\in\{P_F,P_F^\dagger\}$. If $P=P_F$ we set $S=H_F$ and $Y=X_F$. If $P=P_F^\dagger$ we set $S=H_F^\dagger$ and $Y=X_F^\dagger$. Note that if $M\in\Mod_S$ then $Y\otimes_SM$ and $\Hom_S(\Hom_S(Y,S),M)$ are naturally objects of $\Rep_R^\infty(P)$. For the second case note that the action of the open subgroup $I_F$ is trivial. Further, there is a unique homomorphism
\[
\tau_{M,F}:Y\otimes_SM\longrightarrow \Hom_S(\Hom_S(Y,S),M),
\]
of smooth $R$-linear $P$-representations sending $x\otimes m$ to the $S$-linear map $(\varphi\mapsto\varphi(x)\cdot m):\Hom_S(Y,S)\to M$. Note that if we view the unit element $1\in S$ as an element of $Y$ via $S=Y^{I_{C(F)}}$ then the map
$M\to\im(\tau_{M,F})^{I_{C(F)}}$ given by $m\mapsto\tau_{M,F}(1\otimes m)$ is a homomorphism of $S$-modules.
\begin{thm}\label{quasi_inverse_OS}
Assume that $R$ is a quasi-Frobenius ring. The assignment $\t_F=(M\mapsto\im(\tau_{M,F}))$ is functorial in $M$ and quasi-inverse to the equivalences of Theorem \ref{Cabanes} and Theorem \ref{Cabanes_general} (ii), respectively.
\end{thm}
\begin{proof}
Clearly, the formation of $\t_F(M)$ is functorial in $M$. Let us first treat the case $P=P_F$. Since the functor $\Hom_S(\Hom_S(Y,S),\cdot)$ preserves injections, so does $\t_F$. Moreover, the $S$-module $\Hom_S(Y,S)$ is finitely generated because its underlying $R$-module is contained in $\Hom_R(Y,S)$ which is finitely generated and free. Therefore, the functor $\Hom_S(\Hom_S(Y,S),\cdot)$ commutes with filtered unions and so does $\t_F$. As a consequence, it suffices to show that if $M\in\Mod_S^\fg$ then $\t_F(M)$ satisfies condition (H) and $\t_F:\Mod_S^\fg\to\Rep_R^H(P)^\fg$ is quasi-inverse to $(\cdot)^{I_{C(F)}}$. Here $(\cdot)^\fg$ refers to the notation introduced in the proof of Theorem \ref{Cabanes}.\\

If $M$ is a finitely generated $S$-module then there is an embedding $M\hookrightarrow S^n$ into a finitely generated free $S$-module (cf.\ Proposition \ref{selfinjective} (i) and \cite{Lam}, Theorem 15.11). Consider the commutative diagram
\[
\xymatrix{Y\otimes_SM\ar[rr]^>>>>>>>>>>>{\tau_{M,F}}\ar[d] && \Hom_S(\Hom_S(Y,S),M) \ar@{_{(}->}[d] \\
Y\otimes_SS^n \ar[rr]^>>>>>>>>>>{\tau_{S^n,F}} && \Hom_S(\Hom_S(Y,S),S^n).}
\]
We claim that the map $\tau_{S^n,F}$ is bijective. To see this we may assume $n=1$ in which case $\tau_{S,F}$ can be identified with the duality map of the $S$-module $Y$. The latter is bijective because over a quasi-Frobenius ring every finitely generated module is reflexive (cf.\ \cite{Lam}, Theorem 15.11). As a consequence, $\t_F(M)\cong\im(Y\otimes_SM\to Y\otimes_SS^n)$. It follows from Proposition \ref{quasi_inverse} that $\t_F(M)$ satisfies (H) and that the natural map $M\to\t_F(M)^{I_{C(F)}}$ constructed above is bijective.\\

Conversely, if $V\in\Rep_R^H(P)^\fg$ then there is an embedding $V^{I_{C(F)}}\hookrightarrow S^n$ of $S$-modules and the proof of Theorem \ref{Cabanes} shows that inside $Y^n=Y\otimes_SS^n$ we have $V\cong R[P]\cdot V^{I_{C(F)}}=\im(Y\otimes_SV^{I_{C(F)}}\to Y\otimes_SS^n)\cong\t_F(V^{I_{C(F)}})$.\\

Now we treat the case $P=P_F^\dagger$. By Proposition \ref{I_F_invariants} (i) there is a natural isomorphism $X_F^\dagger\otimes_{H_F^\dagger}(\cdot)\cong X_F\otimes_{H_F}(\cdot)$. Together with \cite{BAC}, I.2.9 Proposition 10, this also gives $\Hom_{H_F^\dagger}(X_F^\dagger,H_F^\dagger)\cong\Hom_{H_F}(X_F,H_F)\otimes_{H_F}H_F^\dagger$ as left $H_F$-modules because $X_F$ is finitely generated over $H_F$ and $H_F^\dagger$ is free over $H_F$ (cf.\ Proposition \ref{free} (i)). We thus obtain a natural isomorphism $\Hom_{H_F^\dagger}(\Hom_{H_F^\dagger}(X_F^\dagger,H_F^\dagger),\;\cdot\;)\cong
\Hom_{H_F}(\Hom_{H_F}(X_F,H_F),\;\cdot\;)$. Altogether, up to isomorphism the diagram
\[
\xymatrix{
\Rep_R^H(P_F^\dagger)\ar[r]^{\t_F}\ar[d]_{\res}&\Mod_{H_F^\dagger}\ar[d]^{\res}\\
\Rep_R^H(P_F)\ar[r]^{\t_F}&\Mod_{H_F}
}
\]
commutes if the vertical arrows denote restriction of scalars. Therefore, the case $P=P_F^\dagger$ follows from the case $P=P_F$ already treated.
\end{proof}
We continue to let $P\in\{P_F,P_F^\dagger\}$. If $R\to R'$ is a homomorphism of commutative rings then we have the functor
\[
R'\otimes_R(\cdot):\Rep_R^\infty(P)\to\Rep_{R'}^\infty(P).
\]
Under suitable flatness assumptions it preserves condition (H).
\begin{lem}\label{base_change_H}
\begin{enumerate}[wide]
\item[(i)]If the ring homomorphism $R\to R'$ is flat then the functor $R'\otimes_R(\cdot)$ preserves condition (H).
\item[(ii)]Assume that $R$ is a quasi-Frobenius ring. If $V\in\Rep_R^H(P)$ and if the underlying $R$-module of $V$ is finitely generated and projective then $R'\otimes_RV\in\Rep_{R'}^H(P)$.
\end{enumerate}
\end{lem}
\begin{proof}
As for (i), consider the isomorphism $R'\otimes_R\ind_{I_{C(F)}}^{P_F}(R)\cong \ind_{I_{C(F)}}^{P_F}(R')$ in $\Rep_{R'}^\infty(P_F)$ and note that the functor $R'\otimes_R(\cdot)$ commutes with direct sums, images and filtered unions because $R\to R'$ is flat.\\

As for (ii), write $V\cong\im(\varphi)$ with $\varphi\in\End_{P_F}(X_F^n)$ for some non-negative integer $n$. The underlying $R$-module of $V$ is injective because the ring $R$ is selfinjective (cf.\ \cite{Lam}, Theorem 15.1). Therefore, the injection $V\stackrel{\varphi}{\hookrightarrow}X_F^n$ admits an $R$-linear section. This implies that $R'\otimes_RV\cong\im(R'\otimes_R\varphi)$ satisfies condition (H).
\end{proof}
Finally, let $F'$ be a face of $\BT$ with $F'\subseteq\overline{F}$. Recall from (\ref{inclusions}) that we have $I_{F'}\subseteq I_F\subseteq P_F\subseteq P_{F'}$ in which $I_F$ is a normal subgroup of $P_F$. Consequently, we have the functors
\[
\Rep_R^\infty(P_{F'})\stackrel{(\cdot)^{I_F}}{\longrightarrow}\Rep_R^\infty(P_F)\quad\mbox{and}\quad
\Rep_R^\infty(P_{F'}^\dagger)\stackrel{(\cdot)^{I_F}}{\longrightarrow}\Rep_R^\infty(P_F^\dagger)
\]
of $I_F$-invariants. The non-trivial case of the following result is again due to Cabanes who works over a field of characteristic $p$ (cf.\ \cite{Cab}, Theorem 10). We will sketch his proof in order to convince the reader that it works for more general coefficients.
\begin{prop}\label{invariants_H}
Assume that $R$ is a quasi-Frobenius ring. If $F'$ and $F$ are faces of $\BT$ with $F'\subseteq\overline{F}\subseteq \Cbar$ then the functors $(\cdot)^{I_F}:\Rep_R^\infty(P_{F'})\to \Rep_R^\infty(P_F)$ and $(\cdot)^{I_F}:\Rep_R^\infty(P_{F'}^\dagger)\to \Rep_R^\infty(P_F^\dagger)$ preserve condition (H).
\end{prop}
\begin{proof}
By definition it suffices to treat the pair $(P_{F'}, P_F)$. Since the functor $(\cdot)^{I_F}$ preserves filtered unions we only need to treat representations whose underlying $R$-modules are finitely generated. By Proposition \ref{I_F_invariants} (iii) we have the isomorphism $X_{F'}^{I_F}\cong X_F\otimes_{H_F}H_{F'}$ of smooth $P_F$-representations and right $H_{F'}$-modules. Since $H_{F'}$ is free over $H_F$ (cf.\ Proposition \ref{free} (iii)), it follows that $X_{F'}^{I_F}$ satisfies condition (H). Given $V\in\Rep_R^{H}(P_{F'})$ which is finitely generated over $R$ we can embed $V$ into $X_{F'}^n$ for some non-negative integer $n$ and obtain a $P_F$-equivariant embedding $V^{I_F}\hookrightarrow (X_{F'}^{I_F})^n$. By Corollary \ref{injective_surjective} (i) it remains to see that $R[P_F]\cdot V^I=V^{I_F}$.\\

Note that $R$ is artinian (cf.\ \cite{Lam}, Theorem 15.1) hence is a direct product of local artinian rings. In any of the factors $p$ is either invertible or nilpotent. Thus, we may write $R=R_1\times R_2$ such that $p$ is invertible in $R_1$ and nilpotent in $R_2$. For any $R$-module $M$ this induces a decomposition $M=M_1\times M_2$ such that the action of $R$ on $M_i$ factors through $R_i$. Since $R_1$ and $R_2$ are both quasi-Frobenius (cf.\ \cite{Lam}, Corollary 3.11 B) we may assume that $p$ is invertible or nilpotent in $R$.\\

If $p$ is invertible in $R$ then the functor $(\cdot)^{I_F}:\Rep_R^\infty(P_{F'})\to \Rep_R^\infty(P_F)$ is exact by the usual averaging argument. Note that $I_F$ is a pro-$p$ group. Choose a $P_{F'}$-equivariant surjection $X_{F'}^n\to V$. It gives rise to the $P_F$-equivariant surjection $(X_{F'}^{I_F})^n\to V^{I_F}$ in which $(X_{F'}^{I_F})^n\cong X_F^n\otimes_{H_F}H_{F'}$ is generated by its $I$-invariants over $P_F$. Therefore, so is $V$.\\

If $p$ is nilpotent in $R$ we mimick the proof of \cite{Cab}, Theorem 10. All representations we consider are representations of the finite split reductive group $H=P_{F'}/I_{F'}\cong[\mathring{\GG}_{F',k}/\mathrm{R}^u(\mathring{\GG}_{F',k})](k)$. Note that $I'/I_{F'}$ is a Borel subgroup of $H$ with unipotent radical $U=I/I_{F'}$. We denote by $\Ubar$ the unipotent radical of the Borel subgroup opposite to $I'/I_{F'}$. Denoting by $J$ and $\overline{J}$ the augmentation ideals of the group rings $R[U]$ and $R[\Ubar]$, respectively, we will first show that
\[
 V=V^U\oplus\overline{J}\cdot V^U.
\]
Note that $V^U+\overline{J}\cdot V^U=R[\Ubar]\cdot V^U$. Thus, in order to prove $V=V^U+\overline{J}\cdot V^U$ it suffices to see that $R[\Ubar]\cdot V^U$ is $H$-stable because $V=R[H]\cdot V^U$ by Lemma \ref{properties_H} (i). This is done as in \cite{Cab}, Lemma 7, which relies only on the structure theory of the group $H$ and works over any coefficient ring. As a consequence, we have
\[
\overline{J}\cdot V=\overline{J}\cdot (V^U+\overline{J}\cdot V^U)=\overline{J}\cdot V^U.
\]
In order to prove $V^U\cap\overline{J}\cdot V^U=0$ note that also the $H$-represen\-tation $V^*$ is generated by its $U$-invariants (cf.\ Lemma \ref{properties_H} (ii)) and hence that $V^*=R[\Ubar]\cdot (V^*)^U$ as for $V$. Since $U$ and $\Ubar$ are conjugate in $H$ we obtain $V^*=R[U]\cdot(V^*)^\Ubar$. For any $R$-submodule $M$ of $V^*$ we set
\[
M^\perp=\{v\in V\;|\;\delta(v)=0\mbox{ for all }\delta\in M\}
\]
and claim that $0=(R[U]\cdot(V^*)^\Ubar)^\perp=\bigcap_{u\in U}u\cdot\overline{J}\cdot V^U$. The left equality comes from $V^*=R[U]\cdot(V^*)^\Ubar$ and the fact that any finitely generated $R$-module is reflexive (cf.\ \cite{Lam}, Theorem 15.11). That the right hand side is contained in the intermediate term is a straightforward computation using that $\overline{J}=\sum_{\overline{u}\in\Ubar}(1-\overline{u})R$. Now $v\in V^U\cap\overline{J}\cdot V^U$ implies  $v\in\bigcap_{u\in U}u\cdot\overline{J}\cdot V^U=0$.\\

By abuse of notation we write $P_F=P_F/I_{F'}\subseteq H$ which is a parabolic subgroup of $H$ with unipotent radical $U_F=I_F/I_{F'}$. We let $\Pbar_F$ denote the parabolic subgroup of $H$ opposite to $P_F$, $L_F=P_F\cap\Pbar_F$ their Levi subgroup and $\Ubar_F$ the unipotent radical of $\Pbar_F$. Further, we set $U_{(F)}=U\cap L_F$, $\Ubar_{(F)}=\Ubar\cap L_F$ and denote by $J_F$, $\overline{J}_F$, $J_{(F)}$ and $\overline{J}_{(F)}$ the augmentation ideals of the group rings $R[U_F]$, $R[\Ubar_F]$, $R[U_{(F)}]$ and $R[\Ubar_{(F)}]$, respectively. Since $U_F\subseteq U$ and since $U_F$ is normalized by $\Ubar_{(F)}\subseteq P_F$ we have
\[
R[\Ubar_{(F)}]\cdot V^U\subseteq V^{U_F}.
\]
Moreover, the decomposition $\Ubar=\Ubar_F\cdot\Ubar_{(F)}$ implies
\begin{eqnarray*}
V&=&R[\Ubar]\cdot V^U=R[\Ubar_F]\cdot R[\Ubar_{(F)}]\cdot V^U\\
&=&R[\Ubar_{(F)}]\cdot V^U+\overline{J}_F\cdot R[\Ubar_{(F)}]\cdot V^U\\
&\subseteq&R[\Ubar_{(F)}]\cdot V^U+\overline{J}_F\cdot V^{U_F}
\end{eqnarray*}
whence $V=R[\Ubar_{(F)}]\cdot V^U+\overline{J}_F\cdot V^{U_F}$. Note that $U_{(F)}\subseteq L_F=P_F\cap \Pbar_F$ normalizes both $U_F$ and $\Ubar_F$ so that $V^{U_F}\cap\overline{J}_F\cdot V^{U_F}$ ist $U_{(F)}$-stable. Since
\[
 (V^{U_F}\cap\overline{J}_F\cdot V^{U_F})^{U_{(F)}}\subseteq V^{U_F\cdot U_{(F)}}\cap\overline{J}\cdot V=V^U\cap\overline{J}\cdot V^U=0
\]
we get $V^{U_F}\cap\overline{J}_F\cdot V^{U_F}=0$ from the fact that $U_{(F)}$ is a $p$-group and since $p$ is nilpotent in $R$ (cf.\ Lemma \ref{I_invariants_non_trivial}). Now if $v\in V^{U_F}\subseteq V$ then we can write $v=v'+v''$ with $v'\in R[\Ubar_{(F)}]\cdot V^U\subseteq V^{U_F}$ and $v''\in \overline{J}_F\cdot V^{U_F}$, as seen above. But then $v''\in V^{U_F}\cap\overline{J}_F\cdot V^{U_F}=0$ and $v=v'\in R[\Ubar_{(F)}]\cdot V^U$. This shows $V^{U_F}=R[\Ubar_{(F)}]\cdot V^U$.
\end{proof}


\subsection{Coefficient systems and pro-$p$ Iwahori-Hecke modules}\label{subsection_3_2}

If $\F\in\Coeff_G(\BT)$ then the oriented chain complex $\C_c^\ori(\BT_{(\bullet)},\F)$ is a complex of smooth $R$-linear $G$-representations and the corresponding complex $\C_c^\ori(\BT_{(\bullet)},\F)^I$ of $I$-invariants is a complex of $H$-modules. Let
\[
M(\F)=H_0(\C_c^\ori(\BT_{(\bullet)},\F)^I)
\]
denote its homology in degree zero so that we obtain the functor
\begin{equation}\label{functor_M}
M(\cdot):\Coeff_G(\BT)\longrightarrow\Mod_H.
\end{equation}
The purpose of this subsection is to show that if the ring $R$ is quasi-Frobenius then the functor $M(\cdot)$ yields an equivalence of additive categories when restricted to a suitable full subcategory of $\Coeff_G(\BT)$.
\begin{defin}\label{category_C}
Let $\C$ denote the full subcategory of $\Coeff_G(\BT)$ consisting of all objects $\F$ satisfying the following conditions:
\begin{enumerate}[wide]
\item[(i)]For any face $F$ of $\BT$ the smooth $R$-linear $P_F^\dagger$-representation $\F_F$ satisfies condition (H) (cf.\ Definition \ref{condition_H} (ii)).
\item[(ii)]For any two faces $F'$ and $F$ of $\A$ such that $F'\subseteq\overline{F}$ and $C(F')=C(F)$ the restriction map $t^F_{F'}$ of $\F^I\in\Coeff(\A)$ is bijective (cf.\ \S\ref{subsection_2_2}).
\end{enumerate}
\end{defin}
Note that by Definition \ref{category_C} (i) and Remark \ref{H_I_F_trivial} any coefficient system $\F\in\C$ is automatically of level zero, i.e.\ $\C$ is a full subcategory of $\Coeff_G^0(\BT)$. Further, recall that if $C(F)=C(F')$ then $t^F_{F'}:\F_F^{I_{C(F)}}\to\F_{F'}^{I_{C(F)}}$ is obtained from the restriction map $r^F_{F'}$ of the coefficient system $\F$ by passage to the invariants under $I_{C(F)}$ (cf.\ \S\ref{subsection_2_2}).\\

As suggested by Theorem \ref{diagrams}, the transitivity properties of the $G$-action on $\BT$ imply that it suffices to check the conditions of Definition \ref{category_C} on the closed chamber $\Cbar$.
\begin{lem}\label{category_C_chamber}For an object $\F\in\Coeff_G(\BT)$ the following are equivalent.
\begin{enumerate}[wide]
\item[(i)]$\F$ is an object of the category $\C$.
\item[(ii)]For all faces $F'$ and $F$ of $\BT$ with $F'\subseteq\overline{F}\subseteq\Cbar$ we have $\F_F\in\Rep_R^H(P_F^\dagger)$ and the transition map $t^F_{F'}:\F_F^I\to\F_{F'}^I$ is bijective.
\end{enumerate}
\end{lem}
\begin{proof}
Clearly, (i) implies (ii). Conversely, assume that $\F$ satisfies the conditions in (ii). An arbitrary face of $\BT$ is of the form $gdF$ for some face $F\subseteq\Cbar$, $d\in\tilde{D}_F$ and $g\in I$ (cf.\ the proof of Lemma \ref{face_representatives}). Using the isomorphisms (\ref{phi_gd}) and the isomorphism $c_{gd,F}:\F_F\to\F_{gdF}$ it is straightforward to see that $\F_F\in\Rep_R^H(P_F^\dagger)$ implies $\F_{gdF}\in\Rep_R^H(P_{gdF}^\dagger)$.\\

If $F'$ and $F$ are faces of $\A$ with $F'\subseteq F$ and $C(F')=C(F)$ then there is an element $w\in W_\aff$ with $wC=C(F)$ (cf.\ \cite{BGL}, V.3.2, Th\'eor\`eme 1). In the notation of Lemma \ref{face_representatives} we have $[F']=w^{-1}F'$, $[F]=w^{-1}F$ and may consider the diagram
\[
\xymatrix{
\F_{[F]}^I \ar[r]^{c_{w,[F]}} \ar[d]_{t^{[F]}_{[F']}} & \F_F^{I_{C(F)}} \ar[d]^{t^F_{F'}} \\
\F_{[F']}^I \ar[r]_{c_{w,[F']}} & \F_{F'}^{I_{C(F)}}
}
\]
with bijective horizontal arrows. Note that since $C(F)=C(F')$ the vertical maps are obtained from the restriction maps $r^{[F]}_{[F']}$ and $r^F_{F'}$ by passage to the invariants under $I$ and $I_{C(F)}$, respectively. Therefore, the diagram commutes. It follows that together with $t^{[F]}_{[F']}$ also $t^F_{F'}$ is bijective.
\end{proof}
\begin{rem}\label{basic_0_diagrams}
Assume that $G=\mathrm{GL}_2(K)$ and that $R$ is a field of characteristic $p$. It follows from Lemma \ref{category_C_chamber} that under the equivalence $\Coeff_G(\BT)\cong\Diag(\Cbar)$ of Theorem \ref{diagrams} the category $\C$ corresponds to the full subcategory of basic $0$-diagrams $(D_1,D_0,r)$ considered in \cite{BP}, \S9, for which $D_0$ satisfies condition (H) as a representation of $\mathfrak{K}_0=P_{x_0}$. Note that this implies $R[P_{x_0}]\cdot r(D_1)=D_0$ by Lemma \ref{properties_H} (i). The latter condition, as well as the condition in Definition \ref{category_C} (ii) is also satisfied by the coefficient systems corresponding to the canonical diagrams of Hu (cf.\ \cite{Hu}).
\end{rem}
For trivial reasons, the fixed point system $\F_V$ associated to a representation $V\in\Rep_R^\infty(G)$ (cf.\ Example \ref{fixed_point_system}) always satisfies condition (ii) of Definition \ref{category_C}. However, it does not always satisfy condition (i) even if $R[G]\cdot V^I=V$. If $R$ is an algebraically closed field of characteristic $p$, for example, and if $G=\GL_2(\mathbb{Q}_p)$ then there is an irreducible $G$-representation $V$ for which the $P_{x_0}$-representation $V^{I_{x_0}}$ is not generated by its $I$-invariants (cf.\ \cite{OS1}, Remark 3.2.3). By Lemma \ref{properties_H} (i) $V^{I_{x_0}}$ does not satisfy condition (H). On the other hand, if $V=X$ then $\F_X\in\C$ as follows from Proposition \ref{free} (ii), Proposition \ref{I_F_invariants} (ii) and Lemma \ref{category_C_chamber} (i).\\

Still, there is a rather strong connection between the objects of $\C$ and suitable fixed point systems. In order to explain this, assume that $R$ is a quasi-Frobenius ring and let $\F\in\Coeff_G(\BT)$. Once the $P_x$-representation $\F_x$ satisfies condition (H) for all vertices $x\in\BT_{0}$, we have $\F\in\C$ if and only if locally around $x$ the system $\F$ is isomorphic to a fixed point sheaf in the sense of Ronan-Smith (cf.\ \cite{RS}, page 322). More precisely, recall that the star $\St(x)$ of $x$ is the union of all faces of $\BT$ containing $x$ in their closure. It is an open neighborhood of $x$ in $\BT$ with a simplicial action of the group $P_x$. In particular, we have the category $\Coeff_{P_x}(\St(x))$ of $P_x$-equivariant coefficient systems on $\St(x)$ at our disposal. Given a representation $V_x\in\Rep_R^\infty(P_x)$, for example, the family $\F_{V_x}=(V_x^{I_F})_{F\subseteq \St(x)}$ is naturally an object of $\Coeff_{P_x}(\St(x))$. Note that by (\ref{inclusions}) we have $P_F\subseteq P_x$ for any face $F\subseteq\St(x)$.
\begin{prop}\label{characterization_C}
Assume that $R$ is a quasi-Frobenius ring. For any object $\F\in\Coeff_G(\BT)$ the following are equivalent.
\begin{enumerate}[wide]
\item[(i)]$\F$ is an object of the category $\C$.
\item[(ii)]For every vertex $x\in\BT_{0}$ there is a representation $V_x\in\Rep_R^H(P_x)$ and an isomorphism $\F|_{\St(x)}\cong\F_{V_x}$ in $\Coeff_{P_x}(\St(x))$.
\item[(iii)]For every vertex $x\in\Cbar$ there is a representation $V_x\in\Rep_R^H(P_x)$ and an isomorphism $\F|_{\St(x)}\cong\F_{V_x}$ in $\Coeff_{P_x}(\St(x))$.
\end{enumerate}
\end{prop}
\begin{proof}
Assume that $\F\in\C$, let $x$ be a vertex of $\BT$ and set $D=C(x)$. Let $F$ be a face of $\BT$ with $x\in\overline{F}\subseteq\overline{D}$. Note that in this situation we have $C(F)=D$ by the uniqueness assertion in Lemma \ref{closest_chamber}. We claim that the restriction map $r^F_x:\F_F\to \F_x$ is injective with image $\F_x^{I_F}$. Since $\F_F$ satisfies condition (H), the action of $I_F$ on $\F_F$ is trivial (cf.\ Remark \ref{H_I_F_trivial}). Since $r^F_x$ is $P_F$-equivariant, its image is contained in $\F_x^{I_F}$. The $P_F$-representation $\F_x^{I_F}$ satisfies condition (H) by Proposition \ref{invariants_H}. By Corollary \ref{injective_surjective} (iii) it suffices to see that $r^F_x$ is bijective on $I_D$-invariants. Choosing $g\in I$ with $gD\subseteq\A$ (cf.\ the proof of Lemma \ref{face_representatives}) the diagram
\[
\xymatrix{\F_F^{I_D} \ar[r]^{c_{g,F}} \ar[d]_{r^F_x} & \F_{gF}^{I_{gD}} \ar[d]^{r^{gF}_{gx}} \\
\F_x^{I_D} \ar[r]_{c_{g,x}} & \F_{gx}^{I_{gD}}}
\]
is commutative with bijective horizontal arrows. Note that $C(gx)=gD=C(gF)$ by Lemma \ref{closest_chamber} so that the right vertical arrow is equal to the restriction map $t^{gF}_{gx}$ of the coefficient system $\F^I$ on $\A$. Since this is bijective (cf.\ Definition \ref{category_C} (ii)) so is the vertical arrow on the left.\\

If $F'$ is an arbitrary face of $\BT$ with $F'\subseteq\St(x)$ then there are elements $h,h'\in G_\aff$ with $h'C(F')=C=hD$ (cf.\ Lemma \ref{face_representatives}). This implies $hx,h'x\in\Cbar$ and hence $hx=h'x$ by the uniqueness assertion in Lemma \ref{face_representatives}. Thus, $\gamma=h^{-1}h'\in G_\aff\cap P_x^\dagger=P_x$ (cf.\ \cite{OS1}, Lemma 4.10) with $x=\gamma x\in \gamma\overline{F}'\subseteq \gamma \overline{C(F')}=\overline{D}=\overline{C(x)}$. As a consequence, the uniqueness assertion of Lemma \ref{closest_chamber} gives $C(\gamma F')=C(\gamma x)=C(x)=D$. The commutativity of the diagram
\[
\xymatrix{
\F_{F'} \ar[r]^{c_{\gamma,F'}} \ar[d]_{r^{F'}_x} & \F_{\gamma F'} \ar[d]^{r^{\gamma F'}_x} \\
\F_x \ar[r]_\gamma & \F_x
}
\]
and the above arguments show that $r^{F'}_x$ is injective with image $\F_x^{I_{F'}}$. Setting $V_x=\F_x\in\Rep_R^H(P_x)$ the family $(r^F_x)_{F\subseteq\St(x)}:\F|_{\St(x)}\to\F_{V_x}$ is an isomorphism in $\Coeff_{P_x}(\St(x))$.\\

Trivially, (ii) implies (iii). Thus, it remains to show that any coefficient system $\F\in\Coeff_G$ satisfying (iii) is an object of the category $\C$. If $F'$ and $F$ are faces of $\BT$ with $F'\subseteq\overline{F}\subseteq\Cbar$ choose an arbitrary vertex $x\in F'$. By assumption, $\F_F\cong\F_{V_x}^{I_F}$ satisfies condition (H) as a representation of $P_F$ (cf.\ Proposition \ref{invariants_H}) and hence of $P_F^\dagger$. Moreover, the transition map $t^F_{F'}$ may be identified with the identity map on $V_x^I$, hence is bijective. It follows from Lemma \ref{category_C_chamber} that $\F\in\C$.
\end{proof}
\begin{rem}\label{star_indetermined}
Let $\F\in\Coeff_G(\BT)$. We point out that the Ronan-Smith sheaves $\F|_{\St(x)}\cong\F_{V_x}$ with $x\in\Cbar$ do not necessarily determine $\F$ as an object of $\Coeff_G(\BT)$. In fact, they only capture the actions of the parahoric subgroups $P_F$ with $F\subseteq\Cbar$. However, in order to spread out a coefficient system on $\Cbar$ to all of $\BT$ in a $G$-equivariant way, one needs compatible actions of the stabilizer groups $P_F^\dagger$ (cf.\ Definition \ref{diagram} and Theorem \ref{diagrams}).
\end{rem}
\begin{thm}\label{equivalence}
If $R$ is a quasi-Frobenius ring then the functor $M(\cdot):\C\to\Mod_H$ is an equivalence of additive categories.
\end{thm}
\begin{proof}
Recall that we set $X=\ind_I^G(R)\in\Rep_R^\infty(G)$ and denote by $\F_X=(X^{I_F})_F\in\Coeff_G(\BT)$ the corresponding fixed point system (cf.\ Example \ref{fixed_point_system}). Since the $G$-representation $X$ carries a commuting right $H$-module structure $\F_X$ is a $G$-equivariant coefficient system of right $H$-modules. Given $M\in\Mod_H$ we set $\F_X\otimes_HM=(X^{I_F}\otimes_HM)_F$ and obtain the functor
\[
(M\mapsto\F_X\otimes_HM):\Mod_H\longrightarrow\Coeff_G(\BT).
\]
In order to prove the essential surjectivity we choose an embedding $M\hookrightarrow E$ of $M$ into an injective $H$-module $E$ and set
\[
\F(M)=\im(\F_X\otimes_HM\longrightarrow\F_X\otimes_HE)\in\Coeff_G(\BT).
\]
Let us first show that $\F(M)\in\C$ by checking the conditions in Lemma \ref{category_C_chamber} (ii). If $F$ is a face of $\BT$ with $F\subseteq\Cbar$ then there is an isomorphism $X^{I_F}\otimes_HM\cong X_F^\dagger\otimes_{H_F^\dagger}M$ in $\Rep_R^\infty(P_F^\dagger)$ (cf.\ Proposition \ref{I_F_invariants} (ii)). Since $H$ is free over $H_F^\dagger$ (cf.\ Proposition \ref{free} (ii)) $E$ is an  injective $H_F^\dagger$-module via restriction of scalars (cf.\ \cite{Lam}, Corollary 3.6A). Therefore,
\[
\F(M)_F\cong\im(X_F^\dagger\otimes_{H_F^\dagger}M\longrightarrow X_F^\dagger\otimes_{H_F^\dagger}E)
\]
is an object of $\Rep_R^H(P_F^\dagger)$ and the natural map $M\to\F(M)_F^I$ is an isomorphism of $H_F^\dagger$-modules (cf.\ Theorem \ref{Cabanes_general} (i)). If $F'\subseteq\overline{F}$ with $C(F')=C(F)$ then the commutativity of the diagram
\[
\xymatrix{
M \ar[d]_{\mathrm{id}_M} \ar[rr]^\cong && \F(M)_F^I \ar[d]^{t^F_{F'}} \\
M \ar[rr]_\cong && \F(M)_{F'}^I
}
\]
implies that $t^F_{F'}$ is bijective. Thus, $\F(M)\in\C$ as claimed. Further, by Proposition \ref{acyclic} we have the commutative diagram
\[
\xymatrix{
M \ar[d]_{\mathrm{id}_M} \ar[rr]^\cong && \F(M)_C^I \ar[d]^{t^C_F} \ar[rr]^{\iota_C} && M(\F(M)) \ar[d]^{\mathrm{id}_{M(\F(M))}} \\
M \ar[rr]_\cong && \F(M)_F^I \ar[rr]_{\iota_F} && M(\F(M))
}
\]
in which all arrows are bijective and the lower horizontal arrows are $H_F^\dagger$-linear. It follows that the upper horizontal map $M\to M(\F(M))$ is bijective and $H_F^\dagger$-linear for all faces $F\subseteq\Cbar$. Since the subalgebras $H_F^\dagger$ of $H$ with $F\subseteq\Cbar$ generate $H$ (cf.\ Lemma \ref{Hecke_generation}) the map $M\to M(\F(M))$ is an isomorphism of $H$-modules. This proves the essential surjectivity of $M(\cdot)$.\\

To prove that $M(\cdot)$ is fully faithful let $\F,\G\in\C$ and $f\in\Hom_\C(\F,\G)=\Hom_{\Coeff_G(\BT)}(\F,\G)$. If $F$ is a face of $\BT$ with $F\subseteq\Cbar$ we have the commutative diagram
\[
\xymatrix{
\F_F^I \ar[rr]^{\iota_F} \ar[d]_{f_F^I} && M(\F) \ar[d]^{M(f)} \\
\G_F^I \ar[rr]_{\iota_F} && M(\G)
}
\]
of $H_F^\dagger$-modules in which the horizontal arrows are bijective (cf.\ Proposition \ref{acyclic} (ii)). Thus, $M(f)=0$ implies $f_F^I=0$. Since both $P_F^\dagger$-representations $\F_F$ and $\G_F$ satisfy condition (H) we have $f_F=0$ by Proposition \ref{properties_H} (i) and then $f=0$ by Theorem \ref{diagrams}.\\

If $g\in\Hom_H(M(\F),M(\G))$ and if $F\subseteq\Cbar$ we define $g_F=\iota_F^{-1}\circ g\circ\iota_F\in\Hom_{H_F^\dagger}(\F_F^I,\G_F^I)$. By Theorem \ref{Cabanes_general} (ii) there is a homomorphism $f_F:\F_F\to\G_F$ of $P_F^\dagger$-representations such that $f_F^I=g_F$. Let $F'$ be a face of $\BT$ with $F'\subseteq\overline{F}$ and denote by $r^F_{F'}:\F_F\to\F_{F'}$ and $s^F_{F'}:\G_F\to\G_{F'}$ the corresponding restriction maps of $\F$ and $\G$, respectively. We claim that $f_{F'}\circ r^F_{F'}=s^F_{F'}\circ f_F$. By (\ref{inclusions}) and Remark \ref{H_I_F_trivial} the maps $r^F_{F'}$ and $s^F_{F'}$ take values in $\F_{F'}^{I_F}$ and $\G_{F'}^{I_F}$, respectively. Thus, we need to prove that the diagram
\[
\xymatrix{
\F_F \ar[rr]^{f_F} \ar[d]_{r^F_{F'}} && \G_F \ar[d]^{s^F_{F'}} \\
\F_{F'}^{I_F} \ar[rr]_{f_{F'}} && \G_{F'}^{I_F}
}
\]
is commutative. Since $P_F\subseteq P_{F'}$ by (\ref{inclusions}) the latter may be viewed as a diagram in $\Rep_R^H(P_F)$ (cf.\ Proposition \ref{invariants_H}). By Corollary \ref{injective_surjective} (iii) the commutativity can be checked on $I$-invariants where it holds by construction and because of Proposition \ref{acyclic} (iii).\\

Finally, given $g\in P_C^\dagger$ we claim that the diagram
\[
\xymatrix{
\F_F \ar[rr]^{f_F} \ar[d]_{c_{g,F}} && \G_F \ar[d]^{c_{g,F}} \\
\F_{gF} \ar[rr]_{f_{gF}} && \G_{gF}
}
\]
is commutative. Since $I\subseteq P_F^\dagger\cap P_{gF}^\dagger$ both $f_F$ and $f_{gF}$ are $I$-equivariant. Therefore, the decomposition (\ref{Bruhat_parahoric}) of $P_C^\dagger$ and the properties of the $G$-actions on $\F$ and $\G$ allow us to assume $g=\omega\in N_G(T)\cap P_C^\dagger$ such that the image of $\omega$ in $\tilde{W}$ lies in $\tilde{\Omega}$. We endow the $P_{\omega F}^\dagger$-represen\-ta\-tions $\F_{\omega F}$ and $\G_{\omega F}$ with an action of $P_F^\dagger$ through conjugation with $\omega$. Then the above diagram may again be viewed as a diagram in $\Rep_R^H(P_F^\dagger)$ because of the isomorphisms (\ref{phi_gd}). Note that $\tilde{\Omega}\subseteq\tilde{D}_F$ by Remark \ref{Omega_in_DF}. By Corollary \ref{injective_surjective} (iii) it suffices to prove the commutativity after passage to the $I$-invariants where it follows from (\ref{iota_equivariant}), the $H$-linearity of $g$, as well as from $f_F^I=\iota_F^{-1}\circ g\circ\iota_F$ and $f_{\omega F}^I=\iota_{\omega F}^{-1}\circ g\circ \iota_{\omega F}$. \\

Altogether, we have shown that the family $(f_F)_{F\subseteq\Cbar}:\res(\F)\to\res(\G)$ is a homomorphism of diagrams. By Proposition \ref{diagrams} it extends to a homomorphism $f\in\Hom_{\Coeff_G(\BT)}(\F,\G)$. Choosing an arbitrary face $F$ of $\BT$ contained in $\Cbar$ we have $M(f)=\iota_F\circ f_F^I\circ\iota_F^{-1}=g$ by construction.
\end{proof}
For any $H$-module $M$ it follows that up to isomorphism the coefficient system $\F(M)\in\C$ constructed in the proof of Theorem \ref{equivalence} does not depend on the choice of the embedding $M\hookrightarrow E$ into an injective $H$-module $E$. Of course, this can easily be proved directly. What is more, since $M(\cdot)$ is an equivalence and since $M(\F(M))\cong M$ by the proof of Theorem \ref{equivalence} such an isomorphism is unique and the assignment $M\mapsto\F(M)$ is a functor which is quasi-inverse to $M(\cdot)$. Once again a presumably more natural construction of a quasi-inverse can be given by making use of an idea of \cite{OS2}, \S1.2.\\

Let $M\in\Mod_H$ and let $F$ be an arbitrary face of $\BT$. As in \S\ref{subsection_3_1} the action of $P_F^\dagger$ on $X^{I_F}$ makes $X^{I_F}\otimes_HM$ and $\Hom_H(\Hom_H(X^{I_F},H),M)$ objects of $\Rep_R^\infty(P_F^\dagger)$. Moreover, there is a unique homomorphism
\[
 \tau_{M,F}:X^{I_F}\otimes_HM\longrightarrow\Hom_H(\Hom_H(X^{I_F},H),M)
\]
of smooth $R$-linear $P_F^\dagger$-representations sending $x\otimes m$ to the $H$-linear map $(\varphi\mapsto\varphi(x)\cdot m)$. We denote by $\t_F(M)=\im(\tau_{M,F})$
its image and obtain the functor $(M\mapsto\t_F(M)):\Mod_H\to\Rep_R^\infty(P_F^\dagger)$. Varying $F$ the family $\F(M)=(\t_F(M))_F$ is a $G$-equivariant coefficient system on $\BT$ because of the $G$-action on $X$. We thus obtain the functor
\[
 \F(\cdot)=(\t_F(\cdot))_F:\Mod_H\to\Coeff_G(\BT).
\]
\begin{thm}\label{quasi_inverse_coefficient_system}
Assume that $R$ is a quasi-Frobenius ring. The functor $\F(\cdot):\Mod_H\to\Coeff_G(\BT)$ takes values in the category $\C$ and is quasi-inverse to the equivalence $M(\cdot):\C\to\Mod_H$ of Theorem \ref{equivalence}.
\end{thm}
\begin{proof}
Let $F$ be a face of $\BT$ with $F\subseteq\Cbar$. By Proposition \ref{I_F_invariants} (ii) there is an isomorphism $X^{I_F}\otimes_HM\cong X_F^\dagger\otimes_{H_F^\dagger}M$ of smooth $R$-linear $P_F^\dagger$-representations. Further, by Proposition \ref{I_F_invariants} (i) and (ii), and by \cite{BAC}, I.2.9 Proposition 10, we have
\begin{eqnarray*}
 \Hom_H(X^{I_F},H)&\cong&\Hom_{H_F}(X_F,H)\cong H\otimes_{H_F} \Hom_{H_F}(X_F,H_F)\\
&\cong&H\otimes_{H_F^\dagger}H_F^\dagger\otimes_{H_F}\Hom_{H_F}(X_F,H_F)\\
&\cong &H\otimes_{H_F^\dagger}\Hom_{H_F^\dagger}(X_F^\dagger,H_F^\dagger)
\end{eqnarray*}
as left $H$-modules because $X_F$ is finitely generated over $H_F$ and $H$ is free over $H_F^\dagger$ and $H_F$ (cf.\ Proposition \ref{free} (ii)). Consequently, there is a commutative diagram
\[
 \xymatrix{
 X^{I_F}\otimes_HM \ar[rr]^>>>>>>>>>>>>{\tau_{M,F}} \ar[d]_\cong && \Hom_H(\Hom_H(X^{I_F},H),M) \ar[d]^\cong \\
 X_F^\dagger\otimes_{H_F^\dagger}M \ar[rr]_>>>>>>>>>>{\tau_{M,F}} && \Hom_{H_F^\dagger}(\Hom_{H_F^\dagger} ( X_F^\dagger,H_F^\dagger),M)
 }
\]
in which the vertical arrows are isomorphisms and the lower horizontal arrow was introduced in \S\ref{subsection_3_1}. Actually, it turns out now that giving the same name to both horizontal arrows was only a minor abuse of notation. By Theorem \ref{Cabanes_general} (i) we have $\F(M)_F=\t_F(M)\in\Rep_R^H(P_F^\dagger)$ and the natural map $M\to\t_F(M)^I$ is an isomorphism of $H_F^\dagger$-modules. From this point on one can simply copy the proof of Theorem \ref{equivalence} to see that $\F(M)\in\C$ and that $\F(\cdot)$ is quasi-inverse to $M(\cdot)$.
\end{proof}
Finally, let us recall the construction of \cite{OS2}, \S1.2, which we have alluded to already twice. Given $M\in\Mod_H$ the $G$-action on $X$ induces $G$-actions on $X\otimes_HM$ and $\Hom_H(\Hom_H(X,H),M)$. Moreover, there is a unique homomorphism
\[
 \tau_M:X\otimes_HM\longrightarrow\Hom_H(\Hom_H(X,H),M)
\]
of $R$-linear $G$-representations sending $x\otimes m$ to the homomorphism of $H$-modules $(\varphi\mapsto\varphi(x)\cdot m)$. We denote by $\t(M)=\im(\tau_M)$
its image and obtain the functor
\[
 \t=(M\mapsto\t(M)):\Mod_H\to\Rep_R^\infty(G).
\]
Note that $\t(M)$ naturally is a quotient of $X\otimes_HM$ so that the $G$-action is smooth. If $F$ is a face of $\BT$ then the $(P_F^\dagger,H)$-equivariant inclusion $X^{I_F}\subseteq X$ induces a homomorphism $\Hom_H(X,H)\to\Hom_H(X^{I_F},H)$ of left $H$-modules and a commutative diagram
\[
\xymatrix{
 X^{I_F}\otimes_HM \ar[rr]^>>>>>>>>>{\tau_{M,F}} \ar[d] && \Hom_H(\Hom_H(X^{I_F},H),M) \ar[d] \\
 X\otimes_HM \ar[rr]_>>>>>>>>>>>{\tau_M} && \Hom_H(\Hom_H(X,H),M)
 }
\]
of $R$-linear $P_F^\dagger$-representations. This in turn gives rise to a homomorphism $\t_F(M)\to\t(M)$ in $\Rep_R^\infty(P_F^\dagger)$. Since the action of $I_F$ on $\t_F(M)$ is trivial, it factors through the inclusion $\t_F(M)\to\t(M)^{I_F}\hookrightarrow\t(M)$. Letting $F$ vary, one obtains homomorphisms
\begin{equation}\label{comparison_coefficient_systems}
\F_X\otimes_HM\twoheadrightarrow\F(M)\longrightarrow\F_{\t(M)}\longrightarrow\K_{\t(M)}
\end{equation}
of $G$-equivariant coefficient systems on $\BT$ where $\F_{\t(M)}$ and $\K_{\t(M)}$ denote the fixed point system and the constant coefficient system associated with the smooth $G$-representation $\t(M)$, respectively (cf.\ Examples \ref{constant_system} and \ref{fixed_point_system}). Passing to the homology in degree zero we obtain the following result.
\begin{prop}\label{comparison_representations}
For any $M\in\Mod_H$ there are surjective homomorphisms
\[
X\otimes_HM\twoheadrightarrow\mathrm{H}_0(\BT,\F(M))\twoheadrightarrow\t(M)
\]
of smooth $R$-linear $G$-representations which are functorial in $M$.
\end{prop}
\begin{proof}
By \cite{SS}, Theorem II.3.1, and \cite{OS1}, Remark 3.2.1, there is an isomorphism $\mathrm{H}_0(\BT,\F_X)\cong X$ whence $\mathrm{H}_0(\BT,\F_X\otimes_HM)\cong X\otimes_HM$ if $M$ is free over $H$. The right exactness of the functor $\mathrm{H}_0(\BT,\cdot)$ then implies $\mathrm{H}_0(\BT,\F_X\otimes_HM)\cong X\otimes_HM$ for any $M\in\Mod_H$. On the other hand, $\mathrm{H}_0(\BT,\K_{\t(M)})\cong\t(M)$ because $\BT$ is contractible. Altogether, applying $\mathrm{H}_0(\BT,\cdot)$ to (\ref{comparison_coefficient_systems}) yields homomorphisms
\[
X\otimes_HM\to\mathrm{H}_0(\BT,\F(M))\to \t(M)
\]
in $\Rep_R^\infty(G)$ the left one of which is surjective by the right exactness of $\mathrm{H}_0(\BT,\cdot)$. Unwinding definitions, the composition turns out to be $\tau_M$, hence is surjective, too. Thus, also the right homomorphism is surjective.
\end{proof}
\begin{rem}\label{t_isomorphism}
If $R$ is a field of characteristic zero then we shall see in Theorem \ref{description_functors_0} that all comparison morphisms in (\ref{comparison_homomorphisms}) and (\ref{link_help}) are isomorphisms. If $p$ is nilpotent in $R$ then the comparison homomorphism $\F_X\otimes_HM\to\F(M)$ is bijective only in exceptional cases (cf.\ Theorem \ref{exceptional_flat}). In any case, the precise relation to $\t(M)$ remains unclear.
\end{rem}
If $V\in\Rep_R^\infty(G)$ and if $M\in\Mod_H$ then there are natural homomorphisms of $G$-equivariant coefficient systems
\begin{equation}\label{comparison_homomorphisms}
 \F_X\otimes_HV^I\longrightarrow\F_V\quad\mbox{and}\quad\F_X\otimes_HM\to\F(M)
\end{equation}
on $\BT$ the second one of which is part of the construction of the functor $\F(\cdot)$. In order to construct the first one consider the $G$-equivariant map $X\otimes_HV^I\to V$ sending $f\otimes v$ to $\sum_{g\in G/I}f(g)gv$. For any face $F$ of $\BT$ it induces a $P_F^\dagger$-equivariant map $X^{I_F}\otimes_HV^I\to(X\otimes_HV^I)^{I_F}\to V^{I_F}$. Letting $F$ vary, the family of these is the required homomorphism $\F_X\otimes_HV^I\to\F_V$. Note that via the construction in \S\ref{subsection_2_2} these in turn induce homomorphisms
\[
\xymatrix@R=4pt{
 \F_X^I\otimes_HV^I \ar[r] & (\F_X\otimes_HV^I)^I \ar[r] & \F_V^I & \mbox{and}\\
 \F_X^I\otimes_HM \ar[r] & (\F_X\otimes_HM)^I \ar[r] & \F(M)^I &
 }
\]
of coefficient systems of $R$-modules on $\A$.
\begin{prop}\label{I_isomorphisms}
Assume that $R$ is a quasi-Frobenius ring. If $V\in\Rep_R^\infty(G)$ and if $M\in\Mod_H$ then the homomorphisms $\F_X^I\otimes_HV^I\to\F_V^I$ and $\F_X^I\otimes_HM\to\F(M)^I$ in $\Coeff(\A)$ are isomorphisms.
\end{prop}
\begin{proof}
For the first homomorphism this is shown in \cite{OS1}, Proposition 6.3, but the proof also works in the second case. Since we are working over general coefficients let us quickly recall the argument. By conjugation as in (\ref{d_conjugation}) and since the homomorphisms are induced by homomorphisms of $G$-equivariant coefficient systems it suffices to prove that the maps $(\F_X^I\otimes_HV^I)_F\longrightarrow(\F_V^I)_F$ and $(\F_X^I\otimes_HM)_F\longrightarrow\F(M)^I_F$ are isomorphisms for any face $F\subseteq\Cbar$. In the first case, this is the isomorphism $H\otimes_HV^I\to V^I$. In the second case this is the natural map $H\otimes_HM\cong M\to\t_F(M)^I=\F(M)_F^I$ which was shown to be bijective in Theorem \ref{quasi_inverse_OS}.
\end{proof}
\begin{rem}\label{GP_resolution}
If $M\in\Mod_H$ then Proposition \ref{restriction_chain_complexes} and Proposition \ref{I_isomorphisms} give isomorphisms of complexes
\begin{eqnarray*}
\C_c^\ori(\BT_{(\bullet)},\F(M))^I & \cong & \C_c^\ori(\A_{(\bullet)},\F(M)^I) \cong \C_c^\ori(\A_{(\bullet)},\F_X^I\otimes_HM)\\
& \cong & \C_c^\ori(\A_{(\bullet)},\F_X^I)\otimes_HM \cong \C_c(\BT_{(\bullet)},\F_X)^I\otimes_HM.
\end{eqnarray*}
If $R$ is a field then this is the Gorenstein projective resolution of $M$ constructed in \cite{OS1}, \S6.
\end{rem}
Given a representation $V\in\Rep_R^\infty(G)$ we have the $G$-equivariant coefficient systems $\F_V$ and $\F(V^I)$ on $\BT$. As seen above, they are linked through natural homomorphisms
\begin{equation}\label{link_help}
 \xymatrix{
 \F(V^I) & \F_X\otimes_HV^I \ar[r]\ar[l] & \F_V.
 }
\end{equation}
\begin{cor}\label{link}
If $V\in\Rep_R^\infty(G)$ then there is an isomorphism
\[\F(V^I)^I\cong\F_V^I\]
in $\Coeff(\A)$ which is natural in $V$.
\end{cor}
\begin{proof}
As seen above, the homomorphisms (\ref{link_help}) induce a diagram
\[
 \xymatrix{
 & \F_X^I\otimes_HV^I \ar[d]\ar[rd]\ar[ld] &\\
 \F(V^I)^I & (\F_X\otimes_HV^I)^I \ar[r]\ar[l] & \F_V^I.
 }
\]
in $\Coeff(\A)$ which is natural in $V$. Since the oblique arrows are isomorphisms (cf.\ Proposition \ref{I_isomorphisms}) the claim follows.
\end{proof}


\section{Applications to representation theory}\label{section_4}


\subsection{Homology in degree zero}\label{subsection_4_1}

Consider the $0$-th homology functor $\mathrm{H}_0(\BT,\cdot):\Coeff_G(\BT)\longrightarrow\Rep_R^\infty(G)$ introduced in \S\ref{subsection_2_1}.
\begin{defin}\label{representation_C}
Let $\Rep_R^\C(G)$ be the full subcategory of $\Rep_R^\infty(G)$ consisting of all objects which are isomorphic to a representation of the form $\mathrm{H}_0(\BT,\F)$ for some object $\F\in\C$ (cf.\ Definition \ref{category_C}).
\end{defin}
Recall also that we have the functor $(\cdot)^I:\Rep_R^\infty(G)\longrightarrow\Mod_H$. The aim of this subsection is to study its behavior on the full subcategory $\Rep_R^\C(G)$. The most complete results will be obtained in the case that $p$ is invertible in $R$ or that $R$ is even a field of characteristic zero. If $p$ is nilpotent in $R$ we will discuss the case of semisimple rank one at the end of this section.
\begin{thm}\label{I_equivalence}
Assume that $R$ is a quasi-Frobenius ring.
If $p$ is invertible in $R$ then the functor $(\cdot)^I:\Rep_R^\C(G)\to\Mod_H$ is an equivalence of categories with quasi-inverse $\mathrm{H}_0(\BT,\F(\cdot)):\Mod_H\to\Rep_R^\C(G)$.
\end{thm}
\begin{proof}
Since $p$ is invertible in $R$ and since $I$ is a pro-$p$ group the functor $(\cdot)^I:\Rep_R^\infty(G)\to\Mod_H$ is exact by the usual averaging argument. Given $M\in\Mod_H$ we therefore have a natural isomorphism of $H$-modules
\begin{eqnarray*}
 \mathrm{H}_0(\BT,\F(M))^I&=&(\mathrm{H}_0(\C_c^\ori(\BT_{(\bullet)},\F(M))))^I\\
 &\cong&\mathrm{H}_0(\C_c^\ori(\BT_{(\bullet)},\F(M))^I)\\
 &=&M(\F(M))\cong M
\end{eqnarray*}
by Theorem \ref{quasi_inverse_coefficient_system}. In order to show that also the other composition is isomorphic to the identity functor we may start with a representation of the form $V=\mathrm{H}_0(\BT,\F)$ for some object $\F\in\C$. As above, the exactness of $(\cdot)^I$ implies that there is a natural isomorphism $V^I\cong M(\F)$ of $H$-modules and thus a natural isomorphism
\[
 \mathrm{H}_0(\BT,\F(V^I))\cong\mathrm{H}_0(\BT,\F(M(\F)))\cong\mathrm{H}_0(\BT,\F)=V
\]
in $\Rep_R^\C(G)$ by Theorem \ref{quasi_inverse_coefficient_system} again.
\end{proof}
\begin{cor}\label{H_0_equivalence}
If $R$ is a quasi-Frobenius ring in which $p$ is invertible then the functor $\mathrm{H}_0(\BT,\cdot):\C\to\Rep_R^\C(G)$ is an equivalence of categories.
\end{cor}
\begin{proof}
As seen above, the functors $\mathrm{H}_0(\BT,\cdot)^I=(\cdot)^I\circ\mathrm{H}_0(\BT,\cdot)$ and $M(\cdot)$ from $\C$ to $\Mod_H$ are isomorphic. Therefore, the corollary is a consequence of Theorem \ref{equivalence} and Theorem \ref{I_equivalence}.
\end{proof}
Of course, there is a much more direct way to realize $\Mod_H$ as a full subcategory of $\Rep_R^\infty(G)$ if $p$ is invertible in $R$. If $J$ is an open pro-$p$ subgroup of $G$ and if $V\in\Rep_R^\infty(G)$ then we have the $R$-linear endomorphism
\[
e_J:V\longrightarrow V,\quad v\mapsto (J:J_v)^{-1}\sum_{g\in J/J_v}gv,
\]
of $V$ where $J_v$ denotes the centralizer of $v$ in $J$. Note that $(J:J_v)$ is a power of $p$ hence is invertible in $R$. Clearly, $e_J$ is idempotent and equivariant for the action of the normalizer of $J$ in $G$. It gives rise to the decomposition $V=\im(e_J)\oplus\ker(e_J)$ of $R$-modules with $\im(e_J)=V^J$. By definition $e_J$ commutes with any $J$-equivariant endomorphism of the $R$-module $V$.\\

We shall denote by $\Rep_R^I(G)$ the full subcategory of $\Rep_R^\infty(G)$ consisting of all representations $V$ which are generated by their $I$-invariants, i.e.\ for which $R[G]\cdot V^I=V$.
\begin{lem}\label{realization_0}
Assume that $p$ is invertible in $R$. If $M\in\Mod_H$ then the natural map $M\to(X\otimes_HM)^I$ is an isomorphism of $H$-modules. The functor $X\otimes_H(\cdot):\Mod_H\to\Rep^I_R(G)$ is fully faithful.
\end{lem}
\begin{proof}Denote by $1\in H$ the unit element of $H$ viewed as an $I$-invariant element of $X$. If $y=\sum_jx_j\otimes m_j\in X\otimes_HM $ is $I$-invariant then $y=e_I(y)=\sum_je_I(x_j)\otimes m_j=1\otimes(\sum_je_I(x_j)m_j)$ lies in the image of $M\to X\otimes_HM$. Since $X=H\oplus\ker(e_I)$ is a decomposition of right $H$-modules this map is also injective, proving the first assertion.\\

Note that the $G$-representation $X\otimes_HM$ is generated by $M$ whence the map $\Hom_G(X\otimes_HM,X\otimes_HN)\to\Hom_H((X\otimes_HM)^I,(X\otimes_HN)^I)$ is injective for all $N\in\Mod_H$. By what we have just proved its composition with the map $\Hom_H(M,N)\to \Hom_G(X\otimes_HM,X\otimes_HN)$ is bijective. This implies the second assertion.
\end{proof}
In fact, $\Rep_R^\C(G)$ is always a full subcategory of $\Rep_R^I(G)$ without any assumptions on $R$.
\begin{prop}\label{I_contains_C}
If $\F\in\Coeff_G(\BT)$ such that $\F_F\in\Rep_R^\infty(P_F^\dagger)$ is generated by its $I_{C(F)}$-invariants for all faces $F$ of $\BT$ then the oriented chain complex $\C_c^\ori(\BT_{(\bullet)},\F)$ consists of objects of $\Rep_R^I(G)$. In particular, the $G$-representation $\mathrm{H}_0(\BT,\F)$ is generated by its $I$-invariants and $\Rep_R^\C(G)$ is a full subcategory of $\Rep_R^I(G)$.
\end{prop}
\begin{proof}
For any $0\leq i\leq d$ there is an isomorphism of $G$-representations
\begin{equation}\label{chain_complex_induced}
\C_c^\ori(\BT_{(i)},\F)\cong\bigoplus_F\ind_{P_F^\dagger}^G(\varepsilon_F\otimes_R\F_F)
\end{equation}
where $F$ runs through a set of representatives of the finitely many $G$-orbits in $\BT_i$ and the character $\varepsilon_F:P_F^\dagger\to\{\pm 1\}$ describes how $P_F^\dagger$ changes any given orientation of $F$. By Lemma \ref{face_representatives} we may assume the corresponding faces $F$ to be contained in $\Cbar$. Now the $G$-representation $\ind_{P_F^\dagger}^G(\varepsilon_F\otimes_R\F_F)$ is generated by the $P_F^\dagger$-subrepresentation $\varepsilon_F\otimes_R\F_F$. Moreover, $(\varepsilon_F\otimes_R\F_F)^I=\varepsilon_F\otimes_R\F_F^I$ by \cite{OS1}, Lemma 3.1, which generates $\varepsilon_F\otimes_R\F_F$ over $P_F^\dagger$ by assumption. This proves the first assertion. The second assertion follows from the fact that the category $\Rep_R^I(G)$ is closed under quotients in $\Rep_R^\infty(G)$. The final assertion then follows from Lemma \ref{properties_H} (i).
\end{proof}
If $R$ is a field of characteristic zero then the categories $\Rep_R^\C(G)$ and $\Rep_R^I(G)$ coincide and the equivalences in \S\ref{subsection_3_2} admit more classical descriptions. Some of this relies on the following fundamental theorem of Bernstein (cf.\ \cite{Ber}, Corollaire 3.9). Since we could not find an explicit reference pertaining to the pro-$p$ Iwahori group $I$ we will give a quick argument reducing the statement to a known case of Bernstein's theorem.
\begin{thm}[Bernstein]\label{Bernstein}
Assume that $R$ is a field of characteristic zero. As a full subcategory of $\Rep_R^\infty(G)$ the category $\Rep_R^I(G)$ is stable under subquotients. The functors
\[
 \xymatrix{
 \Rep_R^I(G) \ar@<1ex>[rr]^>>>>>>>>>{(\cdot)^I} && \Mod_H \ar@<1ex>[ll]^<<<<<<<<<{X\otimes_H(\cdot)}
 }
\]
are mutually quasi-inverse equivalences of abelian categories. In particular, the right $H$-module $X$ is flat.
\end{thm}
\begin{proof}
Let $V$ be an object of $\Rep_R^I(G)$ and let $U\subseteq V$ be a $G$-sub\-represen\-tation. We claim that the natural map $X\otimes_HU^I\to U$ is bijective. To see this, we may assume that the field $R$ is uncountable and algebraically closed. Recall that we fixed the special vertex $x_0\in\Cbar$ and that $I_{x_0}\subseteq I$. By \cite{Ber}, Corollaire 3.9, the full subcategory of $\Rep_R^\infty(G)$ generated by their $I_{x_0}$-invariants is stable under subquotients (cf.\ the reasoning in \cite{Ber}, page 29, or \cite{SS}, Theorem I.3). Note that from \cite{Ber}, \S1.8 onwards, Bernstein works over the complex numbers. However, his arguments are valid for any uncountable and algebraically closed field of characteristic zero.\\

Since $X\otimes_HU^I$ and $V$ are generated by their $I$-invariants they are also generated by their $I_{x_0}$-invariants. By Bernstein's result, so is $U$. In order to see that the natural map $X\otimes_HU^I\to U$ is injective, it suffices to check this after passage to $I_{x_0}$-invariants because the kernel of this map is also generated by its $I_{x_0}$-invariants. As in the proof of Lemma \ref{realization_0} the decomposition $X=X^{I_{x_0}}\oplus\ker(e_{I_{x_0}})$ of right $H$-modules gives $(X\otimes_HU^I)^{I_{x_0}}=X^{I_{x_0}}\otimes_HU^I$. Using Proposition \ref{I_F_invariants} (ii) we need to see that the natural map $X_{x_0}\otimes_{H_{x_0}}U^I\to U^{I_{x_0}}$ is injective. Since the category $\Rep_R(P_{x_0}/I_{x_0})$ is semisimple its kernel $W$ is a quotient of $X_{x_0}\otimes_{H_{x_0}}U^I$ hence is generated by its $I$-invariants. However, the map $U^I=(X_{x_0}\otimes_{H_{x_0}}U^I)^I\to U^I$ is the identity whence $W^I=0$ and $W=0$.\\

In order to see that the natural map $X\otimes_HU^I\to U$ is surjective, consider the commutative diagram
\[
 \xymatrix{
 & X\otimes_HU^I \ar[d]\ar[r] & X\otimes_HV^I \ar[d]\ar[r] & X\otimes_H(V/U)^I \ar[d]\ar[r] & 0 \\
 0 \ar[r] & U \ar[r] & V \ar[r] & V/U \ar[r] & 0
 }
\]
in which all vertical arrows are injective by our above reasoning. Moreover, the lower row is exact by definition and the upper row is exact because of the exactness of the functor $(\cdot)^I$ in characteristic zero. The middle and the right vertical arrow are surjective because $V$ and its quotient $V/U$ are objects of the category $\Rep_R^I(G)$. By the snake lemma, the left vertical arrow is surjective, too.\\

This proves that $\Rep_R^I(G)$ is a full subcategory of $\Rep_R^\infty(G)$ which is stable under subquotients and hence is abelian. That the two functors are quasi-inverse to each other follows from our above reasoning and Lemma \ref{realization_0}. If $g:M\to N$ is an injective homomorphism of $H$-modules then the kernel of the induced map $f:X\otimes_HM\to X\otimes_HN$ is generated by its $I$-invariants. However, $\ker(f)^I=\ker(f^I)\cong\ker(g)=0$. Thus, $X\otimes_HM\to X\otimes_HN$ is injective and the right $H$-module $X$ is flat.
\end{proof}
Recall that for any field $R$ an object $V\in\Rep_R^\infty(G)$ is called admissible if the $R$-subspace $V^J$ of $J$-invariants is finite dimensional for any open subgroup $J$ of $G$.
\begin{cor}\label{admissible_finite_length}
Assume that $R$ is a field of characteristic zero. For any object $V\in\Rep_R^I(G)$ the following statements are equivalent.
\begin{enumerate}[wide]
 \item[(i)]$V$ is admissible.
 \item[(ii)]$V$ is of finite length.
 \item[(iii)]The $H$-module $V^I$ is of finite length.
 \item[(iv)]The $H$-module $V^I$ is finite dimensional over $R$.
\end{enumerate}
\end{cor}
\begin{proof}
This follows directly from Theorem \ref{Bernstein} together with the fact that an $H$-module is of finite length if and only if it is of finite $R$-dimension (cf.\ \cite{OS1}, Lemma 6.9)
\end{proof}
If $R$ is a field of characteristic zero then the functor $\F(\cdot)$ can be reinterpreted as follows. Note that this does not rely on Bernstein's Theorem \ref{Bernstein}.
\begin{thm}\label{description_functors_0}
Assume that $R$ is a field of characteristic zero.
\begin{enumerate}[wide]
 \item[(i)]There is an isomorphism $\F(\cdot)\cong\F_X\otimes_H(\cdot):\Mod_H\to\Coeff_G(\BT)$.
 \item[(ii)]There is an isomorphism $\mathrm{H}_0(\BT,\F(\cdot))\cong X\otimes_H(\cdot):\Mod_H\to\Rep_R^I(G)$.
 \item[(iii)]There is an isomorphism $\F((\cdot)^I)\cong\F_{(\cdot)}:\Rep_R^I(G)\to\Coeff_G(\BT)$.
\end{enumerate}
\end{thm}
\begin{proof}
As for (i), let $M\in\Mod_H$ and let $F$ be a face of $\BT$ contained in $\Cbar$. As in the proof of Theorem \ref{quasi_inverse_coefficient_system} one constructs a commutative diagram
\[
 \xymatrix{
 X^{I_F}\otimes_HM \ar[rr]^>>>>>>>>>>>{\tau_{M,F}} \ar[d]_\cong && \Hom_H(\Hom_H(X^{I_F},H),M) \ar[d]^\cong \\
 X_F\otimes_{H_F}M \ar[rr]_>>>>>>>>>>{\tau_{M,F}}&& \Hom_{H_F}(\Hom_{H_F} ( X_F,H_F),M)
 }
\]
by making use of Proposition \ref{I_F_invariants} (ii). We claim that $\tau_{M,F}$ is injective and need to prove this for its lower version only. The decomposition $X_F=H_F\oplus\ker(e_I)$ of right $H_F$-modules shows that $\tau_{M,F}$ induces a bijection on $I$-invariants. Since the category $\Rep_R(P_F/I_F)$ is semisimple, the kernel of $\tau_{M,F}$ is a quotient of $X_F\otimes_{H_F}M$ and hence is generated by its $I$-invariants. Thus, $\ker(\tau_{M,F})=0$ as claimed. Together with Proposition \ref{diagrams} it follows that the comparison homomorphism $\F_X\otimes_HM\to\F(M)$ in (\ref{comparison_homomorphisms}) is a natural isomorphism. This proves (i). As seen in the proof of Proposition \ref{comparison_representations} one obtains (ii) by passing to the homology in degree zero.\\

Now let $V\in\Rep_R^I(G)$ and consider the fixed point system $\F_V\in\Coeff_G(\BT)$. We continue to assume that $F$ is a face of $\BT$ with $F\subseteq\Cbar$. The natural surjection $X\otimes_HV^I\to V$ induces a surjection $(X\otimes_HV^I)^{I_F}\to V^{I_F}$ because $p$ is invertible in $R$. Using Proposition \ref{I_F_invariants} (ii) and the $H$-equivariant decomposition $X=X^{I_F}\oplus\ker(e_{I_F})$ this map can be identified with the natural map
\[
 X_F\otimes_{H_F}V^I\cong X^{I_F}\otimes_HV^I\cong(X\otimes_HV^I)^{I_F}\longrightarrow V^{I_F}.
\]
As before, it induces an isomorphism on $I$-invariants and hence is bijective because the category $\Rep_R(P_F/I_F)$ is semisimple. Since $X^{I_F}\otimes_HV^I\to V^{I_F}$ is the term at $F$ of the comparison homomorphism $\F_X\otimes_HV^I\to\F_V$ in (\ref{link_help}) it follows from Proposition \ref{diagrams} that the latter is an isomorphism.
\end{proof}
As a consequence, we can finally clarify the relation between the categories $\Rep_R^\C(G)$ and $\Rep_R^I(G)$. Moreover, we can reprove a special case of Schneider's and Stuhler's theorem concerning the exactness of oriented chain complexes of fixed point systems on $\BT$ (cf.\ \cite{SS}, Theorem II.3.1). We note that the strategy of our proof is due to Broussous who treated the analogous case of the Iwahori subgroup $I'$ of $G$ (cf.\ \cite{Bro}, \S4).
\begin{cor}\label{Schneider_Stuhler}
Assume that $R$ is a field of characteristic zero.
\begin{enumerate}[wide]
 \item[(i)]The categories $\Rep_R^\C(G)$ and $\Rep_R^I(G)$ coincide.
 \item[(ii)]For any representation $V\in\Rep_R^I(G)$ the augmented oriented chain complex $0\to \C_c^\ori(\BT_{(\bullet)},\F_V)\to V\to 0$ is exact.
\end{enumerate}
\end{cor}
\begin{proof}
Let $V\in\Rep_R^I(G)$. By Theorem \ref{Bernstein} and Theorem \ref{description_functors_0} (ii) we have $V\cong X\otimes_HV^I\cong\mathrm{H}_0(\BT,\F(V^I))$ which is an object of $\Rep_R^\C(G)$ by Theorem \ref{quasi_inverse_coefficient_system}. Using Proposition \ref{I_contains_C} this proves (i).\\

As for (ii), the exactness in degrees $-1$ and $0$ follows from $\mathrm{H}_0(\BT,\F_V)\cong X\otimes_HV^I\cong V$ (cf.\ Theorem \ref{Bernstein} and Theorem \ref{description_functors_0}). Since the functor $(\cdot)^I$ is exact Proposition \ref{acyclic} (i) implies that the higher homology groups of the augmented oriented chain complex of $\F_V$ have trivial $I$-invariants. Since these homology groups are objects of $\Rep_R^I(G)$ (cf.\ Proposition \ref{I_contains_C} and Theorem \ref{Bernstein}) it follows from Theorem \ref{Bernstein} that the homology groups are trivial.
\end{proof}
We would also like to point out that if $R$ is a field of characteristic zero then the equivalence in Theorem \ref{Bernstein} can be used to reinterprete the Zelevinski involution on $\Rep_R^I(G)$ and to reprove its major properties. Given a smooth $R$-linear left (resp.\ right) $G$-representation $V$ and a non-negative integer $i$ we consider the $R$-linear right (resp.\ left) $G$-representation
\[
 \E^i(V)=\mathrm{Ext}^i_{\Rep_R^\infty(G)}(V,\C_c^\infty(G,R)).
\]
Here $\C_c^\infty(G,R)$ denotes the $R$-module of compactly supported maps $G\to R$ endowed with its $G$-actions by left and right translation. In order to simplify the formulation of the following statements we identify the categories of left and right $G$-representations through the anti-automorphism $g\mapsto g^{-1}$ of $G$.
\begin{lem}\label{computing_Ext}
Assume that $R$ is a field of characteristic zero. If $V\in\Rep_R^I(G)$ admits a central character or if $\GG$ is semisimple then $\E^i(V)$ is the $i$-th homology group of the complex
\begin{equation}\label{dual_complex}
 \Hom_G(C_c^\ori(\BT_{(\bullet)},\F_V),\C_c^\infty(G,R)).
\end{equation}
If $V$ is admissible then this is a complex in $\Rep_R^I(G)$. In this case $\E^i(V)$ is an object of $\Rep_R^I(G)$ for any $i\geq 0$.
\end{lem}
\begin{proof}
The augmented oriented chain complex of $\F_V$ is a resolution of $V$ by Corollary \ref{Schneider_Stuhler} (ii). If $V$ admits a central character then it consists of projective objects of the category $\Rep_R^\infty(G)$ (cf.\ \cite{SS}, Proposition II.2.2). If $\GG$ is semisimple then this is true more generally because of (\ref{chain_complex_induced}) and by Frobenius reciprocity. Note that if $\GG$ is semisimple then the groups $P_F^\dagger$ are compact and the categories are $\Rep_R^\infty(P_F^\dagger)$ semisimple. Therefore, (\ref{dual_complex}) computes $\E^\bullet(V)$.\\

If $V$ is admissible then (\ref{chain_complex_induced}) implies that the oriented chain complex of $\F_V$ consists of finitely generated $G$-representations. Therefore, the $G$-action on (\ref{dual_complex}) is smooth. More precisely, for any term in the decomposition (\ref{chain_complex_induced}) there are $G$-equi\-variant isomorphisms
\begin{small}
\begin{eqnarray*}
 \Hom_G(\ind_{P_F^\dagger}^G(\varepsilon_F\otimes_RV^{I_F}),\C_c^\infty(G,R)) & \cong &
 \Hom_{P_F^\dagger}(\varepsilon_F\otimes_RV^{I_F},\C_c^\infty(I_F\backslash G,R))\\
 & \cong & [\ind_{I_F}^G(R)\otimes(\varepsilon_F\otimes_RV^{I_F})^*]^{P_F^\dagger}\\
 & \cong & \ind^G_{P_F^\dagger}((\varepsilon_F\otimes_RV^{I_F})^*).
\end{eqnarray*}
\end{small}As seen in the proof of Proposition \ref{description_functors_0} the $P_F^\dagger$-representation $V^{I_F}$ is generated by its $I$-invariants and so is $(V^{I_F})^*$ by Lemma \ref{properties_H} (iii). Moreover, $\varepsilon_F$ is trivial on $I$ by \cite{OS1}, Lemma 3.1. Therefore, (\ref{dual_complex}) is a complex in $\Rep_R^I(G)$. It follows from Theorem \ref{Bernstein} that so are its homology groups.
\end{proof}
The following proof of the Zelevinski conjecture for $\Rep_R^I(G)$ makes essential use of results of Ollivier and Schneider concerning the homological properties of the pro-$p$ Iwahori-Hecke algebra $H$ (cf.\ \cite{OS1}, \S6). In this case we can avoid the more sophisticated methods employed by Schneider and Stuhler (cf.\ \cite{SS}, \S{III}.3). We denote by $\Rep_R^I(G)^\adm$ the full subcategory of $\Rep_R^\infty(G)$ consisting of admissible representations generated by their $I$-invariants.
\begin{thm}\label{Zelevinski}
Assume that $R$ is a field of characteristic zero and that $\GG$ is semisimple. If $V\in\Rep_R^I(G)^\adm$ then $\E^i(V)=0$ unless $i=d$. Moreover, $\E^d(V)\in\Rep_R^I(G)^\adm$. The functor $\E^d:\Rep_R^I(G)^\adm\to\Rep_R^I(G)^\adm$ is an anti-involution of categories. In particular, it preserves irreducible objects.
\end{thm}
\begin{proof}
For any object $W\in\Rep_R^I(G)$ the isomorphism $W\cong X\otimes_HW^I$ of Theorem \ref{Bernstein} induces an isomorphism
\begin{eqnarray*}
 \Hom_G(W,\C_c^\infty(G,R)) & \cong & \Hom_G(X\otimes_HW^I,\C_c^\infty(G,R)) \\
 & \cong & \Hom_H(W^I,\Hom_G(X,\C_c^\infty(G,R))) \\
 & \cong & \Hom_H(W^I,\C_c^\infty(I\backslash G,R))
\end{eqnarray*}
of right $G$-representations. Note once more that $\C_c^\infty(I\backslash G,R)=H\oplus\ker(e_I)$ whence passage to the $I$-invariants yields an isomorphism of right $H$-modules
\[
 \Hom_G(W,\C_c^\infty(G,R))^I\cong\Hom_H(W^I,H).
\]
By the proof of Proposition \ref{description_functors_0} the $P_F$-representation $V^{I_F}$ is generated by its $I$-invariants for all $F\subseteq\Cbar$. By Proposition \ref{I_contains_C} the chain complex $\C_c^\infty(\BT_{(\bullet)},\F_V)$ is a complex in $\Rep_R^I(G)$. As seen in the proof of Lemma \ref{computing_Ext} it consists of projective objects. It follows from Proposition \ref{acyclic} (i), Theorem \ref{Bernstein} and Theorem \ref{description_functors_0} that $\C_c^\infty(\BT_{(\bullet)},\F_V)^I$ is a projective resolution of the $H$-module $V^I$. Together with Lemma \ref{computing_Ext} we obtain a functorial isomorphism
\[
 \E^i(V)^I\cong\mathrm{Ext}^i_H(V^I,H)
\]
of right $H$-modules for all $i\geq 0$. By \cite{OS1}, Theorem 6.16, the ring $H$ is Auslander-Gorenstein. Since $V^I$ is of finite length this implies $\E^i(V)^I=0$ for $i<d$ by \cite{OS1}, Corollary 6.17, and thus $\E^i(V)=0$ by Theorem \ref{Bernstein} and Lemma \ref{computing_Ext}. Moreover, the functor $\mathrm{Ext}^d_H(\cdot,H)$ is an equivalence between the categories of left and right $H$-modules of finite length, respectively, which is quasi-inverse to itself (cf.\ \cite{Iwa}, Theorem 8). Together with Theorem \ref{Bernstein} and Corollary \ref{admissible_finite_length} this proves the remaining statements.
\end{proof}
\begin{rem}\label{d_step_duality}
The arguments in the proof of Theorem \ref{Zelevinski} and \cite{ASZ}, Theorem 1.2, show more generally that the functors $(\E^i)_{0\leq i\leq d}$ induce a $(d+1)$-step duality on the full subcategory of $\Rep_R^I(G)$ consisting of all objects $V$ for which the $H$-module $V^I$ is finitely generated. This goes way beyond the case of admissible representations considered classically.
\end{rem}
If $R$ is a field of characteristic zero and if $M\in\Mod_H$ then the oriented chain complex $\C_c^\ori(\BT_{(\bullet)},\F(M))$ is acyclic (cf.\ Theorem \ref{Bernstein}, Theorem \ref{description_functors_0} and Corollary \ref{Schneider_Stuhler} (ii)). In general, its exactness properties remain an open question. If $p$ is nilpotent in $R$ we can at least treat the case of semisimple rank one.
\begin{lem}\label{I_invariants_non_trivial}
Assume that $p$ is nilpotent in $R$ and that $H$ is a pro-$p$ group. If $V\in\Rep_R^\infty(H)$ is non-zero then $V^H\not=0$.
\end{lem}
\begin{proof}
Let $v\in V\setminus\{0\}$. Since $p$ is nilpotent there is a positive integer $n$ with $p^nv=0$ and $p^{n-1}v\not=0$. Set $w=p^{n-1}v$ and let $W=R[H]\cdot w$ denote the subrepresentation of $V$ generated by $w$. Since $H$ is compact and the stabilizer of $w$ in $H$ is open, $W$ is an $\mathbb{F}_p$-vector space with a linear action of $H$ that factors through a finite $p$-group. By \cite{Ser}, \S8, Proposition 26, we have $W^H\not=0$ (cf.\ also \cite{Pas}, Lemma 2.1, for the case $R=\overline{\mathbb{F}}_p$).
\end{proof}
\begin{prop}\label{semisimple_rank_one}
Assume that $R$ is a quasi-Frobenius ring in which $p$ is nilpotent and that the semisimple rank of $\GG$ is equal to one. If $M\in\Mod_H$ then the augmented oriented chain complex
\[
0\longrightarrow \C_c^\ori(\BT_{(1)},\F(M))\longrightarrow \C_c^\ori(\BT_{(0)},\F(M))\longrightarrow \mathrm{H}_0(\BT,\F(M))\longrightarrow 0
\]
is exact and there is a natural $H$-linear injection $M\hookrightarrow \mathrm{H}_0(\BT,\F(M))^I$. In particular, the functor $(M\mapsto \mathrm{H}_0(\BT,\F(M))):\Mod_H\to\Rep_R^\infty(G)$ maps non-zero modules to non-zero $G$-representations.
\end{prop}
\begin{proof}
Since $\F(M)\in\C$ by Theorem \ref{quasi_inverse_coefficient_system}, the map $\C_c^\ori(\BT_{(1)},\F(M))^I\to \C_c^\ori(\BT_{(0)},\F(M))^I$ is injective by Proposition \ref{acyclic} (i). It follows from the left exactness of $(\cdot)^I$ and Lemma \ref{I_invariants_non_trivial} that the map $\C_c^\ori(\BT_{(1)},\F(M))\to \C_c^\ori(\BT_{(0)},\F(M))$ is injective, too. The remaining assertions are a consequence of Theorem \ref{quasi_inverse_coefficient_system}.
\end{proof}
\begin{rem}\label{non_admissible}
Recall from (\ref{chain_complex_induced}) that the oriented chain complex of $\F(M)$ is a finite direct sum of representations of the form $\ind_{P_F^\dagger}^G(\varepsilon_F\otimes_R\F(M)_F)$ for suitable faces $F\subseteq\Cbar$. If the underlying $R$-module of $M$ is finitely generated then so is the underlying $R$-module of $\varepsilon_F\otimes_R\F(M)_F$ because it is a quotient of $X_F\otimes_{H_F}M$. Under the assumptions of Proposition \ref{semisimple_rank_one} the $G$-representation $\mathrm{H}_0(\BT,\F(M))$ may therefore be called finitely presented as in \cite{Hu}, \S4.1. Now assume that $K$ is of characteristic $p$, $G=\GL_2(K)$,  $R=\overline{\mathbb{F}}_p$ and $M$ is simple and supersingular. Then \cite{Hu}, Corollaire 1.4, and \cite{OV}, Theorem 5.3, imply that the $G$-representation $\mathrm{H}_0(\BT,\F(M))$ cannot be admissible and irreducible. Moreover, Lemma \ref{I_invariants_non_trivial} implies that in this situation the inclusion $M\hookrightarrow \mathrm{H}_0(\BT,\F(M))^I$ has to be proper.
\end{rem}
Assume that $R=\overline{\mathbb{F}}_p$ with $p\not=2$ and that $G=\GL_2(\Qp)$ or $G=\mathrm{SL}_2(\Qp)$. By fundamental results of Ollivier and Koziol $(\cdot)^I:\Rep_R^I(G)\to\Mod_H$ and $X\otimes_H(\cdot):\Mod_H\to\Rep^I_R(G)$ are mutually inverse equivalences of categories (cf.\ \cite{Oll1}, Th\'eor\`eme 1.2 (a) and \cite{Koz}, Corollary 5.3 (1)). In part, this relies on exceptional flatness properties of the $H$-module $X$ (cf.\ \cite{Koz}, Corollary 4.10 (1)). The following proposition relies on analogous flatness properties of the finite universal Hecke modules $X_F$ over $H_F$. In exceptional cases we can thus generalize the results found in Theorem \ref{description_functors_0}. We content ourselves with formulating them for the groups $\mathrm{SL}_2$, $\GL_2$ and $\mathrm{PGL}_2$.
\begin{prop}\label{exceptional_flat}
Assume that $R$ is a quasi-Frobenius ring in which $p$ is nilpotent and that $G=\mathrm{SL}_2(K)$, $G=\GL_2(K)$ or $G=\mathrm{PGL}_2(K)$ with $|k|=p$.
\begin{enumerate}[wide]
\item[(i)]There is an isomorphism $\F(\cdot)\cong\F_X\otimes_H(\cdot):\Mod_H\to\Coeff_G(\BT)$.
\item[(ii)]The complexes $\C_c^\ori(\BT_{(\bullet)},\F(M))$ and $\C_c^\ori(\BT_{(\bullet)},\F_X)\otimes_HM$ are naturally isomorphic and acyclic for any $M\in\Mod_H$.
\item[(iii)]There is an isomorphism $\H_0(\BT,\F(\cdot))\cong X\otimes_H(\cdot):\Mod_H\to\Rep_R^\infty(G)$.
\end{enumerate}
\end{prop}
\begin{proof}
Note first that if $F$ is a face of $\BT$ contained in $\Cbar$ then the $H_F$-module $X_F$ is projective. If $F=C$ then this follows from $X_C=H_C\cong R[T_0/T_1]$. Now assume that $F$ is a vertex. The group $\GL_2(K)$ acts transitively on the set of vertices of $\BT$. Moreover, it acts on $\mathrm{SL}_2(K)$ by outer automorphisms. In all cases we may therefore assume $F=x_0$. There is an integer $n$ with $p^nR=0$. If $X_F'$ and $H_F'$ denote the corresponding objects defined over $\mathbb{Z}/p^n\mathbb{Z}$ then $H_F\cong H_F'\otimes_{\mathbb{Z}/p^n\mathbb{Z}}R$ and $X_F'\otimes_{H_F'}H_F\cong X_F'\otimes_{\mathbb{Z}/p^n\mathbb{Z}}R\cong X_F$. We may therefore assume $R=\mathbb{Z}/p^n\mathbb{Z}$.\\

If $G=\mathrm{SL}_2(K)$ then $P_F/I_F\cong\mathrm{SL}_2(k)$ and the assertion is proved in \cite{GK}, Lemma 2.2. Let us therefore assume $G=\GL_2(K)$ or $G=\mathrm{PGL}_2(K)$ whence $P_F/I_F\cong\GL_2(k)$ or $P_F/I_F\cong\mathrm{PGL}_2(k)$. By \cite{BAC}, III.5.4 Proposition 3, we may assume $n=1$. By faithfully flat base change we may further assume $R=\overline{\mathbb{F}}_p$. The case of $\GL_2(k)$ is then treated in \cite{OSe}, Th\'eor\`eme B (2) and Proposition 2.15. In order to deduce the case of $\mathrm{PGL}_2(k)$ let $H_F'$ and $X_F'$ denote the corresponding objects for $\GL_2(k)$. Viewing $X_F$ as a $\GL_2(k)$-representation via inflation the natural map $X_F'\otimes_{H_F'}H_F=X_F'\otimes_{H_F'}X_F^I\to X_F$ is bijective by \cite{OSe}, Th\'eor\`eme A and Th\'eor\`eme B (2). Since it is $H_F$-right linear the claim follows.\\

Conjugation as in (\ref{phi_gd}) and Proposition \ref{I_F_invariants} (ii) imply that the $H$-module $X^{I_F}$ is finitely generated projective for any face $F$ of $\BT$. Thus, all maps
\[
\tau_{M,F}: X^{I_F}\otimes_HM\longrightarrow\Hom_H(\Hom_H(X^{I_F},H),M)
\]
are injective. This proves part (i). Together with Proposition \ref{semisimple_rank_one}, part (ii) is an immediate consequence using the trivial relation $\C_c^\ori(\BT_{(\bullet)},\F_X\otimes_HM)=\C_c^\ori(\BT_{(\bullet)},\F_X)\otimes_HM$. By \cite{OS1}, Remark 3.2.1, the augmented oriented chain complex
\begin{equation}\label{universal_resolution}
0\longrightarrow \C_c^\ori(\BT_{(1)},\F_X)\longrightarrow \C_c^\ori(\BT_{(0)},\F_X)\longrightarrow X\longrightarrow 0
\end{equation}
is exact. Therefore, part (iii) follows from (ii) and the right exactness of the functor $(\cdot)\otimes_HM$.
\end{proof}


\subsection{Homotopy categories and their localizations}\label{subsection_4_2}

Over an arbitrary coefficient ring $R$, the functor $(\cdot)^I:\Rep_R^I(G)\to\Mod_H$ is faithful but not necessarily full. This concerns, for example, the case where $R$ is an algebraically closed field of characteristic $p$ and $G=\GL_2(K)$ with $K$ of characteristic zero and $q>p$ or $K=\mathbb{F}_p((T))$ with $p\not=2$ (cf.\ \cite{Oll1}, Th\'eor\`eme). However, one can single out an important full subcategory of $\Rep_R^I(G)$ on which the functor $(\cdot)^I$ is fully faithful. The following result generalizes \cite{Oll2}, Lemma 3.6.
\begin{prop}\label{invariants_induced_representation}
If $F$ is a face of $\BT$ with $F\subseteq\Cbar$ and if $V\in\Rep_R^\infty(P_F^\dagger)$ then the $P_F^\dagger$-equivariant inclusion $V\hookrightarrow\ind_{P_F^\dagger}^G(V)$ induces an isomorphism
\[
H\otimes_{H_F^\dagger}V^I\longrightarrow \ind_{P_F^\dagger}^G(V^{I_F})^I
\]
of $H$-modules given by $h\otimes m\mapsto h\cdot m$.
\end{prop}
\begin{proof}
By (\ref{Bruhat_decomposition}) and (\ref{Bruhat_parahoric}) we have $G=\coprod_{d\in D_F/\Omega_F}IdP_F^\dagger$. Fixing a system of represenatives of $D_F/\Omega_F$ the corresponding Hecke operators $\tau_d$ form a basis of the right $H_F^\dagger$-modules $H$ (cf.\ the proof of Proposition \ref{free} (ii)). Therefore, it suffices to see that $(m\mapsto\tau_d\cdot m):V^I\to\ind_{P_F^\dagger}^G(V)^I$ maps isomorphically onto the $R$-submodule $W_d$ of $I$-invariant elements with support $IdP_F^\dagger$ and values in $V^{I_F}$ for any $d$.\\

It follows from (\ref{convolution}) that $\tau_d\cdot m$ is the function with support $IdP_F^\dagger$ and values $(\tau_d\cdot m)(idp)=p^{-1}m$ for all $i\in I$ and $p\in P_F^\dagger$. Note that this is an element of $V^{I_F}$ because $I_F\subseteq I$ and $I_F$ is a normal subgroup of $P_F^\dagger$. Therefore, the map $(m\mapsto\tau_d\cdot m):V^I\to W_d$ is well-defined and injective.\\

On the other hand, let $f\in W_d$. Then $f(idp)=p^{-1}f(d)$ where $f(d)$ is an element of $V$ which is invariant under $(d^{-1}Id\cap P_F)I_F=I$. Here the last equation results from Lemma \ref{closest_chamber} and \cite{OS1}, Proposition 4.13 (i). Thus, $f=\tau_d\cdot f(d)$ and the map in question is surjective.
\end{proof}
If $V\in\Rep_R^\infty(G)$ and if $F$ is a face of $\BT$ then we may view $V$ and $V^{I_F}$ as objects of $\Rep_R^\infty(P_F^\dagger)$ via restriction.
\begin{prop}\label{invariants_induction_H}
Let $R$ be quasi-Frobenius ring and let $F$ and $F'$ be faces of $\BT$ contained in $\Cbar$.
\begin{enumerate}[wide]
\item[(i)]If $V\in\Rep_R^H(P_{F'})$ then $\ind_{P_{F'}}^G(V)^{I_F}$ is an object of $\Rep_R^H(P_F)$.
\item[(ii)]If $V\in\Rep_R^H(P_{F'}^\dagger)$ then $\ind_{P_{F'}^\dagger}^G(V)^{I_F}$ is an object of $\Rep_R^H(P_F^\dagger)$.
\end{enumerate}
\end{prop}
\begin{proof}
We first show that (ii) follows from (i). Indeed, the Bruhat decompositions (\ref{Bruhat_decomposition}) and (\ref{Bruhat_parahoric}) show that $\Omega/\Omega_{F'}$ is a system of representatives of the double cosets $G_\aff\backslash G/P_{F'}^\dagger$. Therefore, we have the $G_\aff$-equivariant Mackey decomposition
\[
\ind_{P_{F'}^\dagger}^G(V)\cong\bigoplus_{\omega\in\Omega/\Omega_{F'}}\ind_{G_\aff\cap P_{\omega F'}^\dagger}^{G_\aff}(V^\omega).
\]
Here $\omega F'\subseteq\Cbar$ and $G_\aff\cap P_{\omega F'}^\dagger=P_{\omega F'}$ by \cite{OS1}, Lemma 4.10. Now $V\in\Rep_R^H(P_{F'}^\dagger)$ implies $V^\omega\in \Rep_R^H(P_{\omega F'})$. A second suitable Mackey decomposition shows that the $P_F$-representation $\ind_{P_{\omega F'}}^{G_\aff}(V^\omega)$ is a direct summand of $\ind_{P_{\omega F'}}^G(V^\omega)$. By Corollary \ref{injective_surjective} (ii) it suffices to prove (i).\\

Choose vertices $x$ and $y$ of $\BT$ which are contained in $\overline{F}$ and $\overline{F}'$, respectively. Then $\ind_{P_{F'}}^G(V)\cong\ind_{P_y}^G(\ind_{P_{F'}}^{P_y}(V))$ where $\ind_{P_{F'}}^{P_y}(V)$ satisfies condition (H) as a representation of $P_y$ according to Lemma \ref{induction_H} (ii). Without loss of generality we may therefore assume $F'=y$. Further, by Proposition \ref{invariants_H} it suffices to show that $\ind_{P_y}^G(V)^{I_x}$ satisfies condition (H) as a representation of $P_x$, i.e.\ we may assume $F=x$.\\

It follows from the decompositions (\ref{Bruhat_decomposition}) and (\ref{Bruhat_parahoric}) together with the braid relations (\ref{braid_relations}) that $G=\coprod_{d\in W_x\backslash W/W_y}P_xdP_y$ (cf.\ also \cite{BT}, Proposition 7.4.15). Consequently, we have the $P_x$-equivariant Mackey decomposition
\[
 \ind_{P_y}^G(V)\cong\bigoplus_{d\in W_x\backslash W/W_y}\ind_{P_x\cap dP_yd^{-1}}^{P_x}(V^d)
\]
where the group $P_x\cap dP_yd^{-1}$ acts on the $R$-vector space $V^d=V$ via $g\cdot v=d^{-1}gd\cdot v$. We identify a fixed double coset $d\in W_x\backslash W/W_y$ with its unique representative of minimal length in $W$ (cf.\ \cite{BGL}, Chapitre IV, \S1, Exercise 3, applied to the Coxeter group $W^\aff$, noting that $W_x,W_y\subseteq W_\aff$ and $W\cong W^\aff\rtimes \Omega$ with $\ell(w \omega)=\ell(w)$ for all $w\in W^\aff$ and $\omega\in\Omega$). We may thus assume $d\in D_y$ and at the same time $d^{-1}\in D_x$. We need to see that the $P_x$-representation
\[
 \ind_{P_x\cap dP_yd^{-1}}^{P_x}(V^d)^{I_x} \cong
 \ind_{(P_x\cap dP_yd^{-1})I_x}^{P_x} ((V^d)^{I_x\cap dP_yd^{-1}})
\]
satisfies condition (H). Here $(V^d)^{I_x\cap dP_yd^{-1}}$ is viewed as a representation of $(P_x\cap dP_yd^{-1})I_x$ through the homomorphism
\[
(P_x\cap dP_yd^{-1})I_x\to(P_x\cap dP_yd^{-1})I_x/I_x\cong (P_x\cap dP_yd^{-1})/(I_x\cap dP_yd^{-1}).
\]
Clearly, we may assume that the underlying $R$-module of $V$ is finitely generated. By Lemma \ref{closest_chamber} and \cite{OS1}, Proposition 4.13 (i), we have $(d^{-1}Id\cap P_y)I_y=(d^{-1}Id\cap I)I_y=I$ whence $I'\subseteq (d^{-1}P_xd\cap P_y)I_y\subseteq P_y$. Since $P_y/I_y$ is a group with a $(B,N)$-pair in which $B=I'/I_y$ this implies that
\[
 (d^{-1}P_xd\cap P_y)I_y=P_F
\]
for a face $F$ of $\BT$ with $y\in\overline{F}\subseteq\Cbar$ (cf.\ \cite{BT2}, Th\'eor\`eme 4.6.33). Note that we give a new meaning to the symbol $F$ here which also appears in the formulation of the proposition. As in \cite{BT}, (7.1.2), any element of $\Phi_\aff$ determines a closed half space of $\A$ which is again called an affine root. With this terminology we let $\mathrm{cl}(d^{-1}x,y)$ denote the intersection of all affine roots containing $\{d^{-1}x,y\}$ and claim that $\overline{F}=\mathrm{cl}(d^{-1}x,y)\cap\Cbar$. In order to prove this, note that $\overline{F}$ is precisely the set of fixed points of $P_F$ in $\A$ (cf.\ \cite{BT2}, Corollaire 4.6.29 (i)).\\

Setting $Y=\mathrm{cl}(d^{-1}x,y)\cap\Cbar$, any element of $I_y\subseteq I$ fixes $\Cbar$ pointwise and hence $Y$. Let $Z=\mathbb{C}(K)$ denote the group of $K$-rational points of the connected center $\mathbb{C}$ of $\GG$. In the notation of \cite{BT}, Th\'eor\`eme 6.5, we have $G'=ZG_\aff$. Therefore, it follows from \cite{OS1}, Lemma 4.10, and \cite{BT}, (4.1.1), that the group $d^{-1}P_xd\cap P_y$ is the pointwise stabilizer of $\mathrm{cl}(d^{-1}x,y)\supseteq Y$ in $G^\aff$. Consequently, $P_F$ fixes $Y$ pointwise and $Y\subseteq\overline{F}$ as explained above. Conversely, any point of $\overline{F}$ lies in $\Cbar$ and is fixed by $d^{-1}P_xd\cap P_y$ because $d^{-1}P_xd\cap P_y\subseteq P_F$. By \cite{BT2}, Proposition 4.6.24 (i), the group $d^{-1}P_xd\cap P_y$ contains the group denoted by $\mathfrak{G}^0_{\mathrm{cl}(d^{-1}x,y)}(\o)$ in \cite{BT2}, \S4.6.26. By \cite{BT2}, Corollaire 4.6.29 (i), we get $\overline{F}\subseteq \mathrm{cl}(d^{-1}x,y)$ and thus $\overline{F}\subseteq Y$.\\

Next we claim that $I_F=(d^{-1}I_xd\cap I)I_y$. Note that $I_y$ is a normal subgroup of $P_y$ and hence of $P_F$. As seen above $d^{-1}I_xd\cap I=d^{-1}I_xd\cap P_y$ because $I_x\subseteq I$. Since $d^{-1}I_xd\cap P_y$ a normal subgroup of $d^{-1}P_xd\cap P_y$ we obtain that $(d^{-1}I_xd\cap I)I_y$ is a normal subgroup of $P_F$. Since it is contained in $I$ it is contained in the pro-$p$ radical of $P_F$, i.e.\ in $I_F$. In order to prove the reverse inclusion we have to make a digression into the theory of the Iwahori decomposition. For any real number $r$ we denote by $r+$ the smallest integer strictly greater than $r$. For any face $F'$ of $\BT$ and any root $\alpha\in\Phi$ we set
\begin{eqnarray*}
 f_{F'}(\alpha)&=&-\inf\{\alpha(x)\;|\;x\in F'\} \quad\mbox{and}\\[1ex]
 f^*_{F'}(\alpha)&=&\left\{
 \begin{array}{ll}
  f_{F'}(\alpha)+,&\mbox{ if }\alpha|F'\mbox{ is constant,}\\
  f_{F'}(\alpha),&\mbox{ otherwise.}
 \end{array}
 \right.
\end{eqnarray*}
Recall that the root subgroup $U_\alpha$ of $G$ corresponding to $\alpha$ admits an exhaustive, separated and decreasing filtration by subgroups $U_{\alpha,r}$ with $r\in\mathbb{R}$ as in \cite{SS}, page 103. Since the group $\GG$ is split the jumps of this filtration are precisely the integers. According to \cite{SS}, Proposition I.2.2, the group $I_F$ is generated by $T_1$ and the groups $U_{\alpha,f_F^*(\alpha)}$ with $\alpha\in\Phi$. To prove our claim it remains to see that $U_{\alpha,f_F^*(\alpha)}$ is contained in $(d^{-1}I_xd\cap I)I_y$ for all $\alpha\in\Phi$. Note that we always have $f_y^*(\alpha)=-\alpha(y)+$.\\

If $\alpha|F$ is constant then $f_F^*(\alpha)=-\alpha(y)+$ because $y\in\overline{F}$. In this case, $U_{\alpha,f_F^*(\alpha)}=U_{\alpha,-\alpha(y)+}\subseteq I_y$. So assume that $\alpha|F$ is not constant and hence that $f_F^*(\alpha)=f_F(\alpha)\geq f_y(\alpha)=-\alpha(y)$. If $\alpha(y)\not\in\mathbb{Z}$ then again $U_{\alpha,f_F^*(\alpha)}\subseteq U_{\alpha,-\alpha(y)}=U_{\alpha,-\alpha(y)+}\subseteq I_y$. We may therefore assume that $\alpha(y)\in\mathbb{Z}$. If $f_F(\alpha)>-\alpha(y)$ and if $s=\min\{n\in\mathbb{Z}\;|\;n\geq f_F(\alpha)\}$ then $s\geq-\alpha(y)+1$ and hence $
U_{\alpha,f_F^*(\alpha)}=U_{\alpha,f_F(\alpha)}=U_{\alpha,s}\subseteq U_{\alpha,-\alpha(y)+1}\subseteq I_y$. We may therefore assume $f_F^*(\alpha)=f_F(\alpha)=-\alpha(y)$, i.e.\ $\alpha(z)\geq\alpha(y)$ for all $z\in F$. In this situation, $\alpha(d^{-1}x)>\alpha(y)$ because otherwise $\mathrm{cl}(d^{-1}x,y)$ and hence $F$ would be contained in the affine root $\{z\in\A\;|\;\alpha(z)\leq\alpha(y)\}$. However, this would imply $\alpha(z)=\alpha(y)$ for all $z\in F$ in contradiction to our assumption that $\alpha|F$ is not constant. Writing $d=(\lambda,w)\in W\cong X_*(\TT)\rtimes W_0$ we have
\[
\alpha(y)<\alpha(d^{-1}x)=\alpha(w^{-1}x+\lambda)=w\alpha(x)+\alpha(\lambda),
\]
whence $-w\alpha(x)+\leq-\alpha(y)-\alpha(\lambda)$ because $\alpha(y)+\alpha(\lambda)\in\mathbb{Z}$. Therefore,
\begin{eqnarray*}
dU_{\alpha,f_F^*(\alpha)}d^{-1} &=& dU_{\alpha,-\alpha(y)}d^{-1} =U_{w\alpha,-\alpha(y)-\alpha(\lambda)}\\
&\subseteq& U_{w\alpha,-w\alpha(x)+}\subseteq I_x
\end{eqnarray*}
which implies $U_{\alpha,f_F^*(\alpha)}\subseteq d^{-1}I_xd\cap I_F\subseteq d^{-1}I_xd\cap I$. This completes the proof that $I_F=(d^{-1}I_xd\cap I)I_y=(d^{-1}I_xd\cap P_y)I_y$.\\

Since $d^{-1}\in D_x$ the same arguments show that $(P_x\cap dP_yd^{-1})I_x=P_{F'}$ for some face $F'$ of $\BT$ with $x\in\overline{F'}\subseteq\Cbar$ such that $I_{F'}=(I\cap dI_yd^{-1})I_x$ and $(I\cap dId^{-1})I_x=(I\cap dP_yd^{-1})I_x=I$. Again the notation is not to be confused with the meaning of $F'$ in the initial formulation of the proposition. To complete the proof it suffices to see that the $P_{F'}$-representation $(V^d)^{I_x\cap dP_yd^{-1}}$ satisfies condition (H) (cf.\ Lemma \ref{induction_H} (ii)). We claim it is generated by its $I$-invariants. By definition of the $P_{F'}$-action
\[
((V^d)^{I_x\cap dP_yd^{-1}})^I=(V^d)^I=((V^d)^{(I\cap dId^{-1})I_x}=(V^d)^{I\cap dId^{-1}},
\]
and the claim is equivalent to $V^{d^{-1}I_xd\cap P_y}$ being generated by $V^{d^{-1}Id\cap I}$ as a representation of $d^{-1}P_xd\cap P_y$. Since $I_y$ acts trivially on $V$ we have $V^{d^{-1}I_xd\cap P_y}=V^{(d^{-1}I_xd\cap P_y)I_y}=V^{I_F}$ and $V^{d^{-1}Id\cap I}=V^{(d^{-1}Id\cap I)I_y}=V^I$. Since $(d^{-1}P_xd\cap P_y)I_y=P_F$ the claim follows from Lemma \ref{properties_H} (i) and Proposition \ref{invariants_H}. Likewise, $((V^d)^{I_x\cap dId^{-1}})^*=((V^{I_F})^d)^*=((V^{I_F})^*)^d$ is generated by its $I$-invariants as a representation of $P_{F'}$ because this is true of the $P_F$-representation $(V^{I_F})^*$. By Lemma \ref{properties_H} (ii) the $P_{F'}$-representation $(V^d)^{I_x\cap dId^{-1}}$ satisfies condition (H).
\end{proof}
\begin{rem}\label{subtle_set_of_roots}
We take up the notation of the proof of Proposition \ref{invariants_induction_H}. If $x=y=x_0$ then the fact that $(d^{-1}P_xd\cap P_y)I_y=P_F$ for some face $y\in\overline{F}\subseteq\Cbar$ with $I_F=(d^{-1}Id\cap I)I_y$ is also proved in \cite{Her}, Proposition 3.8, and \cite{Oll2}, \S3.3 Fact 2. In this case the set $\Phi_F$ of affine roots vanishing on $F$ is given by $\Phi_F=\Phi_y\cap d^{-1}\Phi_x$. We note that the latter formula is wrong in the general situation considered above. For example, if $\Phi=A_2$, if $x=y\not=x_0$ and if $d^{-1}=\lambda_1+\lambda_2$ is the sum of the two fundamental dominant cocharacters then $\Phi_y\cap d^{-1}\Phi_x=\emptyset$ whereas $\overline{F}=\mathrm{cl}(d^{-1}x,y)\cap\Cbar=\{y\}$. Thus, $\Phi_F=\Phi_y$.
\end{rem}
\begin{defin}\label{reductive_ind}
\begin{enumerate}[wide]
\item[(i)]Let $\Rep_R^\ind(G)$ denote the full subcategory of $\Rep_R^\infty(G)$ consisting of all representations which are isomorphic to finite direct sums of representations of the form $\ind_{P_F^\dagger}^G(V_F)$ for some face $F$ of $\BT$ contained in $\Cbar$ and some $P_F^\dagger$-representation $V_F$ satisfying condition (H).
\item[(ii)]We denote by $\Mod_H^\ind$ the full subcategory of $\Mod_H$ consisting of all modules which are isomorphic to finite direct sums of modules of the form $H\otimes_{H_F^\dagger}M_F$ for some face $F$ of $\BT$ contained in $\Cbar$ and some $H_F^\dagger$-module $M_F$.
\end{enumerate}
\end{defin}
We continue to denote by $\Rep_R^I(G)$ the full subcategory of $\Rep_R^\infty(G)$ consisting of all objects generated by their $I$-invariants. It follows from Lemma \ref{properties_H} (i) that $\Rep_R^\ind(G)$ is a full subcategory of $\Rep_R^I(G)$.
\begin{thm}\label{equivalence_induced_objects}
If $R$ is a quasi-Frobenius ring then the functor $(\cdot)^I:\Rep_R^\ind(G)\to\Mod_H^\ind$ is an equivalence of additive categories.
\end{thm}
\begin{proof}
The faithfulness follows from the fact that $\Rep_R^\ind(G)$ is a full subcategory of $\Rep_R^I(G)$. The essential surjectivity is a direct consequence of Theorem \ref{Cabanes_general} (ii) and Proposition \ref{invariants_induced_representation}. Let $F$ (resp.\ $F'$) be a face of $\BT$ contained in $\Cbar$, and let $V$ (resp.\ $W$) be a representation of $P_F^\dagger$ (resp.\ $P_{F'}^\dagger$) satisfying condition (H). By Remark \ref{H_I_F_trivial}, Theorem \ref{Cabanes_general} (ii), Proposition \ref{invariants_induced_representation} and Proposition \ref{invariants_induction_H} we have
\begin{eqnarray*}
\Hom_G(\ind_{P_F^\dagger}^G(V),\ind_{P_{F'}^\dagger}^G(W)) & \cong &
\Hom_{P_F^\dagger}(V,\ind_{P_{F'}^\dagger}^G(W)^{I_F})\\
& \cong & \Hom_{H_F^\dagger}(V^I,\ind_{P_{F'}^\dagger}^G(W)^I)\\
& \cong & \Hom_H(H\otimes_{H_F^\dagger}V^I,\ind_{P_{F'}^\dagger}^G(W)^I)\\
& \cong & \Hom_H(\ind_{P_F^\dagger}^G(V)^I,\ind_{P_{F'}^\dagger}^G(W)^I).
\end{eqnarray*}
Unwinding definitions, this is precisely the map induced by $(\cdot)^I$.
\end{proof}
\begin{rem}\label{fully_faithful_hyperspecial}
Assume that $R=\overline{\mathbb{F}}_p$. If $F=F'=x_0$ and if $V$ and $W$ are irreducible then the above chain of isomorphisms already appears in \cite{Oll2}, Corollary 3.14 (i) and \cite{Vig5}, Proposition 7.5.
\end{rem}
If $R$ is a field of characteristic zero then Theorem \ref{Schneider_Stuhler} shows that any representation $V\in\Rep_R^I(G)$ admits a finite resolution by objects of $\Rep_R^\ind(G)$. If $R=\overline{\mathbb{F}}_p$ and if $K$ is of characteristic $p$ then this is generally no longer true (cf.\ \cite{Hu}, Corollaire 5.5). Consequently, even on a derived level $\Rep_R^\ind(G)$ may be a rather small subcategory of $\Rep_R^I(G)$. In contrast, if $R$ is a quasi-Frobenius ring then the categories $\Mod_H^\ind$ and $\Mod_H$ are always derived equivalent in a suitable sense.\\

In order to make this precise, let $\D$ be any additive category and denote by $\mathrm{K}^b\D$ the homotopy category of bounded complexes of $\D$. Note that if $\D'$ is a full additive subcategory of $\D$ then $\mathrm{K}^b\D'$ is naturally a triangulated subcategory of $\mathrm{K}^b\D$, i.e.\ a full additive subcategory for which the inclusion functor is an exact functor of triangulated categories (cf.\ \cite{Sta}, Proposition 13.10.3 and Lemma 13.10.6).\\

In our situation, let $\Sigma'$ (resp.\ $\Sigma$) denote the class of quasi-isomorphisms in $\mathrm{K}^b\Mod_H^\ind$ (resp.\ $\mathrm{K}^b\Mod_H$), i.e.\ the class of morphisms inducing isomorphisms on all homology groups. It is known that $\Sigma'$ and $\Sigma$ are multiplicative systems and that the corresponding localizations are triangulated categories in a natural way (cf.\ \cite{Sta}, Lemma 13.5.4 and Proposition 13.5.5). Note that by \cite{Sta}, Lemma 13.11.6 (3), the localization
\[
 \mathrm{K}^b\Mod_H[\Sigma^{-1}]\cong\mathrm{D}^b(H)
\]
is triangle equivalent to the bounded derived category $\mathrm{D}^b(H)$ of $\Mod_H$.
\begin{prop}\label{module_homotopy_equivalence}
If $R$ is a quasi-Frobenius ring then the functor
\[
\mathrm{K}^b\Mod_H^\ind[(\Sigma')^{-1}]\to\mathrm{D}^b(H)
\]
induced by the inclusion functor $\mathrm{K}^b\Mod_H^\ind\to\mathrm{K}^b\Mod_H$ is an equivalence of triangulated categories. On the full subcategory $\Mod_H$ of $\mathrm{D}^b(H)$ a quasi-inverse is given by assigning to $M\in\Mod_H$ the complex $\C_c^\ori(\BT_{(\bullet)},\F(M))^I$.
\end{prop}
\begin{proof}
Arguing dually to \cite{KS}, Corollary 7.2.2, it suffices to see that for any bounded complex $M_\bullet$ of $H$-modules there is a bounded complex $N_\bullet$ of objects of $\Mod_H^\ind$ and a quasi-isomorphism $N_\bullet\to M_\bullet$.\\

In order to construct $N_\bullet$ note first that if $M\in\Mod_H$ then the complex $\C_c^\ori(\BT_{(\bullet)},\F(M))^I$ is an acyclic complex of objects of $\Mod_H^\ind$ whose homology in degree zero is naturally isomorphic to $M$. This follows from Proposition \ref{acyclic} (i), Theorem \ref{quasi_inverse_coefficient_system}, Proposition \ref{invariants_induced_representation} and the decomposition (\ref{chain_complex_induced}). In other words, we can augment the above complex to an exact resolution $0\to\C_c^\ori(\BT_{(\bullet)},\F(M))^I\stackrel{\varepsilon_M}{\longrightarrow}M\longrightarrow 0$ of $M$.\\

Now let $M_\bullet$ be a bounded complex of $H$-modules. We consider the bounded double complex $C_{\bullet,\bullet}=\C_c^\ori(\BT_{(\bullet)},\F(M_\bullet))^I$ and its subcomplex $C'_{\bullet,\bullet}$ obtained by replacing the row $0\to\C^\ori_c(\BT_{(0)},\F(M_0))^I\to\cdots\to\C^\ori_c(\BT_{(0)},\F(M_n))^I\to 0$ by the row $0\to\ker(\varepsilon_{M_0})\to\cdots\to\ker(\varepsilon_{M_n})\to 0$. Denoting by $N_\bullet=\mathrm{Tot}(C_{\bullet,\bullet})$ and $N_\bullet'=\mathrm{Tot}(C'_{\bullet,\bullet})$ the corresponding total complexes we obtain an exact sequence $0\to N'_\bullet\to N_\bullet\to M_\bullet\to 0$ of complexes of $H$-modules in which $N_\bullet$ is a complex over $\Mod_H^\ind$. Since the columns of $C'_{\bullet,\bullet}$ are exact, so is $N'_\bullet$ by the usual spectral sequence argument. The long exact homology sequence then shows that $N_\bullet\to M_\bullet$ is a quasi-isomorphism.
\end{proof}
\begin{rem}\label{homotopy_equivalence_OS}
The only non-formal part of the proof of Proposition \ref{module_homotopy_equivalence} concerned the essential surjectivity. We note once more that if $R$ is a field then the necessary input from \S\ref{subsection_3_2} can be replaced by the results of \cite{OS1}, \S6, on canonical Gorenstein projective resolutions (cf.\ also Remark \ref{GP_resolution}).
\end{rem}
The additive functor $(\cdot)^I:\Rep_R^\ind(G)\to\Mod_H$ induces an exact triangle functor $\mathrm{K}^b\Rep_R^\ind(G)\to \mathrm{K}^b\Mod_H$ (cf.\ \cite{Sta}, Lemma 13.10.6) that we continue to denote by $(\cdot)^I$. Let $\Sigma''$ denote the class of all morphisms $f$ in $\mathrm{K}^b\Rep_R^\ind(G)$ such that $f^I$ is a quasi-isomorphism.
\begin{thm}\label{homotopy_equivalence}
If $R$ is a quasi-Frobenius ring then $\Sigma''$ is a multiplicative system and the functor
\[
 \mathrm{K}^b\Rep_R^\ind(G)[(\Sigma'')^{-1}]\longrightarrow\mathrm{D}^b(H)
\]
induced by $(\cdot)^I:\Rep_R^\ind(G)\to\Mod_H$ is an equivalence of triangulated categories. On the full subcategory $\Mod_H$ of $\mathrm{D}^b(H)$ a quasi-inverse is given by assigning to $M\in\Mod_H$ the oriented chain complex $\C_c^\ori(\BT_{(\bullet)},\F(M))$ of the $G$-equivariant coefficient system $\F(M)$ on $\BT$.
\end{thm}
\begin{proof}
This follows directly from Theorem \ref{equivalence_induced_objects} and Proposition \ref{module_homotopy_equivalence}.
\end{proof}
We note that if $R$ is a field of characteristic $p$ and if $I$ is $p$-torsion free then a fundamental result of Schneider shows that the unbounded derived category of $\Rep_R^\infty(G)$ is equivalent to the unbounded derived category of DG-modules over a certain DG-version of $H$ (cf.\ \cite{Sch}, Theorem 9). We point out that the equivalence in Theorem \ref{equivalence_induced_objects} is generally not compatible with the homological properties of the two categories. Therefore, it is currently unclear how Theorem \ref{homotopy_equivalence} relates to Schneider's result.


\subsection{The functor to generalized $(\varphi,\Gamma)$-modules}\label{subsection_4_3}

Let $\mathbb{P}$ and $\overline{\mathbb{P}}$ be the Borel subgroups of $\GG$ corresponding to $\Phi^+$ and $\Phi^-$, respectively, and let $\mathbb{U}$ and $\overline{\mathbb{U}}$ denote their unipotent radicals. Setting $P=\mathbb{P}(K)$, $\Pbar=\overline{\mathbb{P}}(K)$, $U=\mathbb{U}(K)$ and $\Ubar=\overline{\mathbb{U}}(K)$ we have the Levi decompositions $P=U\rtimes T$ and $\Pbar=\Ubar\rtimes T$. We let $U_0=U\cap I$, $\Ubar_1=\Ubar\cap I$ so that $I=\Ubar_1T_1 U_0$ by \cite{SS}, Proposition I.2.2. Consider the submonoid $T^+$ of $T$ defined by
\begin{eqnarray*}
T^+ & = & \{t\in T\;|\;\forall\alpha\in\Phi^+:\langle\alpha,\nu(t)\rangle\geq 0\}\\
& = & \{t\in T\;|\;t\Ubar_1t^{-1}\subseteq \Ubar_1\}\\
& = &\{t\in T\;|\;t^{-1}U_0t\subseteq U_0\}.
\end{eqnarray*}
\begin{rem}\label{antidominant}
For $\alpha\in\Phi$ let $U_\alpha$ be the corresponding root subgroup with its filtration by subgroups $U_{\alpha,r}$ as in \cite{SS}, \S I.1. The last two descriptions of $T^+$ follow from $tU_{\alpha,r}t^{-1}=U_{\alpha,r-\langle\alpha,\nu(t)\rangle}$ for all $t\in T$, $\alpha\in\Phi$ and $r\in\mathbb{R}$. Recall that the homomorphism $\nu:T\to X_*(\TT/\mathbb{C})$ is normalized through $\langle\alpha,\nu(t)\rangle=-\mathrm{val}(\alpha(t))$ for all $\alpha\in\Phi$. Thus, $\nu(T^+)=X^+(\TT/\mathbb{C})$ is the set of dominant cocharacters with respect to the chosen set $\Phi^+$ of positive roots. The literature contains other normalizations for which $T^+$ is defined as the submonoid of $T$ consisting of all elements $t$ contracting $U_0$, i.e.\ such that $\mathrm{val}(\alpha(t))\geq 0$ for all $\alpha\in\Phi^+$. In our notation this would be $(T^+)^{-1}$.
\end{rem}
Note that $G^+=IT^+I$ is a submonoid of $G$. Indeed, if $t,t'\in T^+$ and if $\lambda,\lambda'$ denote the respective images in $T/T_0\subseteq W$ then $\ell(\lambda\lambda')=\ell(\lambda)+\ell(\lambda')$ by \cite{Vig4}, Example 5.12. By the braid relations (\ref{braid_relations}) this implies $ItI\cdot It'I=Itt'I=It'I\cdot ItI$. We let $H^+$ be the commutative subalgebra of $H$ consisting of all maps supported on $G^+$.\\

The monoid $G^+$ contains the submonoid $\Pbar^+=\Ubar_1T^+$ of $\Pbar$. Denote by $\Cscr^0\subseteq \A$ the closure of the vector chamber with apex $x_0$ containing $C$ (cf.\ \cite{BT}, (1.3.10)). Given any closed vector chamber $\Cscr$ contained in $\Cscr^0$ we let $\BT^+(\Cscr)=G^+\Cscr $, viewed as a subcomplex of $\BT$. We write $\BT^+=\BT^+(\Cscr^0)$, for short. Note that $\Cscr^0$ is the convex envelope of $T^+x_0$ in $\A$, whence $\Cscr^0$ and $\Cscr$ are stable under $T^+$. Moreover, since $U_0\subseteq tU_0t^{-1}$ fixes $tx_0$ for any $t\in T^+$ it follows from \cite{BT}, Proposition 2.5.4 (iii), that $U_0$ fixes $\Cscr^0$ and hence $\Cscr$ pointwise. Since $G^+=IT^+I=\Ubar_1T^+U_0$ we obtain $\BT^+(\Cscr)=\Pbar^+\Cscr=\Ubar_1\Cscr$.
\begin{rem}\label{half_tree}
If the semisimple rank of $\GG$ is equal to one then $\BT$ is a tree and any $\BT^+(\Cscr)$ is a closed half tree as considered in \cite{GK}, \S3.
\end{rem}
Let $F$ be a face of $\BT$ contained in $\Cscr^0$ and let $C(F)\subseteq\A$ be the chamber associated to $F$ as in Lemma \ref{closest_chamber}. We claim that $C(F)$ is contained in $\Cscr^0$. Otherwise, there would be a root $\alpha\in\Phi^+$ with $\alpha(C(F))<0$. Let $(C_0,\ldots,C_n)$ be a minimal gallery connecting $C_0=C$ and $C_n=C(F)$. Since $F\subseteq \overline{C(F)}\cap\Cscr^0$ we get $\alpha(F)=0$. Moreover, $\alpha(C)>0$ implies that there is an index $i$ such that $C_i$ and $C_{i+1}$ are separated by the wall determined by $\alpha$. If $s_\alpha\in W$ denotes the corresponding reflection then $s_\alpha C_{i+1}=C_i$ and $(C_0,\ldots,C_i,s_\alpha C_{i+2},\ldots,s_\alpha C_n)$ is a gallery with $F=s_\alpha F\subseteq s_\alpha \overline{C_n}$. This contradicts the minimality property of $C_n=C(F)$. Therefore, $C(F)\subseteq \Cscr^0$.\\

Due to this result we obtain the following variants of Propositions \ref{restriction_chain_complexes} and \ref{acyclic}. Given a coefficient system $\F\in\Coeff_G^0(\BT)$ we let $\F^I\in\Coeff(\A)$ be as in \S\ref{subsection_2_2} and consider $\F^I$ as a coefficient system on $\Cscr^0$ via restriction.
\begin{prop}\label{chain_complex_D}
Let $\F\in\Coeff_G^0(\BT)$.
\begin{enumerate}[wide]
 \item[(i)]Restricting $I$-invariant oriented chains from $\BT^+$ to $\Cscr^0$ induces an isomorphism of complexes $(\C^\ori_c(\BT^+_{(\bullet)},\F)^I,\partial_\bullet)\stackrel{\cong}{\longrightarrow}(\C^\ori_c(\Cscr^0_{(\bullet)},\F^I),\partial_\bullet)$ of $R$-modules.
 \item[(ii)]Assume that the restriction maps $t^F_{F'}:\F_F^I\to\F_{F'}^I$ of the coefficient system $\F^I\in\Coeff(\Cscr)$ are bijective for all faces $F'$ and $F$ of $\Cscr^0$ with $F'\subseteq\overline{F}$ and $C(F')=C(F)$. Then the complexes $(\C^\ori_c(\BT^+_{(\bullet)},\F)^I,\partial_\bullet)$ and $(\C^\ori_c(\Cscr^0_{(\bullet)},\F^I),\partial_\bullet)$ are acyclic and the natural map
 \[
  \iota_{x_0}:\F_{x_0}^I\hookrightarrow(\bigoplus_{y\in \BT^+_0}\F_y)^I=\C^\ori_c(\BT^+_{(0)},\F)^I\to\H_0(\C_c^\ori(\BT^+_{(\bullet)},\F)^I)
 \]
is an $R$-linear bijection.
\end{enumerate}
\end{prop}
\begin{proof}
The proof of (i) is identical to that of Proposition \ref{restriction_chain_complexes}. As for (ii), the proof of Proposition \ref{acyclic} carries over once we can prove analogs of \cite{OS1}, Lemma 4.15 and Proposition 4.16. For any non-negative integer $n$ let $\Cscr^0(n)$ denote the set of faces $F$ of $\Cscr^0$ such that $C(F)$ and $C$ have gallery distance less than or equal to $n$. If $\mathrm{Ch}(\Cscr^0(n))$ denotes the set of chambers in $\Cscr^0$ whose gallery distance to $C$ is less than or equal to $n$ then we have the disjoint decomposition
\[
 \Cscr^0(n)=\Cscr^0(n-1)\;\dot{\cup}\dot{\bigcup_{D\in\mathrm{Ch}(\Cscr^0(n))}}\overline{D}\setminus \Cscr^0(n-1)
\]
for all $n>0$. Moreover, if $n>0$ and if $D\in\mathrm{Ch}(\Cscr^0(n))$ then the subcomplexes $\Cscr^0(n-1)$ and $\overline{D}\cup \Cscr^0(n-1)$ of $\Cscr^0$ are contractible. This follows from the proof of \cite{OS1}, Proposition 4.16, because $\Cscr^0$ is convex and the intersection of two star-like subsets of a Euclidean space is again star-like and therefore contractible.
\end{proof}
Let $\Cscr$ be an arbitrary closed vector chamber contained in $\Cscr^0$. Given $\F\in\Coeff_G(\BT)$ and $g\in G^+$ we denote by $\varphi_g$ the endomorphism of the complex $\C_c^\ori(\BT_{(\bullet)}^+(\Cscr),\F)$ given by
\begin{small}
\[
 \varphi_g:\C_c^\ori(\BT_{(\bullet)}^+(\Cscr),\F)\hookrightarrow \C_c^\ori(\BT_{(\bullet)},\F)\stackrel{g}{\longrightarrow}\C_c^\ori(\BT_{(\bullet)},\F)\stackrel{\res}{\longrightarrow} \C_c^\ori(\BT_{(\bullet)}^+(\Cscr),\F).
\]
\end{small}Here the leftmost map is the extension by zero of oriented chains on $\BT^+(\Cscr)$ to oriented chains on $\BT$. Note that this makes $\C_c^\ori(\BT_{(\bullet)}^+(\Cscr),\F)$ a $G^+$-stable subcomplex of $\C_c^\ori(\BT_{(\bullet)},\F)$. Indeed, if $f\in\C_c^\ori(\BT_{(i)}^+(\Cscr),\F)\subseteq\C_c^\ori(\BT_{(i)},\F)$ and $(F,c)\in\BT_{(i)}^+(\Cscr)$ then $g(f)(F,c)=c_{g,g^{-1}F}(f(g^{-1}F,g^{-1}c))$. This is zero if $F\not\subseteq \BT^+(\Cscr)$ because $\BT^+(\Cscr)$ is $G^+$-stable and $f$ is supported on $\BT_{(i)}^+(\Cscr)$. Thus, also $g(f)$ is supported on $\BT_{(i)}^+(\Cscr)$. More precisely, the support of $\varphi_g(f)$ is contained in $g\BT^+(\Cscr)$ as is clear from the explicit formula
\[
 \varphi_g(f)(F,c)=
 \left\{
 \begin{array}{ll}
  c_{g,g^{-1}F}(f(g^{-1}F,g^{-1}c)),&\mbox{ if }g^{-1}F\subseteq\BT^+(\Cscr)\\
  0,&\mbox{ otherwise.}
 \end{array}
 \right.
\]
Note that if $g,h\in G^+$ and if $F$ is a face of $\BT^+(\Cscr)$ with $g^{-1}h^{-1}F\subseteq\BT^+(\Cscr)$ then also $h^{-1}F\subseteq g\BT^+(\Cscr)\subseteq \BT^+(\Cscr)$ since $\BT^+(\Cscr)$ is $G^+$-stable. It then follows directly from the definitions that
\[
 \varphi_1=\id\quad\mbox{and}\quad\varphi_g\circ\varphi_h=\varphi_{gh}\quad\mbox{for all}\quad g,h\in G^+.
\]
Altogether, $\C_c^\ori(\BT_{(\bullet)}^+(\Cscr),\F)$ is a complex of smooth $R$-linear $G^+$-repre\-sen\-ta\-tions via $g\cdot f=\varphi_g(f)$ and $\mathrm{H}_i(\BT^+(\Cscr),\F)$ is an object of $\Rep^\infty_R(G^+)$ for all $i\geq 0$. By abuse of notation we continue to write $\varphi_g$ for the endomorphism of $\mathrm{H}_i(\BT^+(\Cscr),\F)$ induced by $\varphi_g$.
\begin{rem}\label{scalar_restriction}
The inclusion $\C_c^\ori(\BT_{(\bullet)}^+(\Cscr),\F)\subseteq \C_c^\ori(\BT_{(\bullet)},\F)$ of complexes of $G^+$-representations induces an inclusion $\C_c^\ori(\BT_{(\bullet)}^+(\Cscr),\F)^I\subseteq \C_c^\ori(\BT_{(\bullet)},\F)^I$ of complexes of $H^+$-modules. This in turn gives rise to $H^+$-linear maps on the homology groups. Assume that $R$ is a quasi-Frobenius ring. If $M\in \Mod_H$ then Proposition \ref{acyclic}, Theorem \ref{quasi_inverse_coefficient_system} and Proposition \ref{chain_complex_D} imply that for the closed vector chamber $\Cscr=\Cscr^0$ the map $\mathrm{H}_0(\C_c^\ori(\BT^+_{(\bullet)},\F(M))^I)\to \mathrm{H}_0(\C_c^\ori(\BT_{(\bullet)},\F(M))^I)=M(\F(M))\cong M$ is an isomorphism of $H^+$-modules. Thus, the restriction of $\F(M)$ to $\BT^+$ determines the scalar restriction of $M$ to $H^+$.
\end{rem}
If $g\in G^+$ we denote by $\psi_g$ the endomorphism of $\C_c^\ori(\BT_{(\bullet)}^+(\Cscr),\F)$ given by
\begin{small}
\[
 \psi_g:\C_c^\ori(\BT_{(\bullet)}^+(\Cscr),\F)\hookrightarrow \C_c^\ori(\BT_{(\bullet)},\F)\stackrel{g^{-1}}{\longrightarrow}\C_c^\ori(\BT_{(\bullet)},\F)\stackrel{\res}{\longrightarrow} \C_c^\ori(\BT_{(\bullet)}^+(\Cscr),\F).
\]
\end{small}Explicitly, if $0\leq i\leq d$, $f\in\C_c^\ori(\BT_{(i)}^+(\Cscr),\F)$ and $(F,c)\in\BT_{(i)}^+(\Cscr)$ then $\psi_g(f)(F,c)=  c_{g^{-1},gF}(f(gF,gc))$ because $g\BT^+(\Cscr)\subseteq\BT^+(\Cscr)$. This formula shows that
\[
 \psi_1=\id\quad\mbox{and}\quad\psi_g\circ\psi_h=\psi_{hg}\quad\mbox{for all}\quad g,h\in G^+.
\]
Consequently, $\C_c^\ori(\BT_{(\bullet)}^+(\Cscr),\F)$ is a complex of smooth $R$-linear $(G^+)^{-1}$-repre\-sen\-tations via $g^{-1}\cdot f=\psi_g(f)$ and $\mathrm{H}_i(\BT^+(\Cscr),\F)$ is an object of $\Rep_R((G^+)^{-1})$ for all $i\geq 0$. Again, we also denote by $\psi_g$ the induced $R$-linear endomorphism of $\mathrm{H}_i(\BT^+(\Cscr),\F)$ for all $i\geq 0$. As a formal consequence of the definitions we have
\[
 \psi_g\circ\varphi_g=\id\quad\mbox{for all}\quad g\in G^+.
\]
Note that if $g\in I\subseteq G^+\cap(G^+)^{-1}$ then we also have $\varphi_g\circ\psi_g=\id$, and we will write $\varphi_g=g=\psi_{g^{-1}}$.\\

Assume that the ring $R$ is quasi-Frobenius and hence artinian (cf.\ \cite{Lam}, Theorem 15.1). We note that the more general definitions and constructions of \cite{SV} can also be carried out over $R$. Thus, we let
\[
R\llbracket\Pbar^+\rrbracket\cong R\llbracket\Pbar\cap I\rrbracket\otimes_{R[\Pbar\cap I]}R[\Pbar^+],
\]
in analogy to \cite{SV}, \S1, where this ring is denoted by $\Lambda(\Pbar_+)$. If $g\in\Pbar^+$ then we denote by $\delta_g$ the image of $g$ under the maps $\Pbar^+\to R[\Pbar^+]\to R\llbracket\Pbar^+\rrbracket$.\\

Given an object $V\in\Rep^\infty_R((\Pbar^+)^{-1})$ its $R$-linear dual $V^*=\Hom_R(V,R)$ is a pseudocompact $R$-module in the sense of \cite{Bru}, \S1, and carries the structure of a left $R\llbracket\Pbar^+\rrbracket$-module characterized by
\[
 (\delta_g\cdot\ell)(v)=\ell(g^{-1}v)\quad\mbox{for all}\quad g\in\Pbar^+,\ell\in V^*\mbox{ and }v\in V.
\]
Since the ring $R$ is selfinjective the functor $V\mapsto V^*$ is exact. Using that any finitely generated $R$-module is reflexive (cf.\ \cite{Lam}, Theorem 15.11) one can show that it is in fact an equivalence of abelian categories between $\Rep^\infty_R((\Pbar^+)^{-1})$ and the category of pseudocompact $R$-modules $M$ endowed with a continuous $R$-linear action $\Pbar^+\times M\to M$ of $\Pbar^+$. The latter extends to an $R\llbracket\Pbar^+\rrbracket$-module structure in a canonical way (cf.\ \cite{Koh}, Theorem 1.5, for a related result). Likewise, if $V\in\Rep_R^\infty(\Pbar^+)$ then the $R$-module $V^*$ is an $R\llbracket(\Pbar^+)^{-1}\rrbracket$-module in a natural way.\\

If $\Cscr$ runs through the closed vector chambers contained in $\Cscr^0$ then the subcomplexes $\BT^+(\Cscr)$ of $\BT$ form a directed set with respect to reverse inclusion. A cofinal subset is given by the complexes $\BT^+(t\Cscr^0)$ with $t\in T^+$. Consequently, for any $0\leq i\leq d$ the family $(\C^\ori_c(\BT_{(i)}^+(\Cscr),\F)^*)_{\Cscr\subseteq\Cscr^0}$ is an inductive system of $R\llbracket(\Pbar^+)^{-1}\rrbracket$- and $R\llbracket\Pbar^+\rrbracket$-modules whose transition maps are dual to the inclusion maps $\C^\ori_c(\BT_{(i)}^+(\Cscr'),\F)\subseteq \C^\ori_c(\BT_{(i)}^+(\Cscr),\F)$ whenever $\Cscr'\subseteq\Cscr\subseteq\Cscr^0$. Thus,
\begin{equation}\label{etale_complex}
 \varinjlim_{\Cscr\subseteq\Cscr^0}\C^\ori_c(\BT_{(\bullet)}^+(\Cscr),\F)^*,
\end{equation}
is a complex of $R\llbracket(\Pbar^+)^{-1}\rrbracket$- and $R\llbracket\Pbar^+\rrbracket$-modules. Recall that the $R\llbracket\Pbar^+\rrbracket$-module structure is induced by the operators $\psi_p$ with $p\in\Pbar^+$.
\begin{prop}\label{etale}
If $R$ is a quasi-Frobenius ring and if $\F\in\Coeff_G(\BT)$ then the complex (\ref{etale_complex}) and its cohomology groups consist of \'etale $R\llbracket\Pbar^+\rrbracket$-modules in the sense of \cite{SV}, Definition 1.2.
\end{prop}
\begin{proof}
As in \cite{SV}, Proposition 1.3, the category of \'etale $R\llbracket\Pbar^+\rrbracket$-modules is abelian. Therefore, it suffices to show that each member of the complex (\ref{etale_complex}) is \'etale. Fixing $0\leq i\leq d$ and $t\in T^+$, it suffices to see that the endomorphism
\[
\sum_{u\in\Ubar_1/t\Ubar_1t^{-1}}\psi_u^*\circ\psi_t^*\circ\varphi_t^*\circ \psi_{u^{-1}}^*
\]
of $\varinjlim_{\Cscr}\C^\ori_c(\BT_{(i)}^+(\Cscr),\F)^*$ is the identity (cf.\ \cite{SVZ}, Remark 3.3.1). Let $\delta\in \C^\ori_c(\BT_{(i)}^+(\Cscr),\F)^*$ and choose $s\in T^+$ such that $\nu(s)$ is strictly dominant. By definition of the transition maps in the inductive limit it suffices to see that the linear form $\sum_{u\in\Ubar_1/t\Ubar_1t^{-1}}\psi_u^*\circ\psi_t^*\circ\varphi_t^*\circ \psi_{u^{-1}}^*(\delta)$ coincides with $\delta$ upon restriction to $\C^\ori_c(\BT_{(i)}^+(ts\Cscr),\F)$. To prove this we need to see that the endomorphism $\sum_{u\in\Ubar_1/t\Ubar_1t^{-1}}u\circ\varphi_t\circ\psi_t\circ u^{-1}$ of $\C^\ori_c(\BT_{(i)}^+(\Cscr),\F)$ restricts to the identity on the $R$-submodule $\C^\ori_c(\BT_{(i)}^+(ts\Cscr),\F)$.\\

Let $f\in\C^\ori_c(\BT_{(i)}^+(ts\Cscr),\F)$ and $(F,c)\in\BT_{(i)}^+(ts\Cscr)$. Since $\BT^+(ts\Cscr)=\Ubar_1ts\Cscr$ we can write $F=vtsF'$ with $v\in\Ubar_1$ and $F'\subseteq\Cscr$. Assume that $u\in\Ubar_1$ such that $t^{-1}u^{-1}vtsF'\subseteq\BT^+(\Cscr)$, i.e.\ $t^{-1}u^{-1}vtsF'=u'F''$ for some $u'\in\Ubar_1$ and $F''\subseteq\Cscr$. Let $g=(u')^{-1}t^{-1}u^{-1}vts$ so that $F'=g^{-1}F''\subseteq\A\cap g^{-1}\A$. By \cite{BT}, Proposition 7.4.8, there exists an element $n\in N_G(T)$ with $n^{-1}g\in P_{F'}$. Note that $n^{-1}g=n^{-1}s\tilde{u}$ where $n^{-1}s\in N_G(T)$ and $\tilde{u}=s^{-1}(u')^{-1}t^{-1}u^{-1}vts\in\Ubar$. Choosing a vertex $x\in\overline{F'}$ we have $n^{-1}s\tilde{u}\in P_{F'}\subseteq P_x$ and $P_xn^{-1}s\Ubar=P_xn^{-1}s\tilde{u}\Ubar=P_x\Ubar$. The Iwasawa decomposition in \cite{BT}, Corollaire 7.3.2 (i), implies $n^{-1}s\in N_G(T)\cap P_x$ and $\tilde{u}=s^{-1}n\cdot n^{-1}g\in\Ubar\cap P_x$. Recall from \cite{BT2}, \S5.2.4, that $P_x=T_0U_x$ for a normal subgroup $U_x\subseteq P_x$ such that the multiplication map $(\Ubar\cap U_x)\times (N_G(T)\cap U_x)\times (U\cap U_x)\to U_x$ is bijective and such that $\Ubar\cap U_x$ is generated by the subgroups $U_{\alpha,-\alpha(x)}$ with $\alpha\in\Phi^-$. Write $\tilde{u}=a\overline{y}by$ with $a\in T_0$, $\overline{y}\in\Ubar\cap U_x$, $b\in N_G(T)\cap U_x$ and $y\in U\cap U_x$. Then $\Ubar U=\Ubar\tilde{u}U=\Ubar abU$ and the Bruhat decomposition $G=\coprod_{w\in W}\Ubar wT_0U$ implies $b\in T_0$. Writing $\tilde{u}=a\overline{y}a^{-1}\cdot ab\cdot y$ the injectivity of the multiplication map $\Ubar\times T\times U\to G$ then implies $ab=y=1$ and $\tilde{u}=a\overline{y}a^{-1}\in\Ubar\cap U_x$. This is contained in the subgroup $\Ubar_0$ of $\Ubar$ generated by the subgroups $\Ubar_{\alpha,0}$ with $\alpha\in\Phi^-$ because $x\in\Cscr^0$. Since $\nu(s)$ is strictly dominant we obtain $s\Ubar_0s^{-1}\subseteq \Ubar_1$ and $u^{-1}v\in t\Ubar_1 t^{-1}$.\\

Altogether, there is a unique class $u\in\Ubar_1/t\Ubar_1t^{-1}$ for which $t^{-1}u^{-1}F\subseteq\BT^+(\Cscr)$, namely $v\cdot t\Ubar_1t^{-1}$. By definition of $\varphi_t$ we obtain
\begin{small}
\[
 \sum_{u\in\Ubar_1/t\Ubar_1t^{-1}}(u\circ\varphi_t\circ\psi_t\circ u^{-1})(f)(F,c) =(v\circ\varphi_t\circ\psi_t\circ v^{-1})(f)(F,c)=f(F,c)
\]
\end{small}proving the claim.
\end{proof}
If $\F\in\Coeff_G(\BT)$ and if $0\leq i\leq d$ then $\C_c^\ori(\BT_{(i)},\F)$ is a smooth $R$-linear $G$-representation and hence an object of $\Rep_R^\infty(\Pbar)$ via restriction. Assume that $R$ is a quasi-Frobenius ring. As in \cite{SV}, \S2, there is a functor $D$ from $\Rep_R^\infty(\Pbar)$ to the category of $R\llbracket(\Pbar^+)^{-1}\rrbracket$-modules given by
\[
 D(V)=\varinjlim_MM^*,
\]
where $M$ runs through the filtered family of $\Pbar^+$-subrepresentations of $V$ satisfying $R[\Pbar]\cdot M=V$.
\begin{prop}\label{SV_functor}
Assume that $R$ is a quasi-Frobenius ring. If $\F\in\Coeff_G(\BT)$ and if $0\leq i\leq d$ then there is a canonical $R\llbracket(\Pbar^+)^{-1}\rrbracket$-linear surjection
\begin{equation}\label{SV_comparison}
 \varinjlim_{\Cscr\subseteq\Cscr^0}\C^\ori_c(\BT_{(i)}^+(\Cscr),\F)^*\longrightarrow D(\C_c^\ori(\BT_{(i)},\F)).
\end{equation}
If $\F_F$ is a finitely generated $R$-module for any face $F$ of $\BT$ then the map (\ref{SV_comparison}) is an isomorphism.
\end{prop}
\begin{proof}
We will first show that $\C_c^\ori(\BT_{(i)}^+(\Cscr),\F)$ generates $\C_c^\ori(\BT_{(i)},\F)$ as a $\Pbar$-representation for any $\Cscr\subseteq\Cscr^0$. To see this let $f\in \C_c^\ori(\BT_{(i)},\F)$ and note that $f$ is a finite sum of oriented chains $f_j$ such that $f_j$ is supported on $(F_j,\pm c_j)$ for some oriented face $(F_j,c_j)$. By Lemma \ref{face_representatives} and the Iwasawa decomposition $G=\Ubar WI$ there is an element $u_j\in\Ubar$ with $u_jF_j\subseteq\A$. Since $T\cdot\Cscr=\A$ there is an element $g_j\in \Pbar$ with $g_j\cdot F_j\subseteq\Cscr$. This implies $g_j^{-1} f_j\in \C_c^\ori(\BT_{(i)}^+(\Cscr),\F)$ and $f=\sum_jg_j\cdot g_j^{-1}f_j\in R[\Pbar]\cdot\C_c^\ori(\BT_{(i)}^+(\Cscr),\F)$.\\

We thus obtain the required $R\llbracket(\Pbar^+)^{-1}\rrbracket$-linear map (\ref{SV_comparison}) which is surjective because $R$ is selfinjective. For the final statement assume that $\F_F$ is a finitely generated $R$-module for any face $F$ of $\BT$. We need to see that any $\Pbar^+$-subrepresentation $M$ of $V=\C_c^\ori(\BT_{(i)},\F)$ with $R[\Pbar]\cdot M=V$ contains $\C_c^\ori(\BT_{(i)}^+(\Cscr),\F)$ for some closed vector chamber $\Cscr\subseteq\Cscr^0$.\\

Let $\Delta\subseteq\Phi^+$ be the set of positive simple roots. For $\alpha\in\Delta$ let $\lambda_\alpha$ be the corresponding fundamental dominant coweight characterized by $\langle \beta,\lambda_\alpha\rangle=\delta_{\alpha\beta}$ for all $\beta\in\Delta$. If $n$ denotes the index of $X_*(\TT)$ in the coweight lattice of $\Phi$ we choose elements $t_\alpha\in T^+$ with $\nu(t_\alpha)=n\lambda_\alpha$ for all $\alpha\in\Delta$. The set $J$ of closed chambers $\overline{D}$ contained in $\Cscr^0$ with $\min\{\langle\alpha,z\rangle\;|\;z\in\overline{D}\}<n$ for all $\alpha\in\Delta$ is finite. Indeed, let $c=\max\{|\langle\beta,z\rangle|\;|\;z\in\Cbar,\beta\in\Phi\}$ and write $D=wC+\lambda$ with $w\in W_0$ and $\lambda\in X_*(\TT/\mathbb{C})$. Let $\alpha\in\Delta$ and choose $z\in\overline{D}$ such that $\langle\alpha,z\rangle$ attains the above minimum. If $z'\in\overline{D}$ then
\begin{eqnarray*}
0\leq\langle\alpha,z'\rangle&<&|\langle\alpha,z'-z\rangle|+n\\
&=&|\langle w^{-1}\alpha,w^{-1}(z'-\lambda)-w^{-1}(z-\lambda)\rangle|+n\leq n+2c.
\end{eqnarray*}
Thus, $\overline{D}$ is contained in a compact subset of $\A$, proving the above finiteness claim. Any other closed chamber in $\Cscr^0$ is of the form $\prod_{\alpha\in\Delta}t_\alpha^{n_\alpha}\overline{D}$ with suitable non-negative integers $n_\alpha$ and $\overline{D}\in J$, i.e.\ $\bigcup_{\overline{D}\in J}T^+\overline{D}=\Cscr^0$ and $\bigcup_{\overline{D}\in J}T  \overline{D}=T\Cscr=\A=W\Cbar$. The Iwasawa decomposition $G=\Pbar WI$ therefore implies $\BT=G\Cbar=\bigcup_{\overline{D}\in J}\Pbar\cdot\overline{D}$. If we let $J_i=\{F\in \BT_i\;|\;F\subseteq\overline{D}\mbox{ for some }\overline{D}\in J\}$
then we obtain $\Pbar^+J_i=\Ubar_1T^+J_i =\Ubar_1\Cscr^0_i=\BT^+_i(\Cscr^0)$ and $\Pbar J_i=\BT_i$. As a consequence,
\[
 V=\C_c^\ori(\BT_{(i)},\F)\cong\sum_{F\in J_i}\ind_{\Pbar\cap P_F^\dagger}^\Pbar(\varepsilon_F\otimes_R\F_F)
\]
as a $\Pbar$-representation where $\varepsilon_F$ is as in the proof of Proposition \ref{I_contains_C}. Assume that any $\F_F$ is finitely generated over $R$ and put ${V_F=\ind_{\Pbar\cap P_F^\dagger}^\Pbar(\varepsilon_F\otimes_R\F_F)}$. By \cite{SV}, Lemma 2.3, the $\Pbar$-representation $V_F$ is generated by its $\Pbar^+$-sub\-re\-pre\-sen\-ta\-tion $M_F=M\cap V_F$. Since the underlying $R$-module of $\varepsilon_F\otimes_R\F_F$ is finitely generated we can argue as in \cite{SV}, Lemma 3.1, and find an element $t_F\in T^+$ such that $\Pbar^+t_F\cdot(\varepsilon_F\otimes_R\F_F)\subseteq M$. Setting $t=\prod_{F\in J_i}t_F\in T^+$ we see that $M$ contains the subspace of all oriented chains supported on $\{pt(F,\pm 1)\;|\;p\in\Pbar^+,F\in J_i\}=\BT^+_{(i)}(t\Cscr^0)$ because $\Pbar^+tJ_i=\Ubar_1tT^+J_i=\Ubar_1t\Cscr^0_i=\BT_i^+(t\Cscr^0)$. Therefore, $\C_c^\ori(\BT^+_{(i)}(t\Cscr^0),\F)\subseteq M$, as claimed.
\end{proof}
Given an $H$-module $M$ we now study the exactness properties of the complex (\ref{etale_complex}) of \'etale $R\llbracket\Pbar^+\rrbracket$-modules if $\F=\F(M)$. Our results in this direction are limited by the corresponding exactness results for $\C_c^\ori(\BT_{(\bullet)},\F(M))$ in Proposition \ref{semisimple_rank_one}.
\begin{prop}\label{non_vanishing}
Assume that $R$ is a quasi-Frobenius ring in which $p$ is nilpotent and that the semisimple rank of $\GG$ is equal to one. If $M\in\Mod_H$ then the complex $\varinjlim_{\Cscr}\C^\ori_c(\BT_{(\bullet)}^+(\Cscr),\F(M))^*$ is acyclic. If $M$ is non-zero then its $0$-th cohomology group does not vanish.
\end{prop}
\begin{proof}
The acyclicity simply means that the map
\[
\varinjlim_{\Cscr\subseteq\Cscr^0}\C^\ori_c(\BT_{(0)}^+(\Cscr),\F(M))^*\to \varinjlim_{\Cscr\subseteq\Cscr^0}\C^\ori_c(\BT_{(1)}^+(\Cscr),\F(M))^*
\]
is surjective. By the exactness of $\varinjlim$ and $(\cdot)^*$ this amounts to showing that the map $\C^\ori_c(\BT_{(1)}^+(\Cscr),\F(M))\to\C^\ori_c(\BT_{(0)}^+(\Cscr),\F(M))$ is injective for any vector chamber $\Cscr$. However, this is simply the restriction of the map $\C^\ori_c(\BT_{(1)},\F(M))\to\C^\ori_c(\BT_{(0)},\F(M))$ which is injective by our hypotheses and Proposition \ref{semisimple_rank_one}.\\

Let $t\in T^+$ be an element for which $\nu(t)\in X_*(\TT/\mathbb{C})$ is strictly dominant. If $\Cscr$ is a vector chamber contained in $\Cscr^0$ then there is a non-negative integer $n$ with $\BT^+(t^n\Cscr^0)\subseteq\BT^+(\Cscr)$. Since $\C_c^\ori(\BT_{(\bullet)}^+(t^n\Cscr^0),\F(M))$ is the image of the endomorphism $\varphi_t^n$ of $\C_c^\ori(\BT_{(\bullet)}^+,\F(M))$ the $0$-th cohomology group of the complex $\varinjlim_{\Cscr}\C^\ori_c(\BT_{(\bullet)}^+(\Cscr),\F(M))^*$ is isomorphic to
\[
 \varinjlim_{n\geq 0}\H_0(\BT^+,\F(M))^*,
\]
where the transition maps are given by $\varphi_t^*$. If $M$ is non-zero then so is the $R$-module $\H_0(\BT^+,\F(M))$ because $M\cong\F(M)^I_{x_0}\hookrightarrow\H_0(\BT^+,\F(M))^I$ by Theorem \ref{quasi_inverse_coefficient_system} and Proposition \ref{chain_complex_D}. From this point on, the non-vanishing statement is an exercise in linear algebra. Namely, let $V$ be any non-zero $R$-module and let $\varphi$ be an injective $R$-linear endomorphism of $V$. We claim that the $R$-module $\varinjlim_{n\geq 0}V^*$ is non-zero if the transition maps are given by $\varphi^*$. In order to see this we claim there is an $R$-linear map $\delta:V\to R$ such that $\delta\circ\varphi^n\not=0$ for all $n\geq 0$. The image of $\delta$ under the canonical map $V^*\to \varinjlim_{n\geq 0}V^*$ will then be non-zero. In order to construct $\delta$ we first consider the case that the submodule $W=\cap_{n\geq 0}\mathrm{im}(\varphi^n)$ of $V$ is non-zero. Since $R$ is selfinjective any finitely generated submodule $W_0$ of $W$ is reflexive (cf.\ \cite{Lam}, Theorem 15.11). Choosing $W_0\not=0$ there is a non-zero element $\delta_0\in W_0^*$. Since $R$ is selfinjective $\delta_0$ can be extended to a linear form $\delta\in V^*$ with the required properties.\\

Now assume $\cap_{n\geq 0}\mathrm{im}(\varphi^n)=0$. Since $V\not=0$ and since $\varphi$ is injective we have $\mathrm{im}(\varphi^{n+1})\subsetneqq\mathrm{im}(\varphi^n)$ for all $n\geq 0$. For any $n\geq 0$ choose $v_n\in\mathrm{im}(\varphi^n)$ with $v_n\not\in\mathrm{im}(\varphi^{n+1})$ and set $V_n=\sum_{m=0}^nRv_m$. We will inductively construct linear forms $\delta_n:V_n\to R$ with $\delta_{n+1}|_{V_n}=\delta_n$ and $\delta_n(v_n)\not=0$ for all $n\geq 0$. Since $V_0\not=0$ we have $V_0^*\not=0$ as above and choose $0\not=\delta_0\in V_0^*$ arbitrary. Assume that $\delta_n$ has been constructed. Since $R$ is selfinjective we can extend $\delta_n$ to a linear form $\delta_n':V_{n+1}\to R$. If $\delta_n'(v_{n+1})\not=0$ we set $\delta_{n+1}=\delta_n'$. Otherwise, we choose a non-zero linear form $\delta_n'':V_{n+1}/V_n\to R$, using that $v_{n+1}\not\in V_n$. We view $\delta_n''$ as an $R$-linear form $V_{n+1}\to R$ vanishing on $V_n$ and set $\delta_{n+1}=\delta_n'+\delta_n''$. We thus obtain a linear form $\delta_\infty:\sum_{n\geq 0}Rv_n\to R$ with $\delta_\infty(v_n)\not=0$ for all $n\geq 0$. The selfinjectivity of $R$ allows us to extend $\delta_\infty$ to an $R$-linear map $\delta:V\to R$ with the required properties.
\end{proof}
For any $\F\in\Coeff_G(\BT)$ and any $i\geq 0$ the smooth $R$-linear $G$-representation $\H_i(\BT,\F)$ can be viewed as an object of $\Rep_R^\infty(\Pbar)$ via restriction. In the situation considered in \cite{SV}, Schneider and Vign\'eras associate with this object a family of \'etale $R\llbracket\Pbar^+\rrbracket$-modules $(D^j\H_i(\BT,\F))_{j\geq 0}$. These are related to the complex (\ref{etale_complex}) of \'etale $R\llbracket \Pbar^+\rrbracket$-modules as follows. As before, we let $Z=\mathbb{C}(K)$ where $\mathbb{C}$ is the connected component of the center of $\GG$.
\begin{prop}\label{etale_spectral_sequence}
Assume that $K=\Qp$ and that $R=o/\pi^no$ for some positive integer $n$ where $o$ is the valuation ring of a finite field extension of $\Qp$. Let $\F\in\Coeff_G(\BT)$.
\begin{enumerate}[wide]
\item[(i)]If $0\leq i\leq d$ and if $j>0$ then $D^j(\C_c^\ori(\BT_{(i)},\F))=0$. If $Z$ acts on $\F_F$ through a character for any face $F$ of $\BT$ then we have $D^0(\C_c^\ori(\BT_{(i)},\F))=D(\C_c^\ori(\BT_{(i)},\F))$.
\item[(ii)]Assume that the $R$-module $\F_F$ is finitely generated and that $Z$ acts on $\F_F$ through a character for any face $F$ of $\BT$. Then there is an $E_2$-spectral sequence of \'etale $R\llbracket \Pbar^+\rrbracket$-modules
\[
D^j\H_i(\BT,\F)\Longrightarrow\H^{j+i}(\varinjlim_{\Cscr\subseteq\Cscr^0}\C^\ori_c(\BT_{(i)}^+(\Cscr),\F)^*).
\]
If the semisimple rank of $\GG$ is equal to one and if $\F=\F(M)$ for some $M\in\Mod_H$ whose underlying $R$-module is finitely generated and on which $R[ZI/I]\subseteq H$ acts through a character then this spectral sequence degenerates. In this case the \'etale $R\llbracket \Pbar^+\rrbracket$-module $D^j\H_0(\BT,\F(M))$ is the $j$-th cohomology group of the complex (\ref{etale_complex}). In particular, we have $D^j\H_0(\BT,\F(M))=0$ for all $j\geq 1$.
\end{enumerate}
\end{prop}
\begin{proof}
In order to see that the $\Pbar$-representations $\C_c^\ori(\BT_{(i)},\F)$ are acyclic for the $\delta$-functor $(D^j)_{j\geq 0}$ we generalize the proof of \cite{SV}, Lemma 11.8 (i). Let $\tilde{G}$ denote the direct product of $Z$ and the group of $K$-rational points of the universal cover of the derived group of $\GG$ (cf.\ \cite{BoT}, Proposition 2.24). We let $f:\tilde{G}\to G$ denote the canonical group homomorphism. By \cite{BoT}, Th\'eor\`eme 2.20, we can choose a parabolic subgroup $\mathbb{Q}$ and maximal $K$-split torus $\mathbb{S}$ such that on $K$-rational points $f$ restricts to morphisms $Q=\mathbb{Q}(K)\to\Pbar$ and $S=\mathbb{S}(K)\to T$. By the proof of \cite{BoT}, Th\'eor\`eme 2.20, we may identify the unipotent radicals of $\mathbb{Q}$ and $\overline{\mathbb{P}}$. Letting $S^+=\{s\in S\;|\;s\Ubar_1s^{-1}\subseteq \Ubar_1\}$ and $Q^+=\Ubar_1S^+$ we obtain $f(S^+)=f(S)\cap T^+$ and $f(Q^+)=f(Q)\cap \Pbar^+$.\\

Given $V\in\Rep_R^\infty(\Pbar)$ we view $V$ as a representation of $Q$ via inflation along $f$. We claim that the $R\llbracket(Q^+)^{-1}\rrbracket$-module $D(V)$ (computed from the $Q$-representation $V$ using the monoid $Q^+$) is the scalar restriction of the $R\llbracket(\Pbar^+)^{-1}\rrbracket$-module $D(V)$ (computed from the $\Pbar$-representation $V$ using the monoid $P^+$) along the ring homomorphism $R\llbracket(Q^+)^{-1}\rrbracket\to R\llbracket(\Pbar^+)^{-1}\rrbracket$ induced by $f$.
To see this we first show that $T=f(S)T^+$ and thus $\Pbar=f(Q)T^+$. Reducing modulo $T_0$ the first of equality is equivalent to $X_*(\TT)=X_*^+(\TT)+X_*(\mathbb{S})$. Note that $f$ identifies $X_*(\mathbb{S})$ with a finite index subgroup of $X_*(\TT)$. Let $n$ denote the index of $X_*(\mathbb{S})$ in the coweight lattice $\Lambda$ of $\Phi$. We extend a $\mathbb{Z}$-basis of $X_*(\mathbb{C})$ by the fundamental dominant coweights to a $\mathbb{Z}$-basis $(\lambda_i)_{i\in I}$ of $\Lambda$. If $\lambda=\sum_i r_i\lambda_i\in X_*(\TT)$ choose an integer $m$ with $nm+r_i\geq 0$ for all $i\in I$. Then $\lambda'=\sum_inm\lambda_i\in X_*(\mathbb{S})$, $\lambda+\lambda'\in X_*^+(\TT)$ and $\lambda=\lambda+\lambda'-\lambda'\in X_*^+(\TT)+X_*(\mathbb{S})$ as claimed.\\

Now let $M$ be a $\Pbar^+$-subrepresentation of $V$ which generates $V$ over $P$. Then $M$ is $Q^+$-stable with $R[Q]\cdot M=R[f(Q)]\cdot M=R[\Pbar]\cdot M=V$ because $M$ is stable under $T^+\subseteq \Pbar^+$ and $f(Q)T^+=P$. Conversely, let $M$ be a $Q^+$-subrepresentation of $V$ which generates $V$ over $Q$. The monoid $X_*^+(\TT)$ is finitely generated by Gordon's lemma. Moreover, the index $(T:f(S))$ is finite because $f$ is an isogeny. Therefore, also the index $(T_0:f(S)\cap T_0)$ is finite and there are finitely many elements $t_1,\ldots,t_m\in T^+$ with $T^+=\cup_{j=1}^mt_jf(S^+)$. Let $M'$ be any finitely generated $R$-submodule of $M$. Since $\sum_jt_jM'\subseteq V=R[Q]M$ we can argue as in \cite{SV}, Lemma 3.1, and find an element $s(M')\in f(S^+)$ with $\sum_jt_js(M')M'\subseteq M$. This implies $R[\Pbar^+]s(M')M'\subseteq M$ because $T^+=\cup_jt_jf(S^+)$ and because $M$ is $Q^+$-stable. Now consider the $\Pbar^+$-subrepresentation $N=\sum_{M'}R[\Pbar^+]s(M')M'$ of $V$ in which the sum runs over all finitely generated $R$-submodules $M'$ of $M$. If $v\in V=R[Q]M$ there are finitely many elements $q_i\in Q$ and $m_i\in M$ with $v=\sum_iq_im_i$. Setting $M'=\sum_iRm_i$ we have $s(M')m_i\in N$ for all $i$ and therefore $v=\sum_iq_is(M')^{-1}s(M')m_i\in R[Q]N$. In particular, $R[\Pbar]N=V$. Since $N\subseteq M$ a cofinality argument proves our claim concerning $D(V)$. As a formal consequence, we have analogous assertions for $D^j(V)$ if $j\geq 0$. Namely, the universal resolution of $V\in\Rep_R^\infty(\Pbar)$ in \cite{SV}, \S4, is also acyclic for the functor $D$ computed in $\Rep_R^\infty(Q)$. Since the $\delta$-functor $(D^j)_{j\geq 0}$ is coeffaceable this resolution may be used to compute $D^j(V)$ in $\Rep_R^\infty(Q)$.\\

Note that $\BT$ is also the semisimple Bruhat-Tits building of $\tilde{G}$ and that $\tilde{G}$ acts on $\BT$ through $f$. Likewise, we may use $f$ to view $\F$ as a $\tilde{G}$-equivariant coefficient system on $\BT$. The $\tilde{G}$-representation $\C_c^\ori(\BT_{(i)},\F)$ we obtain is the inflation of the $G$-representation we need to analyze. Altogether, our arguments allow us to assume in (i) that $\GG$ is the direct product of its center and its simply connected derived group.\\

Passing to a suitable subset of $J_i$ as introduced in the proof of Proposition \ref{SV_functor} we see that the $\Pbar$-represen\-tation $\C_c^\ori(\BT_{(i)},\F)$ is a finite direct sum of representations of the form $\ind_{\Pbar\cap P_F^\dagger}^\Pbar(\varepsilon_F\otimes_R\F_F)$. Since the derived group of $\GG$ is simply connected we have $\Omega=Z/(Z\cap T_0)$ and $P_F^\dagger=ZP_F$ by (\ref{Bruhat_parahoric}). To ease notation let us put $U_0=\varepsilon_F\otimes_R\F_F$ and $P_0=\Pbar\cap P_F$ so that $\Pbar\cap P_F^\dagger=ZP_0$. Note that $Z_0=P_0\cap Z$ is the maximal compact subgroup of $Z$ and that $Z/Z_0$ is a free abelian group of finite rank. Choosing generators $\zeta_1,\ldots,\zeta_r\in Z/Z_0$ the elements $\zeta_1-1,\ldots,\zeta_r-1\in o[Z/Z_0]$ form a regular sequence and we consider the associated exact and $o[Z/Z_0]$-linear Koszul complex
\[
0\longrightarrow\bigwedge^\bullet o\otimes_oo[Z/Z_0]\longrightarrow o\longrightarrow 0.
\]
By construction it is the tensor product of $o$-linearly split short exact sequences hence is $o$-linearly split itself. We view it as an exact sequence of smooth $ZP_0$-representations via $ZP_0\to ZP_0/P_0\cong Z/Z_0$. Applying the functor $\ind_{ZP_0}^\Pbar(U_0\otimes_o(\cdot))$ and using the isomorphism $\ind_{P_0}^{ZP_0}(U_0)\cong U_0\otimes_oo[Z/Z_0]$ given by $f\mapsto\sum_{z\in Z/Z_0}zf(z)\otimes z$, we obtain the exact sequence
\[
0\longrightarrow\bigwedge^\bullet o\otimes_o\ind_{P_0}^\Pbar(U_0)\longrightarrow \ind_{ZP_0}^\Pbar(U_0)\longrightarrow 0
\]
in $\Rep_R^\infty(\Pbar)$. Let us denote by $D^j(P_0,\cdot)$ the functors of \cite{SV}, \S4, with respect to the subgroup $P_0$ of $\Pbar$. Note that $P_0=T_0(\Ubar\cap P_0)$ by \cite{BT2}, \S5.2.4. If $r=1$ then the arguments given in the proof of \cite{SV}, Lemma 11.8, show that $D^j(P_0,\ind_{ZP_0}^\Pbar(U_0))=0$ for all $j\geq 1$. The general case is proved inductively by splitting up the above resolution into short exact sequences. It now follows from the base change property in \cite{SV}, Proposition 7.1, that also $D^j(\ind_{ZP_0}^\Pbar(U_0))=0$ for all $j\geq 1$. This uses that the ring homomorphisms $\Lambda(P_+)\to\Lambda(P_+')$ considered in \cite{SV}, \S7, are faithfully flat.\\

In order to prove the second assertion in (i) we need to see that the natural map $D(\ind_{ZP_0}^\Pbar(U_0))\to D^0(\ind_{ZP_0}^\Pbar(U_0))$ is bijective. The same base change techniques together with equation (12) of \cite{SV} allow us to work with $P_0$ instead of with $\Pbar\cap I$. Note that by assumption $Z$ acts via a central character not only on $U_0$ but on the entire $\Pbar$-representation $\ind_{ZP_0}^\Pbar(U_0)$. We let $G'$ denote the group of $K$-rational points of the derived group of $\GG$ and set $\Pbar'=\Pbar\cap G'$, $T'=T\cap G'$ and $P_0'=P_0\cap G'$. Note that $\Ubar$ is also the unipotent radical of $\Pbar'$ so that $P_0'=(T_0\cap G')(\Ubar\cap P_0)$ and the submonoid $\Pbar'^+$ of $\Pbar'$ defined by $P_0'$ is equal to $\Pbar'^+=\Pbar^+\cap G'$. Now if $V$ is an arbitrary object of $\Rep_R^\infty(\Pbar)$ on which $Z$ acts by a central character then an $R$-submodule $M$ of $V$ is $\Pbar^+$-stable (resp.\ generates $V$ over $\Pbar$) if and only if $M$ is $\Pbar'^+$-stable (resp.\ generates $V$ over $\Pbar'$) because $\Pbar=Z\Pbar'$ and $\Pbar^+=Z\Pbar'^+$. As above this implies that the scalar restriction of $D(V)$ along $R\llbracket (\Pbar'^+)^{-1}\rrbracket\to R\llbracket (\Pbar^+)^{-1}\rrbracket$ can also be computed in the category $\Rep_R^\infty(\Pbar')$ using the monoid $\Pbar'^+$. Since $Z$ also acts by a central character on any member of the universal resolution of $V\in\Rep_R^\infty(\Pbar)$ in \cite{SV}, \S4, the above coeffaceability arguments show that we have an analogous statement for the modules $D^j(V)$. In other words, in order to show that $D(\ind_{ZP_0}^\Pbar(U_0))\to D^0(\ind_{ZP_0}^\Pbar(U_0))$ is bijective we may view $\ind_{ZP_0}^\Pbar(U_0)$ as a representation of $\Pbar'$ via restriction. However, there is a $\Pbar'$-equivariant isomorphism $\ind_{ZP_0}^\Pbar(U_0)\cong\ind_{P_0'}^{\Pbar'}(U_0)$ and the assertion is proved in \cite{SV}, Lemma 4.3. Note that our assumption on central characters is stronger than the hypothesis of \cite{SV}, Lemma 11.8 (ii), whence our proof is easier.\\

As for (ii), there is an isomorphism
\[
\varinjlim_{\Cscr}\C^\ori_c(\BT_{(\bullet)}^+(\Cscr),\F)^*\cong D(\C_c^\ori(\BT_{(\bullet)},\F)=D^0(\C_c^\ori(\BT_{(\bullet)},\F)
\]
of complexes of $R\llbracket(\Pbar^+)^{-1}\rrbracket$-modules by (i) and Proposition \ref{SV_functor}. Note that both sides are \'etale $R\llbracket\Pbar^+\rrbracket$-modules. The $R\llbracket(\Pbar^+)^{-1}\rrbracket$-linearity implies that the canonical left inverses of the \'etale module structures coincide (cf.\ \cite{SV}, Remark 6.1, as well as the proof of Proposition \ref{etale}). But then the linearization isomorphisms in \cite{SVZ}, Definition 3.1, have identical inverses. This implies that the above isomorphism is $R\llbracket\Pbar^+\rrbracket$-linear. Therefore, we can apply the cohomological formalism developed in \cite{SV}, \S4. We take the functorial resolution $\mathcal{I}_\bullet(\cdot)$ of each member of the oriented chain complex $\C_c^\ori(\BT_{(\bullet)},\F)$, apply the functor $D(\cdot)$ and obtain the double complex $D(\mathcal{I}_\bullet(\C_c^\ori(\BT_{(\bullet)},\F)))$. Consider the two standard spectral sequences converging to the cohomology of the associated total complex.\\

Fixing $j\geq 0$ the $E_1$-terms of one of them are given by the cohomology groups of the complex $D(\mathcal{I}_\bullet(\C_c^\ori(\BT_{(i)},\F)))$. Since the functor $\mathcal{I}_0=\ind_{\Pbar_1}^\Pbar$ is exact, the snake lemma and induction on $j$ show that the functors $\ker{\rho_j}$ considered in \cite{SV}, \S4, are exact for any $j\geq -1$. By \cite{SV}, Lemma 4.1, the functors $D\circ\mathcal{I}_j$ are exact for any $j\geq 0$. Therefore, the $E_1$-terms of the first spectral sequence are $D(\mathcal{I}_j(\H_i(\BT,\F)))$. Passing to the cohomology in the $j$-direction gives the required $E_2$-terms $D^j(\H_i(\BT,\F))$.\\

The cohomology groups of the total complex are now computed using the other spectral sequence. Its $E_1$-terms are $D^j(\C_c^\ori(\BT_{(i)},\F))$. According to (i) the second spectral sequence degenerates. Its abutments are the cohomology groups of the complex
\[
 D^0(\C_c^\ori(\BT_{(\bullet)},\F))=D(\C_c^\ori(\BT_{(\bullet)},\F))\cong \varinjlim_{\Cscr}\C^\ori_c(\BT_{(\bullet)}^+(\Cscr),\F)^*.
\]
The final assertion of the proposition follows from Proposition \ref{non_vanishing}.
Note that if $F$ is a face of $\BT$ with $F\subseteq\Cbar$ then $\F(M)_F$ is a quotient of $X_F\otimes_{H_F}M$ by Proposition \ref{I_F_invariants} (ii), hence is a finitely generated $R$-module. Moreover, since $\F(M)_F$ is generated by $\F(M)_F^I\cong M$ over $P_F^\dagger$ (cf.\ Lemma \ref{properties_H} (i) and Theorem \ref{quasi_inverse_coefficient_system}) the group $Z$ acts on $\F(M)_F$ by a character (cf.\ Remark \ref{group_ring}). Both conditions hold for any $F$ because of Lemma \ref{face_representatives}.
\end{proof}
\begin{rem}\label{generalized_phi_Gamma}
Assume $K=\Qp$ and $R=o/\pi^no$ for the valuation ring $o$ of some finite field extension of $\Qp$ and some positive integer $n$. One can then apply the complete localization and specialization techniques of \cite{SV} to the complex (\ref{etale_complex}) of \'etale $R\llbracket\Pbar^+\rrbracket$-modules. One obtains a complex of not necessarily finitely generated $(\varphi,\Gamma)$-modules in the classical sense of Fontaine. Assume more specifically that $G=\GL_2(\Qp)$, $n=1$ and that $V\in\Rep^I_R(G)$ is admissible. We then have $V\cong X\otimes_HV^I\cong\H_0(\BT,\F(V^I))$ for a suitable choice of $o$ as follows from Theorem \ref{exceptional_flat} (iii) and \cite{Oll1}, Th\'eor\`eme 1.2 (a). If $V$ admits a central character then the complex (\ref{etale_complex}) with $\F=\F(V^I)$ computes the \'etale $R\llbracket\Pbar^+\rrbracket$-modules $D^j(V)$ for $j\geq 0$ (cf.\ Proposition \ref{etale_spectral_sequence}). In fact, $D^j(V)=0$ for $j>0$. Moreover, the results of \cite{SV}, \S11, show that $D^0(V)$ gives rise to the $(\varphi,\Gamma)$-module associated to $V$ in Colmez's $p$-adic local Langlands correspondence for $\GL_2(\Qp)$ (cf.\ \cite{Col}).
\end{rem}


\vspace{1cm}
Universit\"at Duisburg-Essen\\
Fakult\"at f\"ur Mathematik\\
Thea-Leymann-Stra{\ss}e 9\\
D--45127 Essen, Germany\\
{\it E-mail address:} {\ttfamily jan.kohlhaase@uni-due.de}


\end{document}